\documentclass[12pt]{amsart}
\usepackage[margin = 1in]{geometry}
\usepackage[english]{babel}
\usepackage{amsmath}         
\usepackage{mathtools}       
\usepackage{amssymb}   
\usepackage{amsthm}
\usepackage{mathrsfs}   
\usepackage{colonequals}
\usepackage{enumitem}
\usepackage[dvipsnames]{xcolor}
\usepackage{tikz}
\usepackage{tikz-cd}
\usepackage{pgfplots}
\usepackage{float}
\usepackage{subfig}
\usepackage{graphicx}
\usepackage{comment}

\usepackage[figuresleft]{rotating}
\usepackage{adjustbox}
\usepackage{array}
\usepackage{bbm}

\usepackage{hyperref}
\hypersetup{
	colorlinks=true,
	linkcolor=blue,
	citecolor=blue,
	urlcolor=blue}

\newcommand{\cc}{\mathbb{C}}

\newcommand{\one}{\mathbbm{1}}

\newcommand{\D}{\mathcal D}

\newcommand{\K}{\mathcal K}
\renewcommand{\L}{\mathcal L}
\newcommand{\M}{\mathcal M}
\newcommand{\N}{\mathcal N}

\renewcommand{\P}{\mathcal P}

\let\sectionsign\S
\renewcommand{\S}{\mathcal S}

\newcommand{\qq}{\mathbb{Q}}

\newcommand{\Ind}{\operatorname{Ind}}

\newcommand{\Mbar}{\overline{\mathcal M}}
\newcommand{\Mb}{\overline{\mathcal M}}
\newcommand{\Mbargn}{\overline{\mathcal M}_{g,n}}

\newtheorem{theorem}{Theorem}[section]
\newtheorem{proposition}[theorem]{Proposition}
\newtheorem{lemma}[theorem]{Lemma}
\newtheorem{corollary}[theorem]{Corollary}

\theoremstyle{definition}
\newtheorem{definition}[theorem]{Definition}
\newtheorem{example}[theorem]{Example}

\theoremstyle{remark}

\newtheorem{remark}[theorem]{Remark}

\tikzset{
  graph vertex/.style={
    circle,
    draw,
    line width=.65pt,
    minimum size=5.2mm,
    inner sep=0pt,
    font=\scriptsize
  },
  graph edge/.style={
    line width=.65pt,
    line cap=round
  },
  graph marked points/.style={
    line width=1.5pt,
    line cap=round
  },
  graph cut/.style={
    densely dashed,
    line width=.55pt
  },
  graph label/.style={
    font=\scriptsize,
    inner sep=1pt
  }
}

\newcommand{\GraphBox}[2][1]{%
  \ensuremath{\mathord{\vcenter{\hbox{%
    \begin{tikzpicture}[
      x=1cm,
      y=1cm,
      scale=#1,
      transform shape
    ]
      #2
    \end{tikzpicture}%
  }}}}%
}

\newcommand{\GraphLambdaoneB}[1][1]{%
  \GraphBox[#1]{%
    \node[graph vertex] (c1) at (0,0) {$1$};
    \node[graph vertex] (c2) at (1.55,0) {$1$};
    \node[graph vertex] (c3) at (3.10,0) {$1$};

    \draw[graph marked points] (-0.65,0) -- (c1);
    \draw[graph edge] (c1) -- (c2);
    \draw[graph edge] (c2) -- (c3);

    \node[graph label,anchor=east] at (-0.72,0)
      {$[11]$};

    \node[graph label,above] at (0.38,0.02)
      {$p$};

    \node[graph label] at (0,-0.52)
      {$\omega_{B\cup\{p\}}$};
  }%
}

\newcommand{\GraphLambdatwo}[1][1]{%
  \GraphBox[#1]{%
    \node[graph vertex] (c1) at (0,0) {$1$};
    \node[graph vertex] (c2) at (1.55,0) {$1$};
    \node[graph vertex] (c3) at (3.10,0) {$1$};

    \draw[graph marked points] (-0.65,0) -- (c1);
    \draw[graph marked points] (1.55,0.55) -- (c2);
    \draw[graph edge] (c1) -- (c2);
    \draw[graph edge] (c2) -- (c3);

    \node[graph label,anchor=east] at (-0.72,0)
      {$[11]\backslash \{i\}$};

      \node[graph label,anchor=north] at (1.55,0.9)
      {$i$};

    \node[graph label,above] at (0.38,0.02)
      {$p$};

    \node[graph label] at (0,-0.52)
      {$\omega_{{B }\cup\{p\}}$};
  }%
}

\newcommand{\GraphLambdathree}[1][1]{%
  \GraphBox[#1]{%
    \node[graph vertex] (c1) at (0,0) {$1$};
    \node[graph vertex] (c2) at (1.55,0) {$1$};
    \node[graph vertex] (c3) at (3.10,0) {$1$};

    \draw[graph marked points] (-0.65,0) -- (c1);
    \draw[graph marked points] (3.75,0) -- (c3);
    \draw[graph edge] (c1) -- (c2);
    \draw[graph edge] (c2) -- (c3);

    \node[graph label,anchor=east] at (-0.72,0)
      {$[11]\backslash \{i\}$};

      \node[graph label,anchor=west] at (3.82,0)
      {$i$};

    \node[graph label,above] at (0.38,0.02)
      {$p$};

    \node[graph label] at (0,-0.52)
      {$\omega_{{B }\cup\{p\}}$};
  }%
}

\newcommand{\GraphLambdafour}[1][1]{%
 \GraphBox[#1]{%
    \node[graph vertex] (c1) at (0,0) {$1$};
    \node[graph vertex] (c2) at (1.55,0) {$0$};
    \node[graph vertex] (c3) at (3.10,0) {$2$};

    \draw[graph marked points] (-0.65,0) -- (c1);
    \draw[graph marked points] (1.55,0.55) -- (c2);
    \draw[graph edge] (c1) -- (c2);
    \draw[graph edge] (c2) -- (c3);

    \node[graph label,anchor=east] at (-0.72,0)
      {$[11]\backslash \{i\}$};

      \node[graph label,anchor=north] at (1.55,0.9)
      {$i$};

    \node[graph label,above] at (0.38,0.02)
      {$p$};

    \node[graph label] at (0,-0.52)
      {$\omega_{{B }\cup\{p\}}$};
  }%
}

\newcommand{\GraphLambdafive}[1][1]{%
  \GraphBox[#1]{%
    \node[graph vertex] (a1) at (0,0) {$1$};
    \node[graph vertex] (a2) at (1.55,0) {$1$};
    \node[graph vertex] (a3) at (3.10,0) {$1$};

    \draw[graph marked points] (-0.65,0) -- (a1);
    \draw[graph edge] (a1) -- (a2);
    \draw[graph edge] (a2) -- (a3);
    \draw[graph marked points] (a2) -- (1.55,0.72);

    \node[graph label,anchor=east] at (-0.72,0)
      {$i$};

    \node[graph label,above] at (1.55,0.72)
      {$[11]\setminus\{i\}$};

    \node[graph label,above] at (1.20,0.02)
      {$p$};

    \node[graph label,above] at (1.91,0.02)
      {$q$};

    \node[graph label] at (1.55,-0.58)
      {$\omega_{B'\cup p \cup q}$};

  }%
}

\newcommand{\GraphLambdasix}[1][1]{%
  \GraphBox[#1]{%
    \node[graph vertex] (a1) at (0,0) {$1$};
    \node[graph vertex] (a2) at (1.55,0) {$1$};
    \node[graph vertex] (a3) at (3.10,0) {$1$};

    \draw[graph marked points] (-0.65,0) -- (a1);
    \draw[graph edge] (a1) -- (a2);
    \draw[graph edge] (a2) -- (a3);
    \draw[graph marked points] (a2) -- (1.55,0.72);

    \node[graph label,anchor=east] at (-0.72,0)
      {$\{i,j\}$};

    \node[graph label,above] at (1.55,0.72)
      {$[11]\setminus\{i,j\}$};

    \node[graph label,above] at (1.20,0.02)
      {$p$};

    \node[graph label,above] at (1.91,0.02)
      {$q$};

    \node[graph label] at (1.55,-0.58)
      {$\omega_{B'\cup p \cup q}$};

  }%
}

\newcommand{\GraphLambdaseven}[1][1]{%
  \GraphBox[#1]{%
    \node[graph vertex] (a1) at (0,0) {$1$};
    \node[graph vertex] (a2) at (1.55,0) {$1$};
    \node[graph vertex] (a3) at (3.10,0) {$1$};

    \draw[graph marked points] (-0.65,0) -- (a1);
     \draw[graph marked points] (3.75,0) -- (a3);
    \draw[graph edge] (a1) -- (a2);
    \draw[graph edge] (a2) -- (a3);
    \draw[graph marked points] (a2) -- (1.55,0.72);

    \node[graph label,anchor=east] at (-0.72,0)
      {$i$};
       \node[graph label,anchor=west] at (3.82,0)
      {$j$};

    \node[graph label,above] at (1.55,0.72)
      {$[11]\setminus\{i,j\}$};

    \node[graph label,above] at (1.20,0.02)
      {$p$};

    \node[graph label,above] at (1.91,0.02)
      {$q$};

    \node[graph label] at (1.55,-0.58)
      {$\omega_{B'\cup p \cup q}$};

  }%
}

\newcommand{\GraphLambdanine}[1][1]{%
  \GraphBox[#1]{%
    \node[graph vertex] (a1) at (0,0) {$1$};
    \node[graph vertex] (a2) at (1.55,0) {$1$};
    \node[graph vertex] (a3) at (3.10,0) {$1$};

    \draw[graph edge] (a1) -- (a2);
    \draw[graph edge] (a2) -- (a3);
    \draw[graph marked points] (a2) -- (1.55,0.72);

    \node[graph label,above] at (1.55,0.72)
      {$[11]$};

    \node[graph label,above] at (1.20,0.02)
      {$p$};

    \node[graph label,above] at (1.91,0.02)
      {$q$};

    \node[graph label] at (1.55,-0.58)
      {$\omega_{B}^{S_2}$};

  }%
}

\newcommand{\GraphLambdaten}[1][1]{%
  \GraphBox[#1]{%
    \node[graph vertex] (a1) at (0,0) {$1$};
    \node[graph vertex] (a2) at (1.55,0) {$1$};
    \node[graph vertex] (a3) at (3.10,0) {$1$};

    \draw[graph edge] (a1) -- (a2);
    \draw[graph edge] (a2) -- (a3);
    \draw[graph marked points] (a2) -- (1.55,0.72);
    \draw[graph marked points] (-0.65,0) -- (a1);

    \node[graph label,above] at (1.55,0.72)
      {$[11] \backslash \{i\}$};

        \node[graph label,anchor=east] at (-0.72,0)
      {$i$};

    \node[graph label,above] at (1.20,0.02)
      {$p$};

    \node[graph label,above] at (1.91,0.02)
      {$q$};

    \node[graph label] at (1.55,-0.58)
      {$\omega_{B}^{S_2}$};

  }%
}

\newcommand{\GraphLambdaeleven}[1][1]{%
  \GraphBox[#1]{%
    \node[graph vertex] (e1) at (0,0) {$1$};
    \node[graph vertex] (e2) at (2.15,0) {$1$};

    \draw[graph edge] (e1) -- (e2);

    \draw[graph edge] (e2)
      to[out=45,in=-45,looseness=4.5] (e2);

    \draw[graph marked points] (e1) -- (-0.72,0);

    \node[graph label,above] at (0.40,0.02)
      {$p$};

    \node[graph label,anchor=east] at (-0.79,0)
      {$[11]$};

    \node[graph label] at (0,-0.53)
      {$\omega_{B\cup\{p\}}$};
  }%
}

\newcommand{\GraphLambdatwelve}[1][1]{%
  \GraphBox[#1]{%
    \node[graph vertex] (e1) at (0,0) {$1$};
    \node[graph vertex] (e2) at (2.15,0) {$1$};

    \draw[graph edge] (e1) -- (e2);

    \draw[graph edge] (e2)
      to[out=45,in=-45,looseness=4.5] (e2);
       
    \draw[graph marked points] (e2) -- (2.15,0.72);

    \draw[graph marked points] (e1) -- (-0.72,0);

    \node[graph label,above] at (0.40,0.02)
      {$p$};

    \node[graph label,anchor=east] at (-0.79,0)
      {$[11]\backslash \{i\}$};

      \node[graph label,above] at (2.15,0.72)
      {$i$};

    \node[graph label] at (0,-0.53)
      {$\omega_{B\cup\{p\}}$};
  }%
}

\newcommand{\GraphDLambdathirteen}[1][1]{%
  \GraphBox[#1]{%
    \node[graph vertex] (f1) at (0,0) {$1$};
    \node[graph vertex] (f2) at (2.15,0) {$1$};

    \draw[graph marked points] (-0.72,0) -- (f1);

    \draw[graph edge] (f1)
      to[out=45,in=135,looseness=1.25] (f2);

    \draw[graph edge] (f1)
      to[out=-45,in=-135,looseness=1.25] (f2);

    \node[graph label,anchor=east] at (-0.79,0)
      {$[11]$};

    \node[graph label] at (0.5,0.2)
      {$p$};

    \node[graph label] at (0.5,-0.2)
      {$q$};

    \node[graph label] at (0.0,-0.55)
      {$\omega_{B}^{S_2}$};
  }%
}

\newcommand{\GraphDLambdaforteen}[1][1]{%
  \GraphBox[#1]{%
    \node[graph vertex] (f1) at (0,0) {$1$};
    \node[graph vertex] (f2) at (2.15,0) {$1$};

    \draw[graph marked points] (-0.72,0) -- (f1);

    \draw[graph edge] (f1)
      to[out=45,in=135,looseness=1.25] (f2);

    \draw[graph edge] (f1)
      to[out=-45,in=-135,looseness=1.25] (f2);

    \draw[graph marked points] (f2) -- (2.8,0);

    \node[graph label,anchor=east] at (-0.79,0)
      {$[11]\backslash \{i\}$};

    \node[graph label] at (0.5,0.2)
      {$p$};

        \node[graph label,anchor=west] at (2.8,0)
      {$i$};

    \node[graph label] at (0.5,-0.2)
      {$q$};

    \node[graph label] at (0.0,-0.55)
      {$\omega_{B}^{S_2}$};
  }%
}

\newcommand{\GraphLambdafifteen}[1][1]{%
  \GraphBox[#1]{%
    \node[graph vertex] (d1) at (0,0) {$1$};
    \node[graph vertex] (d2) at (1.90,0) {$2$};

    \draw[graph edge] (d1) -- (d2);
    \draw[graph marked points] (d1) -- (-0.72,0);

    \node[graph label,above] at (0.36,0.02)
      {$p$};

    \node[graph label,above] at (1.53,0.02)
      {$q$};

    \node[graph label,anchor=east] at (-.79,0)
      {$[11]$};

    \node[graph label] at (0,-0.53)
    {$\omega_{B\cup\{p\}}$};

    \node[graph label] at (1.90,-0.53)
      {$\psi_q$};
  }%
}

\newcommand{\GraphLambdasixteen}[1][1]{%
  \GraphBox[#1]{%
    \node[graph vertex] (d1) at (0,0) {$1$};
    \node[graph vertex] (d2) at (1.90,0) {$2$};

    \draw[graph edge] (d1) -- (d2);
    \draw[graph marked points] (d1) -- (-0.72,0);
    \draw[graph marked points] (d2) -- (2.55,0);

    \node[graph label,above] at (0.36,0.02)
      {$p$};

    \node[graph label,above] at (1.53,0.02)
      {$q$};

    \node[graph label,anchor=east] at (-.79,0)
      {$[11] \setminus \{i\}$};

       \node[graph label,anchor=west] at (2.55,0)
      {$i$};

    \node[graph label] at (0,-0.53)
    {$\omega_{B\cup\{p\}}$};

    \node[graph label] at (1.90,-0.53)
      {$\psi_q$};
  }%
}

\newcommand{\GraphLambdaseventeen}[1][1]{%
  \GraphBox[#1]{%
    \node[graph vertex] (d1) at (0,0) {$1$};
    \node[graph vertex] (d2) at (1.90,0) {$2$};

    \draw[graph edge] (d1) -- (d2);
    \draw[graph marked points] (d1) -- (-0.72,0);
    \draw[graph marked points] (d2) -- (2.55,0);

    \node[graph label,above] at (0.36,0.02)
      {$p$};

    \node[graph label,above] at (1.53,0.02)
      {$q$};

    \node[graph label,anchor=east] at (-.79,0)
      {$[11] \setminus \{i\}$};

       \node[graph label,anchor=west] at (2.55,0)
      {$i$};

    \node[graph label] at (0,-0.53)
    {$\omega_{B\cup\{p\}}$};

    \node[graph label] at (1.90,-0.53)
      {$\psi_i$};
  }%
}

\newcommand{\GraphDLambdasixteenpush}[1][1]{%
  \GraphBox[#1]{%
    \node[graph vertex] (f1) at (0,0) {$1$};
    \node[graph vertex] (f2) at (2.15,0) {$2$};

    \draw[graph marked points] (0,.72) -- (f1);

    \draw[graph edge] (f1)
      to[out=45,in=135,looseness=1.25] (f2);

    \draw[graph edge] (f1)
      to[out=-45,in=-135,looseness=1.25] (f2);

      \draw[graph edge] (f1)
      to[out=135,in=215,looseness=5] (f1);

    \node[graph label,anchor=south] at (0,.79)
      {$[11]\backslash \{x,y,s,t\}$};

    \node[graph label] at (0.5,0.21)
      {$p$};

    \node[graph label] at (1.65,0.21)
      {$q$};

      \node[graph label] at (0.5,-0.25)
      {$x$};

    \node[graph label] at (1.65,-0.25)
      {$y$};

       \node[graph label] at (-.2,.4)
      {$s$};

    \node[graph label] at (-.2,-0.35)
      {$t$};

    \node[graph label] at (0,-0.7)
      {$\omega_{B \cup \{p\} } $};

     \node[graph label] at (2.15,-0.55)
      {$\psi_q$ };
      
  }%
}

\newcommand{\GraphMoneBxpush}[1][1]{%
  \GraphBox[#1]{%
    \node[graph vertex] (f1) at (0,0) {$1$};
    \node[graph vertex] (f2) at (2.15,0) {$1$};
    \node[graph vertex] (f3) at (4.30,0) {$1$};

    \draw[graph marked points] (2.80,-.3) -- (f2);

    \draw[graph edge] (f1)
      to[out=35,in=145,looseness=1.25] (f2);

    \draw[graph edge] (f1)
      to[out=-35,in=-145,looseness=1.25] (f2);

      \draw[graph edge] (f2)
      to[out=45,in=135,looseness=6] (f2);

      \draw[graph edge] (f2)--(f3);

    \node[graph label,right] at (2.82,-.4)
      {$[11]\backslash \{x,y,s,t\}$};

    \node[graph label] at (2.6,0.15)
      {$q$};

    \node[graph label] at (1.65,0.17)
      {$p$};

      \node[graph label] at (0.5,-0.17)
      {$x$};

    \node[graph label] at (1.65,-0.17)
      {$y$};

       \node[graph label] at (1.7,.55)
      {$s$};

    \node[graph label] at (2.6,.55)
      {$t$};

    \node[graph label] at (2.15,-0.8)
      {$\omega_{B' \cup \{p\} \cup \{q\} } $};

  }%
}

\newcommand{\GraphMthreeten}[1][1]{%
  \GraphBox[#1]{%
    \node[graph vertex] (a1) at (0,0) {$1$};
    \node[graph vertex] (a2) at (1.55,0) {$1$};
    \node[graph vertex] (a3) at (3.10,0) {$1$};

    \draw[graph marked points] (-0.65,0) -- (a1);
    \draw[graph edge] (a1) -- (a2);
    \draw[graph edge] (a2) -- (a3);
    \draw[graph marked points] (a2) -- (1.55,0.72);

    \node[graph label,anchor=east] at (-0.72,0)
      {$i$};

    \node[graph label,above] at (1.55,0.72)
      {$[10]\setminus\{i\}$};

    \node[graph label,above] at (1.20,0.02)
      {$p$};

    \node[graph label,above] at (1.91,0.02)
      {$q$};

    \node[graph label] at (1.55,-0.58)
      {$\omega_{[10]\setminus i \cup p \cup q}$};

  }%
}

\newcommand{\GraphMfournine}[1][1]{%
  \GraphBox[#1]{%
    \node[graph vertex] (a1) at (0,0) {$1$};
    \node[graph vertex] (a2) at (1.55,0) {$1$};
    \node[graph vertex] (a3) at (3.10,0) {$2$};

    \draw[graph edge] (a1) -- (a2);
    \draw[graph edge] (a2) -- (a3);
    \draw[graph marked points] (a2) -- (1.55,0.72);

    \node[graph label,above] at (1.55,0.72)
      {$[9]$};

    \node[graph label,above] at (1.20,0.02)
      {$p$};

    \node[graph label,above] at (1.91,0.02)
      {$q$};

    \node[graph label] at (1.55,-0.58)
      {$\omega_{[9] \cup p \cup q}$};

  }%
}

\newcommand{\GraphNthreeten}[1][1]{%
  \GraphBox[#1]{%
    \node[graph vertex] (b1) at (0,0) {$1$};
    \node[graph vertex] (b2) at (1.55,0) {$1$};
    \node[graph vertex] (b3) at (3.10,0) {$1$};

    \draw[graph edge] (b1) -- (b2);
    \draw[graph edge] (b2) -- (b3);
    \draw[graph marked points] (b2) -- (1.55,0.72);

    \node[graph label,above] at (1.55,0.72)
      {$[10]$};

    \node[graph label,above] at (1.20,0.02)
      {$p$};

    \node[graph label,above] at (1.91,0.02)
      {$q$};

    \node[graph label] at (1.55,-0.58)
      {$\omega_{[10]}^{S_2}$};
  }%
}

\newcommand{\GraphSthreeten}[1][1]{%
  \GraphBox[#1]{%
    \node[graph vertex] (c1) at (0,0) {$1$};
    \node[graph vertex] (c2) at (1.55,0) {$1$};
    \node[graph vertex] (c3) at (3.10,0) {$1$};

    \draw[graph marked points] (-0.65,0) -- (c1);
    \draw[graph edge] (c1) -- (c2);
    \draw[graph edge] (c2) -- (c3);

    \node[graph label,anchor=east] at (-0.72,0)
      {$[10]$};

    \node[graph label,above] at (0.38,0.02)
      {$p$};

    \node[graph label] at (0.48,-0.52)
      {$\omega_{[10]\cup\{p\}}$};
  }%
}

\newcommand{\GraphSthreetenpushforward}[1][1]{%
  \GraphBox[#1]{%
    \node[graph vertex] (c1) at (0,0) {$1$};
    \node[graph vertex] (c2) at (1.55,0) {$1$};
    \node[graph vertex] (c3) at (3.10,0) {$1$};

    \draw[graph marked points] (-0.65,0) -- (c1);
    \draw[graph edge] (c1) -- (c2);
    \draw[graph edge] (c2) -- (c3);
    \draw[graph edge] (c1)to[out=135,in=45,looseness=4.5] (c1);

    \node[graph label,anchor=east] at (-0.72,0)
      {$[8]$};

    \node[graph label,above] at (0.45,0.02)
      {$p$};

    \node[graph label] at (0.0,-0.52)
      {$\omega_{[10]\cup\{p\}}$};
  }%
}

\newcommand{\GraphPthreeten}[1][1]{%
  \GraphBox[#1]{%
    \node[graph vertex] (d1) at (0,0) {$1$};
    \node[graph vertex] (d2) at (1.90,0) {$2$};

    \draw[graph edge] (d1) -- (d2);
    \draw[graph marked points] (d1) -- (-0.72,0);

    \node[graph label,above] at (0.36,0.02)
      {$p$};

    \node[graph label,above] at (1.53,0.02)
      {$q$};

    \node[graph label,anchor=east] at (-.79,0)
      {$[10]$};

    \node[graph label] at (0,-0.53)
    {$\omega_{[10]\cup\{p\}}$};

    \node[graph label] at (1.90,-0.53)
      {$\psi_q$};
  }%
}

\newcommand{\GraphLthreeten}[1][1]{%
  \GraphBox[#1]{%
    \node[graph vertex] (e1) at (0,0) {$1$};
    \node[graph vertex] (e2) at (2.15,0) {$1$};

    \draw[graph edge] (e1) -- (e2);

    \draw[graph edge] (e2)
      to[out=45,in=-45,looseness=4.5] (e2);

    \draw[graph marked points] (e1) -- (-0.72,0);

    \node[graph label,above] at (0.40,0.02)
      {$p$};

    \node[graph label,anchor=east] at (-0.79,0)
      {$[10]$};

    \node[graph label] at (0,-0.53)
      {$\omega_{[10]\cup\{p\}}$};
  }%
}

\newcommand{\GraphDthreeten}[1][1]{%
  \GraphBox[#1]{%
    \node[graph vertex] (f1) at (0,0) {$1$};
    \node[graph vertex] (f2) at (2.15,0) {$1$};

    \draw[graph marked points] (-0.72,0) -- (f1);

    \draw[graph edge] (f1)
      to[out=45,in=135,looseness=1.25] (f2);

    \draw[graph edge] (f1)
      to[out=-45,in=-135,looseness=1.25] (f2);

    \node[graph label,anchor=east] at (-0.79,0)
      {$[10]$};

    \node[graph label] at (0.5,0.2)
      {$p$};

    \node[graph label] at (0.5,-0.2)
      {$q$};

    \node[graph label] at (0.0,-0.55)
      {$\omega_{[10]}^{S_2}$};
  }%
}

\usepackage{amsmath,amssymb}
\usepackage{tikz}
\usetikzlibrary{shapes.geometric}
\usepackage{array,booktabs,tabularx}

\tikzset{
  graph vertex/.style={
    circle, draw, line width=.65pt,
    minimum size=5.2mm, inner sep=0pt,
    font=\scriptsize
  },
  graph genus vertex/.style={
    graph vertex,
    minimum size=6mm
  },
  graph edge/.style={line width=.65pt,line cap=round},
  graph marked points/.style={line width=1.5pt,line cap=round},
  graph label/.style={font=\scriptsize,inner sep=1pt}
}

\providecommand{\GraphBox}[2][1]{%
  \ensuremath{\mathord{\vcenter{\hbox{%
    \begin{tikzpicture}[x=1cm,y=1cm,scale=#1,transform shape]
      #2
    \end{tikzpicture}%
  }}}}%
}

\newcommand{\GraphGenusLabel}[1]{\scalebox{.68}{$#1$}}

\newcommand{\GraphChartSide}[1][1]{%
  \GraphBox[#1]{%
    \node[graph vertex]      (v1) at (0,0) {$1$};
    \node[graph vertex]      (v2) at (1.45,0) {$a$};
    \node[graph vertex] (v3) at (3.10,0) {$b$};
    \draw[graph edge] (v1)--(v2)--(v3);
    \draw[graph marked points] (v1)--(-.62,.3);
    \draw[graph marked points] (v2)--(1.45,.64);
    \draw[graph marked points] (v3)--(3.72,.30);
    \node[graph label,above] at (-.62,.30) {$(X\cup Y)^c$};
    \node[graph label,above] at (1.45,.64) {$X$};
    \node[graph label,right] at (3.72,.30) {$Y$};
    \node[graph label,above] at (.42,.02) {$p$};
    \node[graph label,below] at (0,-.28) {$\omega_{B\cup\{p\}}$};
  }%
}

\newcommand{\GraphChartMiddleAnti}[1][1]{%
  \GraphBox[#1]{%
    \node[graph vertex]      (v1) at (0,0) {$a$};
    \node[graph vertex]      (v2) at (1.55,0) {$1$};
    \node[graph vertex] (v3) at (3.20,0) {$b$};
    \draw[graph edge] (v1)--(v2)--(v3);
    \draw[graph marked points] (v1)--(-.62,.30);
    \draw[graph marked points] (v2)--(1.55,.64);
    \draw[graph marked points] (v3)--(3.82,.30);
    \node[graph label,left]  at (-.62,.30) {$X$};
    \node[graph label,above] at (1.55,.64) {$(X\cup Y)^c$};
    \node[graph label,right] at (3.82,.30) {$Y$};
    \node[graph label,above] at (1.05,.02) {$p$};
    \node[graph label,above] at (2.05,.02) {$q$};
    \node[graph label,below] at (1.55,-.28) {$\omega_{B'\cup\{p,q\}}$};
  }%
}

\newcommand{\GraphChartMiddleSym}[1][1]{%
  \GraphBox[#1]{%
    \node[graph vertex]      (v1) at (0,0) {$a$};
    \node[graph vertex]      (v2) at (1.55,0) {$1$};
    \node[graph vertex] (v3) at (3.20,0) {$b$};
    \draw[graph edge] (v1)--(v2)--(v3);
    \draw[graph marked points] (v1)--(-.62,.30);
    \draw[graph marked points] (v2)--(1.55,.64);
    \draw[graph marked points] (v3)--(3.82,.30);
    \node[graph label,left]  at (-.62,.30) {$X$};
    \node[graph label,above] at (1.55,.64) {$(X\cup Y)^c$};
    \node[graph label,right] at (3.82,.30) {$Y$};
    \node[graph label,above] at (1.05,.02) {$p$};
    \node[graph label,above] at (2.05,.02) {$q$};
    \node[graph label,below] at (1.55,-.28) {$\omega_B^{S_2}$};
  }%
}

\newcommand{\GraphChartLoop}[1][1]{%
  \GraphBox[#1]{%
    \node[graph vertex]      (v1) at (0,0) {$1$};
    \node[graph genus vertex] (v2) at (1.90,0)
      {\GraphGenusLabel{g-2}};
    \draw[graph edge] (v1)--(v2);
    \draw[graph edge] (v2) to[out=45,in=-45,looseness=4.2] (v2);
    \draw[graph marked points] (v1)--(0,.64);
    \draw[graph marked points] (v2)--(1.90,.64);
    \node[graph label,above] at (0,.68) {$X^c$};
    \node[graph label,above] at (1.9,.68) {$X$};
    \node[graph label,above] at (.50,.02) {$p$};
    \node[graph label,below] at (0,-.28) {$\omega_{B\cup\{p\}}$};
  }%
}

\newcommand{\GraphChartDoubleEdge}[1][1]{%
  \GraphBox[#1]{%
    \node[graph vertex]      (v1) at (0,0) {$1$};
    \node[graph genus vertex] (v2) at (2.05,0)
      {\GraphGenusLabel{g-2}};
    \draw[graph edge] (v1) to[out=45,in=135,looseness=1.15] (v2);
    \draw[graph edge] (v1) to[out=-45,in=-135,looseness=1.15] (v2);
    \draw[graph marked points] (v1)--(-.68,0);
    \draw[graph marked points] (v2)--(2.73,0);
    \node[graph label,left] at (-.68,0) {$X^c$};
    \node[graph label,right] at (2.73,.0) {$X$};
    \node[graph label,above] at (.62,.18) {$p$};
    \node[graph label,below] at (.62,-.18) {$q$};
    \node[graph label,below] at (0,-.28) {$\omega_B^{S_2}$};
  }%
}

\newcommand{\GraphChartPsi}[1][1]{%
  \GraphBox[#1]{%
    \node[graph vertex]      (v1) at (0,0) {$1$};
    \node[graph genus vertex] (v2) at (1.95,0)
      {\GraphGenusLabel{g-1}};
    \draw[graph edge] (v1)--(v2);
    \draw[graph marked points] (v1)--(-.68,0);
    \draw[graph marked points] (v2)--(2.73,0);
    \node[graph label,left] at (-.68,0) {$X^c$};
    \node[graph label,right] at (2.73,0) {$X$};
    \node[graph label,above] at (.48,.02) {$p$};
    \node[graph label,above] at (1.48,.02) {$q$};
    \node[graph label,below] at (0,-.28) {$\omega_{B\cup\{p\}}$};
    \node[graph label,below] at (1.95,-.28) {$\psi_i$};
    \node[graph label] at (1,-1) {$i\in X\cup\{q\}$};
  }%
}

\newcommand{\GraphChartKappa}[1][1]{%
  \GraphBox[#1]{%
    \node[graph vertex]      (v1) at (0,0) {$1$};
    \node[graph genus vertex] (v2) at (1.95,0)
      {\GraphGenusLabel{g-1}};
    \draw[graph edge] (v1)--(v2);
    \draw[graph marked points] (v1)--(-.68,0);
    \draw[graph marked points] (v2)--(2.73,0);
    \node[graph label,left] at (-.68,0) {$X^c$};
    \node[graph label,right] at (2.73,0) {$X$};
    \node[graph label,above] at (.48,.02) {$p$};
    \node[graph label,below] at (0,-.28) {$\omega_{B\cup\{p\}}$};
    \node[graph label,below] at (1.95,-.38) {$\kappa_1$};
  }%
}

\newcommand{\ChartName}[2]{%
  \begin{tabular}{@{}c@{}}
    \(\displaystyle #1\)\\[-.2ex]
    {\scriptsize\itshape\shortstack[c]{#2}}
  \end{tabular}%
}
\newcommand{\ChartRestriction}[2]{%
  \par\noindent
  \hangindent=2.7em\hangafter=1
  \makebox[2.7em][l]{\textup{(R#1)}}#2%
}
\newcommand{\ChartRestrictions}[1]{%
  \begin{minipage}[c]{\linewidth}\small #1\end{minipage}%
}

\newlength{\GeneratorChartRowHeight}
\newcommand{\ChartGeneratorCell}[1]{%
  \begin{minipage}[c][\GeneratorChartRowHeight][c]{\linewidth}
    \centering #1
  \end{minipage}%
}

\newcommand{\GeneratorChart}{%
\begin{figure}[ht]
\centering
\begingroup
\renewcommand{\arraystretch}{1.25}
\renewcommand{\tabularxcolumn}[1]{m{##1}}
\setlength{\tabcolsep}{5pt}
\begin{tabularx}{\textwidth}{@{}
  >{\centering\arraybackslash}m{.16\textwidth}|
  >{\centering\arraybackslash}m{.34\textwidth}|
  >{\raggedright\arraybackslash}X@{}}
\hline
\textbf{Name} & \textbf{Generator} & \textbf{Restrictions}\\
\hline
\ChartName{\mathcal S^a_{B,X,Y}}{side}
& \ChartGeneratorCell{\GraphChartSide[.83]} & \\
\hline
\ChartName{\mathcal M^a_{B',X,Y}}{middle\\antisymmetric}
& \ChartGeneratorCell{\GraphChartMiddleAnti[.83]}
& \ChartRestrictions{
    \ChartRestriction{1}{\(a\leq\frac{g-1}{2}\), and if equality holds, \(\min(X)<\min(Y)\).}
    \ChartRestriction{2}{\((a,b,X,Y)\neq(a,a,\varnothing,\varnothing)\).}
    \ChartRestriction{3}{\((a,|X|)\neq(0,2)\).}}
\\
\hline
\ChartName{\mathcal N^a_{B,X,Y}}{middle symmetric}
& \ChartGeneratorCell{\GraphChartMiddleSym[.83]}
& \ChartRestrictions{
    \ChartRestriction{4}{\(a\leq\frac{g-1}{2}\), and if equality holds, \(\min(X)<\min(Y)\).}
    \ChartRestriction{5}{If \((a,|X|)=(0,2)\), then\\ \(\min(X)<\min(B)\).}}
\\
\hline
\ChartName{\mathcal L_{B,X}}{loop}
& \ChartGeneratorCell{\GraphChartLoop[.83]}
& \ChartRestrictions{\ChartRestriction{6}{If \(g=2\), then \(|X|\geq1\).}}
\\
\hline
\ChartName{\mathcal D_{B,X}}{double edge}
& \ChartGeneratorCell{\GraphChartDoubleEdge[.83]}
& \ChartRestrictions{
    \ChartRestriction{7}{If \(g=2\), then \(|X|\geq2\).}
    \ChartRestriction{8}{If \(g=2\) and \(|X|=2\), then\\ \(\min(X)<\min(B)\).}}
\\
\hline
\ChartName{\mathcal P^i_{B,X}}{psi}
& \ChartGeneratorCell{\GraphChartPsi[.83]}
& \ChartRestrictions{\ChartRestriction{9}{\(g\geq3\).}}
\\
\hline
\ChartName{\mathcal K_{B,X}}{kappa}
& \ChartGeneratorCell{\GraphChartKappa[.83]}
& \ChartRestrictions{\ChartRestriction{10}{\(g\geq4\).}}
\\
\hline
\end{tabularx}
\endgroup
\caption{Generators for $S^{13,2}_\omega(\Mb_{g,n})$. Here, we have $X,Y \subset [n]$ of any size such that the resulting graph
is stable. Additionally, $B \subset (X \cup Y)^c$, or $B \subset X^c$ when
$Y$ is not defined, with $|B|=10$. In the second row,
$B' \subset (X \cup Y)^c$ with $|B'|=9$. In the first, second, and third rows, $b=g-a-1$. }
\label{fig:gentable}
\end{figure}
}

\usepackage{amsmath}
\usepackage{tikz}

\tikzset{
  relation vertex/.style={
    circle,
    draw,
    line width=.65pt,
    minimum size=5.2mm,
    inner sep=0pt,
    font=\scriptsize
  },
  relation edge/.style={
    line width=.65pt,
    line cap=round
  },
  relation marking/.style={
    line width=1.45pt,
    line cap=round
  },
  relation label/.style={
    font=\scriptsize,
    inner sep=1pt
  },
  relation decoration/.style={
    font=\scriptsize,
    inner sep=1pt
  }
}

\usepackage{amsmath}
\usepackage{tikz}

\tikzset{
  equation graph vertex/.style={
    circle,
    draw,
    line width=.65pt,
    minimum size=5.2mm,
    inner sep=0pt,
    font=\scriptsize
  },
  equation graph edge/.style={
    line width=.65pt,
    line cap=round
  },
  equation graph marking/.style={
    line width=1.45pt,
    line cap=round
  },
  equation graph label/.style={
    font=\scriptsize,
    inner sep=1pt
  },
  equation graph decoration/.style={
    font=\scriptsize,
    inner sep=1pt
  }
}

\newcommand{\EquationGraphBox}[2][1]{%
  \ensuremath{\mathord{\vcenter{\hbox{%
    \begin{tikzpicture}[
      x=1cm,
      y=1cm,
      scale=#1,
      transform shape
    ]
      #2
    \end{tikzpicture}%
  }}}}%
}

\newcommand{\GraphEquationChainS}[1][1]{%
  \EquationGraphBox[#1]{%
    \node[equation graph vertex] (a) at (0,0) {$1$};
    \node[equation graph vertex] (z) at (1.30,0) {$0$};
    \node[equation graph vertex] (b) at (2.60,0) {$1$};

    \draw[equation graph edge] (a)--(z)--(b);
    \draw[equation graph marking] (a)--(0,.62);
    \draw[equation graph marking] (z)--(1.30,.62);

    \node[equation graph label,above] at (0,.64) {$X$};
    \node[equation graph label,above] at (.40,.02) {$p$};
    \node[equation graph label,above] at (1.30,.64) {$s$};
    \node[equation graph decoration,below] at (0,-.29)
      {$\omega_{X\cup\{p\}}$};
  }%
}

\newcommand{\GraphEquationLoopS}[1][1]{%
  \EquationGraphBox[#1]{%
    \node[equation graph vertex] (a) at (0,0) {$1$};
    \node[equation graph vertex] (z) at (1.38,0) {$0$};

    \draw[equation graph edge] (a)--(z);
    \draw[equation graph edge]
      (z) to[out=125,in=55,looseness=5.0] (z);
    \draw[equation graph marking] (a)--(0,.62);
    \draw[equation graph marking] (z)--(2.08,0);

    \node[equation graph label,above] at (0,.64) {$X$};
    \node[equation graph label,above] at (.42,.02) {$p$};
    \node[equation graph label,right] at (2.08,0) {$s$};
    \node[equation graph decoration,below] at (0,-.29)
      {$\omega_{X\cup\{p\}}$};
  }%
}

\newcommand{\GraphEquationPsiS}[1][1]{%
  \EquationGraphBox[#1]{%
    \node[equation graph vertex] (a) at (0,0) {$1$};
    \node[equation graph vertex] (b) at (1.45,0) {$1$};

    \draw[equation graph edge] (a)--(b);
    \draw[equation graph marking] (a)--(0,.62);
    \draw[equation graph marking] (b)--(2.15,0);

    \node[equation graph label,above] at (0,.64) {$X$};
    \node[equation graph label,above] at (.43,.02) {$p$};
    \node[equation graph label,right] at (2.15,0) {$s$};
    \node[equation graph decoration,below] at (0,-.29)
      {$\omega_{X\cup\{p\}}$};
    \node[equation graph decoration,below] at (1.45,-.29)
      {$-\psi_s$};
  }%
}

\newcommand{\GraphEquationDoubleEdgeS}[1][1]{%
  \EquationGraphBox[#1]{%
    \node[equation graph vertex] (a) at (0,0) {$1$};
    \node[equation graph vertex] (z) at (1.60,0) {$0$};

    \draw[equation graph edge]
      (a) to[out=45,in=135,looseness=1.15] (z);
    \draw[equation graph edge]
      (a) to[out=-45,in=-135,looseness=1.15] (z);
    \draw[equation graph marking] (a)--(0,.62);
    \draw[equation graph marking] (z)--(2.30,0);

    \node[equation graph label,above] at (0,.64) {$X$};
    \node[equation graph label,above] at (.57,.15) {$p$};
    \node[equation graph label,below] at (.57,-.15) {$q$};
    \node[equation graph label,right] at (2.30,0) {$s$};
    \node[equation graph decoration,below] at (0,-.29)
      {$\omega^{S_2}$};
  }%
}

\usepackage{tikz}

\tikzset{
  auxiliary graph vertex/.style={
    circle,
    draw,
    line width=.65pt,
    minimum size=5.2mm,
    inner sep=0pt,
    font=\scriptsize
  },
  auxiliary graph edge/.style={
    line width=.65pt,
    line cap=round
  },
  auxiliary graph marking/.style={
    line width=1.45pt,
    line cap=round
  },
  auxiliary graph label/.style={
    font=\scriptsize,
    inner sep=1pt
  },
  auxiliary graph decoration/.style={
    font=\scriptsize,
    inner sep=1pt
  }
}

\newcommand{\AuxiliaryGraphBox}[2][1]{%
  \ensuremath{\mathord{\vcenter{\hbox{%
    \begin{tikzpicture}[
      x=1cm,
      y=1cm,
      scale=#1,
      transform shape
    ]
      #2
    \end{tikzpicture}%
  }}}}%
}

\newcommand{\GraphChainXComplementY}[1][1]{%
  \AuxiliaryGraphBox[#1]{%
    \node[auxiliary graph vertex] (z) at (0,0) {$0$};
    \node[auxiliary graph vertex] (a) at (1.40,0) {$1$};
    \node[auxiliary graph vertex] (b) at (2.90,0) {$1$};

    \draw[auxiliary graph edge] (z)--(a)--(b);
    \draw[auxiliary graph marking] (z)--(-.62,.48);
    \draw[auxiliary graph marking] (a)--(1.40,.70);
    \draw[auxiliary graph marking] (b)--(3.52,.48);

    \node[auxiliary graph label,above] at (-.75,.55) {$X$};
    \node[auxiliary graph label,above] at (1.40,.70)
      {$[n]\setminus(X\cup Y)$};
    \node[auxiliary graph label,above] at (3.7,.48) {$Y$};
    \node[auxiliary graph label,below] at (1,-.02) {$p$};
    \node[auxiliary graph label,below] at (1.8,-.02) {$q$};
    \node[auxiliary graph decoration,below] at (1.40,-.31)
      {$\omega_{B\cup\{p\}}$};
  }%
}

\newcommand{\GraphKunnethElevenFour}[1][1]{%
  \AuxiliaryGraphBox[#1]{%
    \node[auxiliary graph vertex] (a) at (0,0) {$1$};
    \node[auxiliary graph vertex] (b) at (1.75,0) {$1$};

    \draw[auxiliary graph edge] (a)--(b);
    \draw[auxiliary graph marking] (a)--(-.63,.48);
    \draw[auxiliary graph marking] (b)--(2.38,.48);

    \node[auxiliary graph label,above] at (-.8,.48)
      {$[n]\setminus X'$};
    \node[auxiliary graph label,above] at (2.55,.55) {$X'$};
    \node[auxiliary graph label,above] at (.4,.02) {$s$};
    \node[auxiliary graph label,above] at (1.4,.02) {$t$};
    \node[auxiliary graph label,above] at (0.875,.7) {$H^{11}\otimes H^4$};
  }%
}

\newcommand{\GraphKunnethThirteenTwo}[1][1]{%
  \AuxiliaryGraphBox[#1]{%
    \node[auxiliary graph vertex] (a) at (0,0) {$1$};
    \node[auxiliary graph vertex] (b) at (1.75,0) {$1$};

    \draw[auxiliary graph edge] (a)--(b);
    \draw[auxiliary graph marking] (a)--(-.63,.48);
    \draw[auxiliary graph marking] (b)--(2.38,.48);

    \node[auxiliary graph label,above] at (-.8,.48)
      {$[n]\setminus X'$};
    \node[auxiliary graph label,above] at (2.55,.55) {$X'$};
    \node[auxiliary graph label,above] at (.4,.02) {$s$};
    \node[auxiliary graph label,above] at (1.4,.02) {$t$};
    \node[auxiliary graph label,above] at (0.875,.7) {$H^{13}\otimes H^2$};
  }%
}

\usepackage{amsmath,amssymb}
\usepackage{tikz}

\tikzset{
  case graph vertex/.style={
    circle,
    draw,
    line width=.65pt,
    minimum size=5.2mm,
    inner sep=0pt,
    font=\scriptsize
  },
  case graph edge/.style={line width=.65pt,line cap=round},
  case graph marking/.style={line width=1.45pt,line cap=round},
  case graph label/.style={font=\scriptsize,inner sep=1pt},
  case graph decoration/.style={font=\scriptsize,inner sep=1pt}
}

\newcommand{\CaseGraphBox}[2][1]{%
  \ensuremath{\mathord{\vcenter{\hbox{%
    \begin{tikzpicture}[x=1cm,y=1cm,scale=#1,transform shape]
      #2
    \end{tikzpicture}%
  }}}}%
}

\newcommand{\GraphKunnethThirteenTwoGenusZero}[1][1]{%
  \AuxiliaryGraphBox[#1]{%
    \node[auxiliary graph vertex] (a) at (0,0) {$2$};
    \node[auxiliary graph vertex] (b) at (1.75,0) {$0$};

    \draw[auxiliary graph edge] (a)--(b);
    \draw[auxiliary graph marking] (a)--(-.63,.48);
    \draw[auxiliary graph marking] (b)--(2.38,.48);

    \node[auxiliary graph label,above] at (-.8,.48)
      {$[n]\setminus X'$};
    \node[auxiliary graph label,above] at (2.55,.55) {$X'$};
    \node[auxiliary graph label,above] at (.4,.02) {$s$};
    \node[auxiliary graph label,above] at (1.4,.02) {$t$};
    \node[auxiliary graph label,above] at (0.875,.7) {$H^{13}\otimes H^2$};
  }%
}

\newcommand{\BoundaryCaseOneSource}[1][1]{%
  \CaseGraphBox[#1]{%
    \node[case graph vertex] (a) at (0,0) {$1$};
    \node[case graph vertex] (z) at (1.30,0) {$0$};
    \node[case graph vertex] (b) at (2.60,0) {$1$};
    \draw[case graph edge] (a)--(z)--(b);
    \draw[case graph marking] (a)--(-.56,.46);
    \draw[case graph marking] (z)--(1.30,.63);
    \draw[case graph marking] (b)--(3.16,.46);
    \node[case graph label,above] at (-.75,.5)
      {$[n]\setminus X'$};
    \node[case graph label,above] at (1.30,.7) {$X$};
    \node[case graph label,above] at (3.3,.52) {$S$};
    \node[case graph label,above] at (.43,.02) {$p$};
    \node[case graph decoration,below] at (0,-.29)
      {$\omega_{B\cup\{p\}}$};
  }%
}

\newcommand{\BoundaryCaseOneTarget}[1][1]{%
  \CaseGraphBox[#1]{%
    \node[case graph vertex] (a) at (0,0) {$1$};
    \node[case graph vertex] (z) at (1.30,0) {$0$};
    \node[case graph vertex] (b) at (2.60,0) {$1$};
    \draw[case graph edge] (a)--(z)--(b);
    \draw[case graph marking] (a)--(-.56,.46);
    \draw[case graph marking] (z)--(1.30,.63);
    \draw[case graph marking] (b)--(3.16,.46);
    \node[case graph label,above] at (-.75,.5)
      {$[n]\setminus X'$};
    \node[case graph label,above] at (1.30,.7) {$X$};
    \node[case graph label,above] at (3.3,.52) {$S$};
    \node[case graph label,above] at (.43,.02) {$s$};
    \node[case graph label,above] at (.92,.02) {$t$};
    \node[case graph decoration,below] at (0,-.29)
      {$\omega_{B\cup\{s\}} \cdot (-\psi_s)$};
  }%
}

\newcommand{\BoundaryCaseTwoSource}[1][1]{%
  \CaseGraphBox[#1]{%
    \node[case graph vertex] (a) at (0,0) {$1$};
    \node[case graph vertex] (z) at (1.30,0) {$0$};
    \node[case graph vertex] (b) at (2.60,0) {$1$};
    \draw[case graph edge] (a)--(z)--(b);
    \draw[case graph marking] (a)--(-.56,.46);
    \draw[case graph marking] (z)--(1.30,.63);
    \draw[case graph marking] (b)--(3.16,.46);
    \node[case graph label,above] at (-.56,.5)
      {$A\subsetneq[n]\setminus X'$};
    \node[case graph label,above] at (1.30,.7) {$X$};
    \node[case graph label,above] at (3.3,.52) {$X'$};
    \node[case graph label,above] at (.43,.02) {$p$};
    \node[case graph decoration,below] at (0,-.29)
      {$\omega_{B\cup\{p\}}$};
  }%
}

\newcommand{\BoundaryCaseTwoTarget}[1][1]{%
  \CaseGraphBox[#1]{%
    \node[case graph vertex] (a) at (0,0) {$1$};
    \node[case graph vertex] (z) at (1.30,0) {$0$};
    \node[case graph vertex] (b) at (2.60,0) {$1$};
    \draw[case graph edge] (a)--(z)--(b);
    \draw[case graph marking] (a)--(-.56,.46);
    \draw[case graph marking] (z)--(1.30,.63);
    \draw[case graph marking] (b)--(3.16,.46);
    \node[case graph label,above] at (-.7,.52) {$A$};
    \node[case graph label,above] at (1.30,.7) {$X$};
    \node[case graph label,above] at (3.3,.52) {$X'$};
    \node[case graph label,above] at (.43,.02) {$p$};
    \node[case graph label,above] at (1.73,.02) {$s$};
    \node[case graph label,above] at (2.22,.02) {$t$};
    \node[case graph decoration,below] at (0,-.29)
      {$\omega_{B\cup\{p\}}$};
    \node[case graph decoration,below] at (2.60,-.29) {$-\psi_t$};
  }%
}

\newcommand{\BoundaryCaseThreeSource}[1][1]{%
  \CaseGraphBox[#1]{%
    \node[case graph vertex] (a) at (0,0) {$1$};
    \node[case graph vertex] (z) at (1.30,0) {$0$};
    \node[case graph vertex] (b) at (2.60,0) {$1$};
    \draw[case graph edge] (a)--(z)--(b);
    \draw[case graph marking] (a)--(-.56,.46);
    \draw[case graph marking] (z)--(1.30,.63);
    \draw[case graph marking] (b)--(3.16,.46);
    \node[case graph label,above] at (-.56,.5)
      {$A\subsetneq[n]\setminus X'$};
    \node[case graph label,above] at (1.30,.7) {$X$};
    \node[case graph label,above] at (3.2,.52)
      {$S\subsetneq X'$};
    \node[case graph label,above] at (.43,.02) {$p$};
    \node[case graph decoration,below] at (0,-.29)
      {$\omega_{B\cup\{p\}}$};
  }%
}

\newcommand{\BoundaryCaseThreeTarget}[1][1]{%
  \CaseGraphBox[#1]{%
    \node[case graph vertex] (a)  at (0,0) {$1$};
    \node[case graph vertex] (z1) at (1.10,0) {$0$};
    \node[case graph vertex] (z2) at (2.25,0) {$0$};
    \node[case graph vertex] (b)  at (3.35,0) {$1$};
    \draw[case graph edge] (a)--(z1)--(z2)--(b);
    \draw[case graph marking] (a)--(-.54,.46);
    \draw[case graph marking] (z1)--(1.10,.63);
    \draw[case graph marking] (z2)--(2.25,.63);
    \draw[case graph marking] (b)--(3.89,.46);
    \node[case graph label,above] at (-.7,.52) {$A$};
    \node[case graph label,above] at (1.10,.7) {$X\setminus X'$};
    \node[case graph label,above] at (2.25,.7) {$X'\setminus S$};
    \node[case graph label,above] at (3.95,.52) {$S$};
    \node[case graph label,above] at (.38,.02) {$p$};
    \node[case graph label,above] at (1.48,.02) {$s$};
    \node[case graph label,above] at (1.86,.02) {$t$};
    \node[case graph decoration,below] at (0,-.29)
      {$\omega_{B\cup\{p\}}$};
  }%
}

\newcommand{\BoundaryCaseFourSource}[1][1]{%
  \CaseGraphBox[#1]{%
    \node[case graph vertex] (x) at (0,0) {$0$};
    \node[case graph vertex] (a) at (1.30,0) {$1$};
    \node[case graph vertex] (b) at (2.60,0) {$1$};
    \draw[case graph edge] (x)--(a)--(b);
    \draw[case graph marking] (x)--(-.56,.46);
    \draw[case graph marking] (a)--(1.30,.63);
    \draw[case graph marking] (b)--(3.16,.46);
    \node[case graph label,above] at (-.7,.52) {$X$};
    \node[case graph label,above] at (1.30,.7) {$A$};
    \node[case graph label,above] at (3.3,.52) {$X'$};
    \node[case graph label,above] at (.9,.02) {$p$};
    \node[case graph label,above] at (1.73,.02) {$q$};
    \node[case graph decoration,below] at (1.30,-.29)
      {$\omega_{B\cup\{p\}}$};
  }%
}

\newcommand{\BoundaryCaseFourTarget}[1][1]{%
  \CaseGraphBox[#1]{%
    \node[case graph vertex] (x) at (0,0) {$0$};
    \node[case graph vertex] (a) at (1.30,0) {$1$};
    \node[case graph vertex] (b) at (2.60,0) {$1$};
    \draw[case graph edge] (x)--(a)--(b);
    \draw[case graph marking] (x)--(-.56,.46);
    \draw[case graph marking] (a)--(1.30,.63);
    \draw[case graph marking] (b)--(3.16,.46);
    \node[case graph label,above] at (-.7,.52) {$X$};
    \node[case graph label,above] at (1.30,.7) {$A$};
    \node[case graph label,above] at (3.3,.52) {$X'$};
    \node[case graph label,above] at (.9,.02) {$p$};
    \node[case graph label,above] at (1.73,.02) {$s$};
    \node[case graph label,above] at (2.17,.02) {$t$};
    \node[case graph decoration,below] at (1.30,-.29)
      {$\omega_{B\cup\{p\}}$};
    \node[case graph decoration,below] at (2.60,-.29) {$-\psi_t$};
  }%
}

\newcommand{\BoundaryCaseFiveSource}[1][1]{%
  \CaseGraphBox[#1]{%
    \node[case graph vertex] (x) at (0,0) {$0$};
    \node[case graph vertex] (a) at (1.30,0) {$1$};
    \node[case graph vertex] (b) at (2.60,0) {$1$};
    \draw[case graph edge] (x)--(a)--(b);
    \draw[case graph marking] (x)--(-.56,.46);
    \draw[case graph marking] (a)--(1.30,.63);
    \draw[case graph marking] (b)--(3.16,.46);
    \node[case graph label,above] at (-.56,.5)
      {$X\subseteq[n]\setminus X'$};
    \node[case graph label,above] at (1.30,.7) {$A$};
    \node[case graph label,above] at (3.2,.55)
      {$S\subsetneq X'$};
    \node[case graph label,above] at (.9,.02) {$p$};
    \node[case graph label,above] at (1.73,.02) {$q$};
    \node[case graph decoration,below] at (1.30,-.29)
      {$\omega_{B\cup\{p\}}$};
  }%
}

\newcommand{\BoundaryCaseFiveTarget}[1][1]{%
  \CaseGraphBox[#1]{%
    \node[case graph vertex] (x)  at (0,0) {$0$};
    \node[case graph vertex] (a)  at (1.10,0) {$1$};
    \node[case graph vertex] (z)  at (2.25,0) {$0$};
    \node[case graph vertex] (b)  at (3.35,0) {$1$};
    \draw[case graph edge] (x)--(a)--(z)--(b);
    \draw[case graph marking] (x)--(-.54,.46);
    \draw[case graph marking] (a)--(1.10,.63);
    \draw[case graph marking] (z)--(2.25,.63);
    \draw[case graph marking] (b)--(3.89,.46);
    \node[case graph label,above] at (-.7,.52) {$X$};
    \node[case graph label,above] at (1.10,.67) {$A\setminus X'$};
    \node[case graph label,above] at (2.25,.67) {$X'\setminus S$};
    \node[case graph label,above] at (3.99,.52) {$S$};
    \node[case graph label,above] at (.75,.02) {$p$};
    \node[case graph label,above] at (1.48,.02) {$s$};
    \node[case graph label,above] at (1.86,.02) {$t$};
    \node[case graph decoration,below] at (1.10,-.29)
      {$\omega_{B\cup\{p\}}$};
  }%
}

\newcommand{\BoundaryCaseSixSource}[1][1]{%
  \CaseGraphBox[#1]{%
    \node[case graph vertex] (a) at (0,0) {$1$};
    \node[case graph vertex] (b) at (1.30,0) {$1$};
    \node[case graph vertex] (x) at (2.60,0) {$0$};
    \draw[case graph edge] (a)--(b)--(x);
    \draw[case graph marking] (a)--(-.56,.46);
    \draw[case graph marking] (b)--(1.30,.63);
    \draw[case graph marking] (x)--(3.16,.46);
    \node[case graph label,above] at (-.6,.55)
      {$[n]\setminus X'$};
    \node[case graph label,above] at (1.30,.7) {$S$};
    \node[case graph label,above] at (3.3,.52) {$X$};
    \node[case graph label,above] at (.43,.02) {$p$};
    \node[case graph decoration,below] at (0,-.29)
      {$\omega_{B\cup\{p\}}$};
  }%
}

\newcommand{\BoundaryCaseSixTarget}[1][1]{%
  \CaseGraphBox[#1]{%
    \node[case graph vertex] (a) at (0,0) {$1$};
    \node[case graph vertex] (b) at (1.30,0) {$1$};
    \node[case graph vertex] (x) at (2.60,0) {$0$};
    \draw[case graph edge] (a)--(b)--(x);
    \draw[case graph marking] (a)--(-.56,.46);
    \draw[case graph marking] (b)--(1.30,.63);
    \draw[case graph marking] (x)--(3.16,.46);
    \node[case graph label,above] at (-.6,.55)
      {$[n]\setminus X'$};
    \node[case graph label,above] at (1.30,.7) {$S$};
    \node[case graph label,above] at (3.3,.52) {$X$};
    \node[case graph label,above] at (.43,.02) {$s$};
    \node[case graph label,above] at (.92,.02) {$t$};
    \node[case graph decoration,below] at (0,-.29)
      {$\omega_{B\cup\{s\}}\cdot(-\psi_s)$};
  }%
}

\newcommand{\BoundaryCaseSevenSource}[1][1]{%
  \CaseGraphBox[#1]{%
    \node[case graph vertex] (a) at (0,0) {$1$};
    \node[case graph vertex] (b) at (1.30,0) {$1$};
    \node[case graph vertex] (x) at (2.60,0) {$0$};
    \draw[case graph edge] (a)--(b)--(x);
    \draw[case graph marking] (a)--(-.56,.46);
    \draw[case graph marking] (b)--(1.30,.63);
    \draw[case graph marking] (x)--(3.16,.46);
    \node[case graph label,above] at (-.6,.55)
      {$A\subsetneq[n]\setminus X'$};
    \node[case graph label,above] at (1.30,.7) {$S$};
    \node[case graph label,above] at (3.3,.55) {$X\subseteq X'$};
    \node[case graph label,above] at (.43,.02) {$p$};
    \node[case graph decoration,below] at (0,-.29)
      {$\omega_{B\cup\{p\}}$};
  }%
}

\newcommand{\BoundaryCaseSevenTarget}[1][1]{%
  \CaseGraphBox[#1]{%
    \node[case graph vertex] (a)  at (0,0) {$1$};
    \node[case graph vertex] (z)  at (1.10,0) {$0$};
    \node[case graph vertex] (b)  at (2.25,0) {$1$};
    \node[case graph vertex] (x)  at (3.35,0) {$0$};
    \draw[case graph edge] (a)--(z)--(b)--(x);
    \draw[case graph marking] (a)--(-.54,.46);
    \draw[case graph marking] (z)--(1.10,.63);
    \draw[case graph marking] (b)--(2.25,.63);
    \draw[case graph marking] (x)--(3.89,.46);
    \node[case graph label,above] at (-.7,.52) {$A$};
    \node[case graph label,above] at (1.10,.7) {$S\setminus X'$};
    \node[case graph label,above] at (2.25,.7) {$X'\setminus X$};
    \node[case graph label,above] at (3.99,.52) {$X$};
    \node[case graph label,above] at (.38,.02) {$p$};
    \node[case graph label,above] at (1.48,.02) {$s$};
    \node[case graph label,above] at (1.86,.02) {$t$};
    \node[case graph decoration,below] at (0,-.29)
      {$\omega_{B\cup\{p\}}$};
  }%
}

\newcommand{\BoundaryCaseEightSource}[1][1]{%
  \CaseGraphBox[#1]{%
    \node[case graph vertex] (a) at (0,0) {$1$};
    \node[case graph vertex] (z) at (1.55,0) {$0$};
    \draw[case graph edge] (a)--(z);
    \draw[case graph edge]
      (z) to[out=125,in=55,looseness=4.8] (z);
    \draw[case graph marking] (a)--(-.56,.46);
    \draw[case graph marking] (z)--(2.13,.46);
    \node[case graph label,above] at (-.7,.52)
      {$[n]\setminus X'$};
    \node[case graph label,above] at (2.3,.55) {$X'$};
    \node[case graph label,above] at (.48,.02) {$p$};
    \node[case graph decoration,below] at (0,-.29)
      {$\omega_{B\cup\{p\}}$};
  }%
}

\newcommand{\BoundaryCaseEightTarget}[1][1]{%
  \CaseGraphBox[#1]{%
    \node[case graph vertex] (a) at (0,0) {$1$};
    \node[case graph vertex] (z) at (1.55,0) {$0$};
    \draw[case graph edge] (a)--(z);
    \draw[case graph edge]
      (z) to[out=125,in=55,looseness=4.8] (z);
    \draw[case graph marking] (a)--(-.56,.46);
    \draw[case graph marking] (z)--(2.13,.46);
    \node[case graph label,above] at (-.7,.52)
      {$[n]\setminus X'$};
    \node[case graph label,above] at (2.3,.55) {$X'$};
    \node[case graph label,above] at (.48,.02) {$s$};
    \node[case graph label,above] at (1.04,.02) {$t$};
    \node[case graph decoration,below] at (0,-.29)
      {$\omega_{B\cup\{s\}} \cdot (-\psi_s)$};
  }%
}

\newcommand{\BoundaryCaseNineSource}[1][1]{%
  \CaseGraphBox[#1]{%
    \node[case graph vertex] (a) at (0,0) {$1$};
    \node[case graph vertex] (z) at (1.55,0) {$0$};
    \draw[case graph edge] (a)--(z);
    \draw[case graph edge]
      (z) to[out=125,in=55,looseness=4.8] (z);
    \draw[case graph marking] (a)--(-.56,.46);
    \draw[case graph marking] (z)--(2.13,.46);
    \node[case graph label,above] at (-.56,.52)
      {$A\subsetneq[n]\setminus X'$};
    \node[case graph label,above] at (2.3,.55)
      {$X\supsetneq X'$};
    \node[case graph label,above] at (.48,.02) {$p$};
    \node[case graph decoration,below] at (0,-.29)
      {$\omega_{B\cup\{p\}}$};
  }%
}

\newcommand{\BoundaryCaseNineTarget}[1][1]{%
  \CaseGraphBox[#1]{%
    \node[case graph vertex] (a)  at (0,0) {$1$};
    \node[case graph vertex] (z1) at (1.10,0) {$0$};
    \node[case graph vertex] (z2) at (2.35,0) {$0$};
    \draw[case graph edge] (a)--(z1)--(z2);
    \draw[case graph edge]
      (z2) to[out=125,in=55,looseness=4.8] (z2);
    \draw[case graph marking] (a)--(-.54,.46);
    \draw[case graph marking] (z1)--(1.10,.63);
    \draw[case graph marking] (z2)--(2.93,.46);
    \node[case graph label,above] at (-.7,.55) {$A$};
    \node[case graph label,above] at (1.10,.7) {$X\setminus X'$};
    \node[case graph label,above] at (3.2,.55) {$X'$};
    \node[case graph label,above] at (.38,.02) {$p$};
    \node[case graph label,above] at (1.48,.02) {$s$};
    \node[case graph label,above] at (1.93,.02) {$t$};
    \node[case graph decoration,below] at (0,-.29)
      {$\omega_{B\cup\{p\}}$};
  }%
}

\usepackage{amsmath,amssymb}
\usepackage{tikz}

\tikzset{
  second case vertex/.style={
    circle,
    draw,
    line width=.65pt,
    minimum size=5.2mm,
    inner sep=0pt,
    font=\scriptsize
  },
  second case edge/.style={line width=.65pt,line cap=round},
  second case marking/.style={line width=1.45pt,line cap=round},
  second case label/.style={font=\scriptsize,inner sep=1pt},
  second case decoration/.style={font=\scriptsize,inner sep=1pt}
}

\newcommand{\SecondCaseGraphBox}[2][1]{%
  \ensuremath{\mathord{\vcenter{\hbox{%
    \begin{tikzpicture}[x=1cm,y=1cm,scale=#1,transform shape]
      #2
    \end{tikzpicture}%
  }}}}%
}

\newcommand{\SecondBoundaryCaseOneSource}[1][1]{%
  \SecondCaseGraphBox[#1]{%
    \node[second case vertex] (a) at (0,0) {$1$};
    \node[second case vertex] (z) at (1.30,0) {$0$};
    \node[second case vertex] (b) at (2.60,0) {$1$};
    \draw[second case edge] (a)--(z)--(b);
    \draw[second case marking] (a)--(-.56,.46);
    \draw[second case marking] (z)--(1.30,.63);
    \draw[second case marking] (b)--(3.16,.46);
    \node[second case label,above] at (-.56,.55)
      {$[n]\setminus X'$};
    \node[second case label,above] at (1.30,.7) {$X$};
    \node[second case label,above] at (3.3,.52) {$S$};
    \node[second case label,above] at (.43,.02) {$p$};
    \node[second case decoration,below] at (0,-.29)
      {$\omega_{B\cup\{p\}}$};
  }%
}

\newcommand{\SecondBoundaryCaseOneTarget}[1][1]{%
  \SecondCaseGraphBox[#1]{%
    \node[second case vertex] (a) at (0,0) {$1$};
    \node[second case vertex] (z) at (1.30,0) {$0$};
    \node[second case vertex] (b) at (2.60,0) {$1$};
    \draw[second case edge] (a)--(z)--(b);
    \draw[second case marking] (a)--(-.56,.46);
    \draw[second case marking] (z)--(1.30,.63);
    \draw[second case marking] (b)--(3.16,.46);
    \node[second case label,above] at (-.56,.55)
      {$[n]\setminus X'$};
    \node[second case label,above] at (1.30,.7) {$X$};
    \node[second case label,above] at (3.3,.52) {$S$};
    \node[second case label,above] at (.40,.02) {$s$};
    \node[second case label,above] at (.91,.02) {$t$};
    \node[second case decoration,below] at (0,-.29)
      {$\omega_{B\cup\{s\}}$};
    \node[second case decoration,below] at (1.30,-.29)
      {$-\psi_t$};
  }%
}

\newcommand{\SecondBoundaryCaseTwoSource}[1][1]{%
  \SecondCaseGraphBox[#1]{%
    \node[second case vertex] (a) at (0,0) {$1$};
    \node[second case vertex] (z) at (1.30,0) {$0$};
    \node[second case vertex] (b) at (2.60,0) {$1$};
    \draw[second case edge] (a)--(z)--(b);
    \draw[second case marking] (a)--(-.56,.46);
    \draw[second case marking] (z)--(1.30,.63);
    \draw[second case marking] (b)--(3.16,.46);
    \node[second case label,above] at (-.7,.55) {$A$};
    \node[second case label,above] at (1.30,.7) {$X$};
    \node[second case label,above] at (3.3,.52) {$S$};
    \node[second case label,above] at (.43,.02) {$p$};
    \node[second case decoration,below] at (0,-.29)
      {$\omega_{B\cup\{p\}}$};
    \node[second case label] at (0,-.82)
      {$\bigl(B\subseteq[n]\setminus X'\subsetneq A\bigr)$};
  }%
}

\newcommand{\SecondBoundaryCaseTwoTarget}[1][1]{%
  \SecondCaseGraphBox[#1]{%
    \node[second case vertex] (a)  at (0,0) {$1$};
    \node[second case vertex] (z1) at (1.15,0) {$0$};
    \node[second case vertex] (z2) at (2.35,0) {$0$};
    \node[second case vertex] (b)  at (3.50,0) {$1$};
    \draw[second case edge] (a)--(z1)--(z2)--(b);
    \draw[second case marking] (a)--(-.56,.46);
    \draw[second case marking] (z1)--(1.15,.63);
    \draw[second case marking] (z2)--(2.35,.63);
    \draw[second case marking] (b)--(4.06,.46);
    \node[second case label,above] at (-.6,.55)
      {$[n]\setminus X'$};
    \node[second case label,above] at (1.15,.7) {$A\cap X$};
    \node[second case label,above] at (2.35,.7) {$X$};
    \node[second case label,above] at (4.2,.52) {$S$};
    \node[second case label,above] at (.34,.02) {$s$};
    \node[second case label,above] at (.79,.02) {$t$};
    \node[second case decoration,below] at (0,-.29)
      {$\omega_{B\cup\{s\}}$};
  }%
}

\newcommand{\SecondBoundaryCaseThreeSource}[1][1]{%
  \SecondCaseGraphBox[#1]{%
    \node[second case vertex] (z) at (0,0) {$0$};
    \node[second case vertex] (a) at (1.30,0) {$1$};
    \node[second case vertex] (b) at (2.60,0) {$1$};
    \draw[second case edge] (z)--(a)--(b);
    \draw[second case marking] (z)--(-.56,.46);
    \draw[second case marking] (a)--(1.30,.63);
    \draw[second case marking] (b)--(3.16,.46);
    \node[second case label,above] at (-.7,.52) {$X$};
    \node[second case label,above] at (1.30,.7) {$A$};
    \node[second case label,above] at (3.3,.52) {$S$};
    \node[second case label,above] at (.9,.02) {$p$};
    \node[second case label,above] at (1.7,.02) {$q$};
    \node[second case decoration,below] at (1.30,-.29)
      {$\omega_{B\cup\{p\}}$};
    \node[second case label] at (1.30,-.82)
      {$\bigl(q\notin B\subseteq[n]\setminus X'\subseteq A\bigr)$};
  }%
}

\newcommand{\SecondBoundaryCaseThreeTarget}[1][1]{%
  \SecondCaseGraphBox[#1]{%
    \node[second case vertex] (a) at (0,0) {$1$};
    \node[second case vertex] (z) at (1.52,0) {$0$};
    \node[second case vertex] (b) at (2.90,0) {$1$};
    \node[second case vertex] (c) at (1.52,1.3) {$0$};
    \draw[second case edge] (a)--(z)--(b);
    \draw[second case edge] (z)--(c);
    \draw[second case marking] (a)--(-.56,.46);
    \draw[second case marking] (c)--(1.52,1.9);
    \draw[second case marking] (z)--(2.06,.46);
    \draw[second case marking] (b)--(3.46,.46);
    \node[second case label,above] at (-.56,.52)
      {$[n]\setminus X'$};
    \node[second case label,above] at (1.52,1.97) {$X$};
    \node[second case label,above] at (2.1,.53) {$A\cap X'$};
    \node[second case label,above] at (3.6,.52) {$S$};
    \node[second case label,above] at (.40,.02) {$s$};
    \node[second case label,above] at (1.15,.02) {$t$};
    \node[second case label,above] at (1.9,-.35) {$q$};
    \node[second case decoration,below] at (0,-.29)
      {$\omega_{B\cup\{s\}}$};
  }%
}

\newcommand{\SecondBoundaryCaseFourSource}[1][1]{%
  \SecondCaseGraphBox[#1]{%
    \node[second case vertex] (a) at (0,0) {$1$};
    \node[second case vertex] (b) at (1.30,0) {$1$};
    \node[second case vertex] (z) at (2.60,0) {$0$};
    \draw[second case edge] (a)--(b)--(z);
    \draw[second case marking] (a)--(-.56,.46);
    \draw[second case marking] (b)--(1.30,.63);
    \draw[second case marking] (z)--(3.16,.46);
    \node[second case label,above] at (-.56,.55)
      {$[n]\setminus X'$};
    \node[second case label,above] at (1.30,.7) {$S$};
    \node[second case label,above] at (3.3,.52) {$X$};
    \node[second case label,above] at (.43,.02) {$p$};
    \node[second case decoration,below] at (0,-.29)
      {$\omega_{B\cup\{p\}}$};
  }%
}

\newcommand{\SecondBoundaryCaseFourTarget}[1][1]{%
  \SecondCaseGraphBox[#1]{%
    \node[second case vertex] (a) at (0,0) {$1$};
    \node[second case vertex] (b) at (1.30,0) {$1$};
    \node[second case vertex] (z) at (2.60,0) {$0$};
    \draw[second case edge] (a)--(b)--(z);
    \draw[second case marking] (a)--(-.56,.46);
    \draw[second case marking] (b)--(1.30,.63);
    \draw[second case marking] (z)--(3.16,.46);
    \node[second case label,above] at (-.6,.55)
      {$[n]\setminus X'$};
    \node[second case label,above] at (1.30,.7) {$S$};
    \node[second case label,above] at (3.3,.52) {$X'$};
    \node[second case label,above] at (.40,.02) {$s$};
    \node[second case label,above] at (.91,.02) {$t$};
    \node[second case decoration,below] at (0,-.29)
      {$\omega_{B\cup\{s\}}$};
    \node[second case decoration,below] at (1.30,-.29)
      {$-\psi_t$};
  }%
}

\newcommand{\SecondBoundaryCaseFiveSource}[1][1]{%
  \SecondCaseGraphBox[#1]{%
    \node[second case vertex] (a) at (0,0) {$1$};
    \node[second case vertex] (b) at (1.30,0) {$1$};
    \node[second case vertex] (z) at (2.60,0) {$0$};
    \draw[second case edge] (a)--(b)--(z);
    \draw[second case marking] (a)--(-.56,.46);
    \draw[second case marking] (b)--(1.30,.63);
    \draw[second case marking] (z)--(3.16,.46);
    \node[second case label,above] at (-.7,.52) {$A$};
    \node[second case label,above] at (1.30,.7) {$S$};
    \node[second case label,above] at (3.3,.52) {$X$};
    \node[second case label,above] at (.43,.02) {$p$};
    \node[second case decoration,below] at (0,-.29)
      {$\omega_{B\cup\{p\}}$};
    \node[second case label] at (0,-.82)
      {$\bigl(B\subseteq[n]\setminus X'\subsetneq A\bigr)$};
  }%
}

\newcommand{\SecondBoundaryCaseFiveTarget}[1][1]{%
  \SecondCaseGraphBox[#1]{%
    \node[second case vertex] (a)  at (0,0) {$1$};
    \node[second case vertex] (z1) at (1.15,0) {$0$};
    \node[second case vertex] (b)  at (2.30,0) {$1$};
    \node[second case vertex] (z2) at (3.45,0) {$0$};
    \draw[second case edge] (a)--(z1)--(b)--(z2);
    \draw[second case marking] (a)--(-.56,.46);
    \draw[second case marking] (z1)--(1.15,.63);
    \draw[second case marking] (b)--(2.30,.63);
    \draw[second case marking] (z2)--(4.01,.46);
    \node[second case label,above] at (-.6,.55)
      {$[n]\setminus X'$};
    \node[second case label,above] at (1.15,.7) {$A\cap X'$};
    \node[second case label,above] at (2.30,.7) {$S$};
    \node[second case label,above] at (4.2,.52) {$X$};
    \node[second case label,above] at (.34,.02) {$s$};
    \node[second case label,above] at (.79,.02) {$t$};
    \node[second case decoration,below] at (0,-.29)
      {$\omega_{B\cup\{s\}}$};
  }%
}

\newcommand{\SecondBoundaryCaseSixSource}[1][1]{%
  \SecondCaseGraphBox[#1]{%
    \node[second case vertex] (a) at (0,0) {$1$};
    \node[second case vertex] (z) at (1.65,0) {$0$};
    \draw[second case edge] (a)--(z);
    \draw[second case edge]
      (z) to[out=125,in=55,looseness=4.8] (z);
    \draw[second case marking] (a)--(-.56,.46);
    \draw[second case marking] (z)--(2.23,.46);
    \node[second case label,above] at (-.6,.55)
      {$[n]\setminus X'$};
    \node[second case label,above] at (2.4,.55) {$X'$};
    \node[second case label,above] at (.48,.02) {$p$};
    \node[second case decoration,below] at (0,-.29)
      {$\omega_{B\cup\{p\}}$};
  }%
}

\newcommand{\SecondBoundaryCaseSixTarget}[1][1]{%
  \SecondCaseGraphBox[#1]{%
    \node[second case vertex] (a) at (0,0) {$1$};
    \node[second case vertex] (z) at (1.65,0) {$0$};
    \draw[second case edge] (a)--(z);
    \draw[second case edge]
      (z) to[out=125,in=55,looseness=4.8] (z);
    \draw[second case marking] (a)--(-.56,.46);
    \draw[second case marking] (z)--(2.23,.46);
    \node[second case label,above] at (-.6,.55)
      {$[n]\setminus X'$};
    \node[second case label,above] at (2.4,.55) {$X'$};
    \node[second case label,above] at (.47,.02) {$s$};
    \node[second case label,above] at (1.10,.02) {$t$};
    \node[second case decoration,below] at (0,-.29)
      {$\omega_{B\cup\{s\}}$};
    \node[second case decoration,below] at (1.65,-.29)
      {$-\psi_t$};
  }%
}

\newcommand{\SecondBoundaryCaseSevenSource}[1][1]{%
  \SecondCaseGraphBox[#1]{%
    \node[second case vertex] (a) at (0,0) {$1$};
    \node[second case vertex] (z) at (1.65,0) {$0$};
    \draw[second case edge] (a)--(z);
    \draw[second case edge]
      (z) to[out=125,in=55,looseness=4.8] (z);
    \draw[second case marking] (a)--(-.56,.46);
    \draw[second case marking] (z)--(2.23,.46);
    \node[second case label,above] at (-.7,.52) {$A$};
    \node[second case label,above] at (2.45,.55)
      {$X\supsetneq X'$};
    \node[second case label,above] at (.48,.02) {$p$};
    \node[second case decoration,below] at (0,-.29)
      {$\omega_{B\cup\{p\}}$};
    \node[second case label] at (0,-.82)
      {$\bigl(B\subseteq[n]\setminus X'\subsetneq A\bigr)$};
  }%
}

\newcommand{\SecondBoundaryCaseSevenTarget}[1][1]{%
  \SecondCaseGraphBox[#1]{%
    \node[second case vertex] (a)  at (0,0) {$1$};
    \node[second case vertex] (z1) at (1.15,0) {$0$};
    \node[second case vertex] (z2) at (2.45,0) {$0$};
    \draw[second case edge] (a)--(z1)--(z2);
    \draw[second case edge]
      (z2) to[out=125,in=55,looseness=4.8] (z2);
    \draw[second case marking] (a)--(-.6,.52);
    \draw[second case marking] (z1)--(1.15,.65);
    \draw[second case marking] (z2)--(3.2,.52);
    \node[second case label,above] at (-.6,.58)
      {$[n]\setminus X'$};
    \node[second case label,above] at (1.15,.7) {$A\cap X'$};
    \node[second case label,above] at (3.25,.58) {$X$};
    \node[second case label,above] at (.34,.02) {$s$};
    \node[second case label,above] at (.79,.02) {$t$};
    \node[second case decoration,below] at (0,-.29)
      {$\omega_{B\cup\{s\}}$};
  }%
}

\newcommand{\SecondBoundaryCaseEightSource}[1][1]{%
  \SecondCaseGraphBox[#1]{%
    \node[second case vertex] (a) at (0,0) {$1$};
    \node[second case vertex] (z) at (1.65,0) {$0$};
    \draw[second case edge]
      (a) to[out=45,in=135,looseness=1.15] (z);
    \draw[second case edge]
      (a) to[out=-45,in=-135,looseness=1.15] (z);
    \draw[second case marking] (a)--(-.56,.46);
    \draw[second case marking] (z)--(2.23,.46);
    \node[second case label,above] at (-.7,.52) {$A$};
    \node[second case label,above] at (2.45,.52) {$X$};
    \node[second case decoration,below] at (0,-.29)
      {$\omega_B^{S_2}$};
    \node[second case label] at (0,-.92)
      {$\bigl(B\subseteq[n]\setminus X'\subseteq A\bigr)$};
  }%
}

\newcommand{\SecondBoundaryCaseEightTarget}[1][1]{%
  \SecondCaseGraphBox[#1]{%
    \node[second case vertex] (a)  at (0,0) {$1$};
    \node[second case vertex] (z1) at (1.15,0) {$0$};
    \node[second case vertex] (z2) at (2.70,0) {$0$};
    \draw[second case edge] (a)--(z1);
    \draw[second case edge]
      (z1) to[out=45,in=135,looseness=1.15] (z2);
    \draw[second case edge]
      (z1) to[out=-45,in=-135,looseness=1.15] (z2);
    \draw[second case marking] (a)--(-.56,.46);
    \draw[second case marking] (z1)--(1.15,.63);
    \draw[second case marking] (z2)--(3.28,.46);
    \node[second case label,above] at (-.6,.52)
      {$[n]\setminus X'$};
    \node[second case label,above] at (1.15,.67) {$A\cap X'$};
    \node[second case label,above] at (3.45,.52) {$X$};
    \node[second case label,above] at (.34,.02) {$s$};
    \node[second case label,above] at (.79,.02) {$t$};
    \node[second case decoration,below] at (0,-.29)
      {$\omega_{B\cup\{s\}}$};
  }%
}

\newcommand{\SecondBoundaryCaseList}{%
  \begingroup
  \allowdisplaybreaks[4]
  \begin{alignat*}{3}
    \text{1.}
      &\quad \mathcal S^0_{B,X,S}
      &&\qquad \SecondBoundaryCaseOneSource[.72]
      &&\qquad \longrightarrow\quad \SecondBoundaryCaseOneTarget[.72]
      \\[.85em]
    \text{2.}
      &\quad \mathcal S^0_{B,X,S}
      &&\qquad \SecondBoundaryCaseTwoSource[.72]
      &&\qquad \longrightarrow\quad \SecondBoundaryCaseTwoTarget[.72]
      \\[.85em]
    \text{3.}
      &\quad \mathcal N'_{B,X,S}
      &&\qquad \SecondBoundaryCaseThreeSource[.72]
      &&\qquad \longrightarrow\quad \SecondBoundaryCaseThreeTarget[.72]
      \\[.85em]
    \text{4.}
      &\quad \mathcal S^1_{B,S,X}
      &&\qquad \SecondBoundaryCaseFourSource[.72]
      &&\qquad \longrightarrow\quad \SecondBoundaryCaseFourTarget[.72]
      \\[.85em]
    \text{5.}
      &\quad \mathcal S^1_{B,S,X}
      &&\qquad \SecondBoundaryCaseFiveSource[.72]
      &&\qquad \longrightarrow\quad \SecondBoundaryCaseFiveTarget[.72]
      \\[.85em]
    \text{6.}
      &\quad \mathcal L_{B,X'}
      &&\qquad \SecondBoundaryCaseSixSource[.72]
      &&\qquad \longrightarrow\quad \SecondBoundaryCaseSixTarget[.72]
      \\[.85em]
    \text{7.}
      &\quad \mathcal L_{B,X}
      &&\qquad \SecondBoundaryCaseSevenSource[.72]
      &&\qquad \longrightarrow\quad \SecondBoundaryCaseSevenTarget[.72]
      \\[.85em]
    \text{8.}
      &\quad \mathcal D_{B,X}
      &&\qquad \SecondBoundaryCaseEightSource[.72]
      &&\qquad \longrightarrow\quad \SecondBoundaryCaseEightTarget[.72]
  \end{alignat*}
  \endgroup
}

\tikzset{
  product case vertex/.style={
    circle,
    draw,
    line width=.65pt,
    minimum size=5.2mm,
    inner sep=0pt,
    font=\scriptsize
  },
  product case edge/.style={line width=.65pt,line cap=round},
  product case marking/.style={line width=1.45pt,line cap=round},
  product case label/.style={font=\scriptsize,inner sep=1pt},
  product case decoration/.style={font=\scriptsize,inner sep=1pt}
}

\newcommand{\ProductCaseGraphBox}[2][1]{%
  \ensuremath{\mathord{\vcenter{\hbox{%
    \begin{tikzpicture}[x=1cm,y=1cm,scale=#1,transform shape]
      #2
    \end{tikzpicture}%
  }}}}%
}

\newcommand{\ProductCaseFirstFactor}[1][1]{%
  \ProductCaseGraphBox[#1]{%
    \node[product case vertex] (x) at (0,0) {$0$};
    \node[product case vertex] (a) at (1.30,0) {$1$};
    \node[product case vertex] (y) at (2.65,0) {$1$};
    \draw[product case edge] (x)--(a)--(y);
    \draw[product case marking] (x)--(-.56,.48);
    \draw[product case marking] (a)--(1.30,.66);
    \draw[product case marking] (y)--(3.21,.48);
    \node[product case label,above] at (-.6,.55) {$X$};
    \node[product case label,above] at (1.30,.72) {$A$};
    \node[product case label,above] at (3.3,.55) {$Y$};
    \node[product case label,above] at (.88,.02) {$p$};
    \node[product case decoration,below] at (1.30,-.30)
      {$\omega_{B\cup\{p\}}$};
  }%
}

\newcommand{\ProductCaseSecondFactor}[1][1]{%
  \ProductCaseGraphBox[#1]{%
    \node[product case vertex] (g) at (0,0) {$2$};
    \node[product case vertex] (z) at (1.55,0) {$0$};
    \draw[product case edge] (g)--(z);
    \draw[product case marking] (g)--(-.64,.48);
    \draw[product case marking] (z)--(2.17,.48);
    \node[product case label,above] at (-.85,.52)
      {$[n]\setminus X'$};
    \node[product case label,above] at (2.3,.55) {$X'$};
    \node[product case label,above] at (.48,.02) {$s$};
    \node[product case label,above] at (1.06,.02) {$t$};
    \node[product case label] at (.775,.8) {$H^{13} \otimes H^2$};
  }%
}

\newcommand{\ProductCaseEqualTarget}[1][1]{%
  \ProductCaseGraphBox[#1]{%
    \node[product case vertex] (x) at (0,0) {$1$};
    \node[product case vertex] (a) at (1.30,0) {$1$};
    \node[product case vertex] (y) at (2.65,0) {$0$};
    \draw[product case edge] (x)--(a)--(y);
    \draw[product case marking] (x)--(-.56,.48);
    \draw[product case marking] (a)--(1.30,.66);
    \draw[product case marking] (y)--(3.21,.48);
    \node[product case label,above] at (-.6,.55) {$Y$};
    \node[product case label,above] at (1.30,.72) {$A$};
    \node[product case label,above] at (3.3,.55) {$X$};
    \node[product case label,above] at (2.24,.02) {$t$};
    \node[product case label,above] at (1.69,.02) {$s$};
    \node[product case decoration,below] at (2.65,-.30) {$-\psi_t$};
    \node[product case decoration,below] at (1.30,-.30)
      {$\omega_{B\cup\{s\}}$};
  }%
}

\newcommand{\ProductCaseUnequalTarget}[1][1]{%
  \ProductCaseGraphBox[#1]{%
    \node[product case vertex] (x)  at (0,0) {$1$};
    \node[product case vertex] (d)  at (1.20,0) {$1$};
    \node[product case vertex] (a)  at (2.55,0) {$0$};
    \node[product case vertex] (y)  at (3.95,0) {$0$};
    \draw[product case edge] (x)--(d)--(a)--(y);
    \draw[product case marking] (x)--(-.56,.48);
    \draw[product case marking] (d)--(1.20,.68);
    \draw[product case marking] (a)--(2.55,.68);
    \draw[product case marking] (y)--(4.51,.48);
    \node[product case label,above] at (-.6,.55) {$Y$};
    \node[product case label,above] at (1.1,.73)
      {$A\setminus(X'\setminus X)$};
    \node[product case label,above] at (3,.73)
      {$X'\setminus X$};
    \node[product case label,above] at (4.6,.55) {$X$};
    \node[product case label,above] at (2.12,.02) {$t$};
    \node[product case label,above] at (1.63,.02) {$s$};
    \node[product case decoration,below] at (1.2,-.30)
      {$\omega_{B\cup\{s\}}$};
  }%
}

\newcommand{\ProductCaseDisplay}{%
  \begin{equation*}
    \ProductCaseFirstFactor[.82]
    \mathbin{\times}
    \ProductCaseSecondFactor[.82]
    =
\begin{cases}
  \ProductCaseEqualTarget[.82]
    & \text{if }X'=X,
    \\[1.4em]

  \ProductCaseUnequalTarget[.82]
    &
    \begin{aligned}
      &\text{if }X\subsetneq X',\\[-.15em]& X'\setminus X\subseteq A,\\[-.15em]
      &\text{and }B\subseteq A\setminus X',
    \end{aligned}
    \\[1.4em]

  \hspace{1.5cm}0 & \text{otherwise.}
\end{cases}
  \end{equation*}
}
  
\usepackage{amsmath,amssymb}
\usepackage{tikz}

\tikzset{
  second product vertex/.style={
    circle,
    draw,
    line width=.65pt,
    minimum size=5.2mm,
    inner sep=0pt,
    font=\scriptsize
  },
  second product edge/.style={line width=.65pt,line cap=round},
  second product marking/.style={line width=1.45pt,line cap=round},
  second product label/.style={font=\scriptsize,inner sep=1pt},
  second product decoration/.style={font=\scriptsize,inner sep=1pt}
}

\newcommand{\SecondProductGraphBox}[2][1]{%
  \ensuremath{\mathord{\vcenter{\hbox{%
    \begin{tikzpicture}[x=1cm,y=1cm,scale=#1,transform shape]
      #2
    \end{tikzpicture}%
  }}}}%
}

\newcommand{\SecondProductFirstFactor}[1][1]{%
  \SecondProductGraphBox[#1]{%
    \node[second product vertex] (a) at (0,0) {$0$};
    \node[second product vertex] (y) at (1.35,0) {$1$};
    \node[second product vertex] (x) at (2.70,0) {$1$};
    \draw[second product edge] (a)--(y)--(x);
    \draw[second product marking] (a)--(-.56,.48);
    \draw[second product marking] (y)--(1.35,.66);
    \draw[second product marking] (x)--(3.26,.48);
    \node[second product label,above] at (-.6,.55) {$X$};
    \node[second product label,above] at (1.35,.72) {$Y$};
    \node[second product label,above] at (3.35,.55) {$A$};
    \node[second product label,above] at (2.27,.02) {$p$};
    \node[second product decoration,below] at (2.7,-.30)
      {$\omega_{B\cup\{p\}}$};
  }%
}

\newcommand{\SecondProductSecondFactor}[1][1]{%
  \SecondProductGraphBox[#1]{%
    \node[second product vertex] (g) at (0,0) {$2$};
    \node[second product vertex] (z) at (1.55,0) {$0$};
    \draw[second product edge] (g)--(z);
    \draw[second product marking] (g)--(-.64,.48);
    \draw[second product marking] (z)--(2.17,.48);
    \node[second product label,above] at (-.85,.52)
      {$[n]\setminus X'$};
    \node[second product label,above] at (2.3,.55) {$X'$};
    \node[second product label,above] at (.48,.02) {$s$};
    \node[second product label,above] at (1.06,.02) {$t$};
    \node[second product label] at (.775,.8)
      {$H^{13}\otimes H^2$};
  }%
}

\newcommand{\SecondProductEqualTarget}[1][1]{%
  \SecondProductGraphBox[#1]{%
    \node[second product vertex] (a) at (0,0) {$1$};
    \node[second product vertex] (y) at (1.35,0) {$1$};
    \node[second product vertex] (x) at (2.70,0) {$0$};
    \draw[second product edge] (a)--(y)--(x);
    \draw[second product marking] (a)--(-.56,.48);
    \draw[second product marking] (y)--(1.35,.66);
    \draw[second product marking] (x)--(3.26,.48);
    \node[second product label,above] at (-.6,.55) {$A$};
    \node[second product label,above] at (1.35,.72) {$Y$};
    \node[second product label,above] at (3.35,.55) {$X$};
    \node[second product label,above] at (.43,.02) {$p$};
    \node[second product label,above] at (1.78,.02) {$s$};
    \node[second product label,above] at (2.27,.02) {$t$};
    \node[second product decoration,below] at (0,-.30)
      {$\omega_{B\cup\{p\}}$};
    \node[second product decoration,below] at (2.70,-.30)
      {$-\psi_t$};
  }%
}

\newcommand{\SecondProductUnequalTarget}[1][1]{%
  \SecondProductGraphBox[#1]{%
    \node[second product vertex] (a) at (0,0) {$1$};
    \node[second product vertex] (y) at (1.45,0) {$1$};
    \node[second product vertex] (d) at (2.90,0) {$0$};
    \node[second product vertex] (x) at (4.15,0) {$0$};
    \draw[second product edge] (a)--(y)--(d)--(x);
    \draw[second product marking] (a)--(-.56,.48);
    \draw[second product marking] (y)--(1.45,.70);
    \draw[second product marking] (d)--(2.90,.70);
    \draw[second product marking] (x)--(4.71,.48);
    \node[second product label,above] at (-.6,.55) {$A$};
    \node[second product label,above] at (1.00,.75)
      {$Y\setminus(X'\setminus X)$};
    \node[second product label,above] at (3.00,.75)
      {$X'\setminus X$};
    \node[second product label,above] at (4.8,.55) {$X$};
    \node[second product label,above] at (.45,.02) {$p$};
    \node[second product label,above] at (1.93,.02) {$s$};
    \node[second product label,above] at (2.42,.02) {$t$};
    \node[second product decoration,below] at (0,-.30)
      {$\omega_{B\cup\{p\}}$};
  }%
}

\newcommand{\SecondProductCaseDisplay}{%
  \begin{equation*}
    \SecondProductFirstFactor[.82]
    \mathbin{\times}
    \SecondProductSecondFactor[.82]
    =
  \begin{cases}
    \SecondProductEqualTarget[.82]
      & \text{if }X'=X,
      \\[1.4em]

    \SecondProductUnequalTarget[.82]
      &
      \begin{aligned}
        &\text{if }X\subsetneq X'\text{ and}\\[-.15em]
        &X'\setminus X\subseteq Y,
      \end{aligned}
      \\[1.4em]

    \hspace{1.5cm}0 & \text{otherwise.}
  \end{cases}
  \end{equation*}%
}

\usepackage{amsmath,amssymb}
\usepackage{tikz}

\newcommand{\GraphLPrimeCutFactor}[1][1]{%
  \AuxiliaryGraphBox[#1]{%
    \node[auxiliary graph vertex] (a) at (0,0) {$1$};
    \node[auxiliary graph vertex] (z) at (1.45,0) {$0$};

    \draw[auxiliary graph edge] (a)--(z);
    \draw[auxiliary graph marking] (a)--(0,.67);
    \draw[auxiliary graph marking] (z)--(2.07,.50);
    \draw[auxiliary graph marking] (z)--(2.10,0);

    \node[auxiliary graph label,above] at (0,.75)
      {$[n]\setminus X$};
    \node[auxiliary graph label,above] at (2.07,.57)
      {$X\setminus X'$};
    \node[auxiliary graph label,above] at (.48,.02) {$p$};
    \node[auxiliary graph label,right] at (2.13,.0) {$s$};
    \node[auxiliary graph decoration,below] at (0,-.31)
      {$\omega_{B\cup\{p\}}$};
  }%
}

\newcommand{\GraphLPrimeIrreducible}[1][1]{%
  \AuxiliaryGraphBox[#1]{%
    \node[auxiliary graph vertex] (z) at (0,0) {$0$};

    \draw[auxiliary graph edge]
      (z) to[out=60,in=120,looseness=4.6] (z);
    \draw[auxiliary graph marking] (z)--(-.66,0);
    \draw[auxiliary graph marking] (z)--(.62,.48);

    \node[auxiliary graph label,left] at (-.69,.0) {$t$};
    \node[auxiliary graph label,above] at (.67,.55) {$X'$};
  }%
}

\newcommand{\GraphLPrimeBoundary}[1][1]{%
  \AuxiliaryGraphBox[#1]{%
    \node[auxiliary graph vertex] (z) at (0,0) {$0$};
    \node[auxiliary graph vertex] (a) at (1.45,0) {$1$};

    \draw[auxiliary graph edge] (z)--(a);
    \draw[auxiliary graph marking] (z)--(-.66,0);
    \draw[auxiliary graph marking] (z)--(0,.65);
    \draw[auxiliary graph marking] (a)--(2.07,.48);

    \node[auxiliary graph label,left] at (-.69,0) {$t$};
    \node[auxiliary graph label,above] at (0,.72) {$T$};
    \node[auxiliary graph label,above] at (2.07,.52)
      {$X'\setminus T$};
  }%
}

\newcommand{\GraphLPrimePsi}[1][1]{%
  \AuxiliaryGraphBox[#1]{%
    \node[auxiliary graph vertex] (a) at (0,0) {$1$};

    \draw[auxiliary graph marking] (a)--(-.66,0);
    \draw[auxiliary graph marking] (a)--(.62,.48);

    \node[auxiliary graph label,left] at (-.7,.0) {$t$};
    \node[auxiliary graph label,above] at (.68,.55) {$X'$};
    \node[auxiliary graph decoration,below] at (0,-.31)
      {$-\psi_t$};
  }%
}

\newcommand{\LPrimeProductDisplay}{%
  \begin{align*}
    \mathcal L'_{B,X}
    \mathbin{\times}
\GraphKunnethThirteenTwo[.82] =&\GraphLPrimeCutFactor[.82]\\
     &\mathbin{\otimes} 
    \left(
      \GraphLPrimeIrreducible[.82]
      +24\!\sum_{\substack{\varnothing\neq T\subseteq X'}}
        \GraphLPrimeBoundary[.82]
      +24\,\GraphLPrimePsi[.82]
    \right).
  \end{align*}%
}

\newcommand{\GraphEtaPrimeDoubleEdge}[1][1]{%
  \AuxiliaryGraphBox[#1]{%
    \node[auxiliary graph vertex] (a) at (0,0) {$1$};
    \node[auxiliary graph vertex] (z) at (1.75,0) {$0$};

    \draw[auxiliary graph edge]
      (a) to[out=45,in=135,looseness=1.20] (z);
    \draw[auxiliary graph edge]
      (a) to[out=-45,in=-135,looseness=1.20] (z);

    \draw[auxiliary graph marking] (a)--(0,.66);
    \draw[auxiliary graph marking] (z)--(2.38,.48);
    \draw[auxiliary graph marking] (z)--(2.38,-.48);

    \node[auxiliary graph label,above] at (0,.69) {$B$};
    \node[auxiliary graph label,above] at (2.46,.52) {$i$};
    \node[auxiliary graph label,below] at (2.45,-.5) {$j$};

    \node[auxiliary graph decoration,below] at (0,-.33)
      {$\omega_B^{S_2}$};
  }%
}

\newcommand{\GraphEtaPrimeChainSource}[1][1]{%
  \AuxiliaryGraphBox[#1]{%
    \node[auxiliary graph vertex] (a) at (0,0) {$1$};
    \node[auxiliary graph vertex] (b) at (1.35,0) {$1$};
    \node[auxiliary graph vertex] (z) at (2.65,0) {$0$};

    \draw[auxiliary graph edge] (a)--(b)--(z);
    \draw[auxiliary graph marking] (a)--(0,.66);
    \draw[auxiliary graph marking] (z)--(3.27,.48);
    \draw[auxiliary graph marking] (z)--(3.27,-.48);

    \node[auxiliary graph label,above] at (0,.69) {$B$};
    \node[auxiliary graph label,above] at (.43,.02) {$p$};
    \node[auxiliary graph label,above] at (3.4,.52) {$i$};
    \node[auxiliary graph label,below] at (3.4,-.52) {$j$};

    \node[auxiliary graph decoration,below] at (0,-.33)
      {$\omega_{B\cup\{p\}}$};
  }%
}

\newcommand{\GraphEtaPrimeChainTarget}[1][1]{%
  \AuxiliaryGraphBox[#1]{%
    \node[auxiliary graph vertex] (a) at (0,0) {$1$};
    \node[auxiliary graph vertex] (b) at (1.35,0) {$1$};
    \node[auxiliary graph vertex] (z) at (2.65,0) {$0$};

    \draw[auxiliary graph edge] (a)--(b)--(z);
    \draw[auxiliary graph marking] (a)--(0,.66);
    \draw[auxiliary graph marking] (z)--(3.27,.48);
    \draw[auxiliary graph marking] (z)--(3.27,-.48);

    \node[auxiliary graph label,above] at (0,.69) {$B$};
    \node[auxiliary graph label,above] at (.43,.02) {$s$};
    \node[auxiliary graph label,above] at (3.4,.52) {$i$};
    \node[auxiliary graph label,below] at (3.4,-.52) {$j$};

    \node[auxiliary graph decoration,below] at (0,-.33)
      {$\omega_{B\cup\{s\}}\cdot(-\psi_s)$};
  }%
}

\usepackage{amsmath,amssymb}
\usepackage{tikz}

\newcommand{\GraphMijExpandedOne}[1][1]{%
  \AuxiliaryGraphBox[#1]{%
    \node[auxiliary graph vertex] (a)  at (0,0) {$1$};
    \node[auxiliary graph vertex] (zj) at (1.25,0) {$0$};
    \node[auxiliary graph vertex] (zi) at (2.55,0) {$0$};
    \node[auxiliary graph vertex] (b)  at (3.80,0) {$1$};

    \draw[auxiliary graph edge] (a)--(zj)--(zi)--(b);
    \draw[auxiliary graph marking] (a)--(0,.66);
    \draw[auxiliary graph marking] (zj)--(1.25,.66);
    \draw[auxiliary graph marking] (zi)--(2.55,.66);

    \node[auxiliary graph label,above] at (0,.74) {$B$};
    \node[auxiliary graph label,above] at (1.25,.74) {$j$};
    \node[auxiliary graph label,above] at (2.55,.74) {$i$};
    \node[auxiliary graph label,above] at (.40,.02) {$p$};
    \node[auxiliary graph label,above] at (1.68,.02) {$s$};
    \node[auxiliary graph label,above] at (2.12,.02) {$t$};
    \node[auxiliary graph decoration,below] at (0,-.33)
      {$\omega_{B\cup\{p\}}$};
  }%
}

\newcommand{\GraphMijExpandedTwo}[1][1]{%
  \AuxiliaryGraphBox[#1]{%
    \node[auxiliary graph vertex] (a)  at (0,0) {$1$};
    \node[auxiliary graph vertex] (zj) at (1.25,0) {$0$};
    \node[auxiliary graph vertex] (zi) at (2.55,0) {$0$};

    \draw[auxiliary graph edge] (a)--(zj)--(zi);
    \draw[auxiliary graph edge]
      (zi) to[out=60,in=120,looseness=4.3] (zi);
    \draw[auxiliary graph marking] (a)--(0,.66);
    \draw[auxiliary graph marking] (zj)--(1.25,.66);
    \draw[auxiliary graph marking] (zi)--(3.17,.48);

    \node[auxiliary graph label,above] at (0,.74) {$B$};
    \node[auxiliary graph label,above] at (1.25,.74) {$j$};
    \node[auxiliary graph label,above] at (3.20,.57) {$i$};
    \node[auxiliary graph label,above] at (.40,.02) {$p$};
    \node[auxiliary graph label,above] at (1.68,.02) {$s$};
    \node[auxiliary graph label,above] at (2.12,.02) {$t$};
    \node[auxiliary graph decoration,below] at (0,-.33)
      {$\omega_{B\cup\{p\}}$};
  }%
}

\newcommand{\GraphMijExpandedThree}[1][1]{%
  \AuxiliaryGraphBox[#1]{%
    \node[auxiliary graph vertex] (a)  at (0,0) {$1$};
    \node[auxiliary graph vertex] (zj) at (1.25,0) {$0$};
    \node[auxiliary graph vertex] (b)  at (2.55,0) {$1$};

    \draw[auxiliary graph edge] (a)--(zj)--(b);
    \draw[auxiliary graph marking] (a)--(0,.66);
    \draw[auxiliary graph marking] (zj)--(1.25,.66);
    \draw[auxiliary graph marking] (b)--(3.17,.48);

    \node[auxiliary graph label,above] at (0,.74) {$B$};
    \node[auxiliary graph label,above] at (1.25,.74) {$j$};
    \node[auxiliary graph label,above] at (3.20,.57) {$i$};
    \node[auxiliary graph label,above] at (.40,.02) {$p$};
    \node[auxiliary graph label,above] at (1.68,.02) {$s$};
    \node[auxiliary graph label,above] at (2.12,.02) {$t$};
    \node[auxiliary graph decoration,below] at (0,-.33)
      {$\omega_{B\cup\{p\}}$};
    \node[auxiliary graph decoration,below] at (2.55,-.33)
      {$-\psi_t$};
  }%
}

\newcommand{\GraphMijCommonFactor}[1][1]{%
  \AuxiliaryGraphBox[#1]{%
    \node[auxiliary graph vertex] (a)  at (0,0) {$1$};
    \node[auxiliary graph vertex] (zj) at (1.25,0) {$0$};

    \draw[auxiliary graph edge] (a)--(zj);
    \draw[auxiliary graph marking] (zj)--(1.92,0);
    \draw[auxiliary graph marking] (a)--(0,.66);
    \draw[auxiliary graph marking] (zj)--(1.25,.66);

    \node[auxiliary graph label,above] at (0,.74) {$B$};
    \node[auxiliary graph label,above] at (1.25,.74) {$j$};
    \node[auxiliary graph label,above] at (.40,.02) {$p$};
    \node[auxiliary graph label,right] at (1.97,.02) {$s$};
    \node[auxiliary graph decoration,below] at (0,-.33)
      {$\omega_{B\cup\{p\}}$};
  }%
}

\newcommand{\GraphMijRightOne}[1][1]{%
  \AuxiliaryGraphBox[#1]{%
    \node[auxiliary graph vertex] (zi) at (0,0) {$0$};
    \node[auxiliary graph vertex] (b)  at (1.35,0) {$1$};

    \draw[auxiliary graph edge] (zi)--(b);
    \draw[auxiliary graph marking] (zi)--(-.66,0);
    \draw[auxiliary graph marking] (zi)--(0,.66);

    \node[auxiliary graph label,left] at (-.72,.02) {$t$};
    \node[auxiliary graph label,above] at (0,.74) {$i$};
  }%
}

\newcommand{\GraphMijRightTwo}[1][1]{%
  \AuxiliaryGraphBox[#1]{%
    \node[auxiliary graph vertex] (zi) at (0,0) {$0$};

    \draw[auxiliary graph marking] (zi)--(-.66,0);
    \draw[auxiliary graph edge]
      (zi) to[out=60,in=120,looseness=4.3] (zi);
    \draw[auxiliary graph marking] (zi)--(.62,.48);

    \node[auxiliary graph label,left] at (-.72,.02) {$t$};
    \node[auxiliary graph label,above] at (.65,.55) {$i$};
  }%
}

\newcommand{\GraphMijRightThree}[1][1]{%
  \AuxiliaryGraphBox[#1]{%
    \node[auxiliary graph vertex] (b) at (0,0) {$1$};

    \draw[auxiliary graph marking] (b)--(-.66,0);
    \draw[auxiliary graph marking] (b)--(.62,.48);

    \node[auxiliary graph label,left] at (-.7,0) {$t$};
    \node[auxiliary graph label,above] at (.67,.54) {$i$};
    \node[auxiliary graph decoration,below] at (0,-.33)
      {$-\psi_t$};
  }%
}

\newcommand{\EtaPrimeMijPullbackDisplay}{%
  \begin{align*}
    (\eta')^*(M_{ij})
      &={}
        24\,\GraphMijExpandedOne[.78]
        +\GraphMijExpandedTwo[.78]
        +24\,\GraphMijExpandedThree[.78]
        \\[1.2em]
      &={}
        \GraphMijCommonFactor[.82]
        \mathbin{\otimes}
        \left(
          24\,\GraphMijRightOne[.82]
          +\GraphMijRightTwo[.82]
          +24\,\GraphMijRightThree[.82]
        \right).
  \end{align*}%
}

\usepackage{amsmath,amssymb}
\usepackage{tikz}

\newcommand{\GraphMijPairExpandedPsi}[1][1]{%
  \AuxiliaryGraphBox[#1]{%
    \node[auxiliary graph vertex] (a) at (0,0) {$1$};
    \node[auxiliary graph vertex] (z) at (1.45,0) {$0$};
    \node[auxiliary graph vertex] (b) at (2.90,0) {$1$};

    \draw[auxiliary graph edge] (a)--(z)--(b);
    \draw[auxiliary graph marking] (a)--(0,.66);
    \draw[auxiliary graph marking] (z)--(1.17,.66);
    \draw[auxiliary graph marking] (z)--(1.73,.66);

    \node[auxiliary graph label,above] at (0,.69) {$B$};
    \node[auxiliary graph label,above] at (1.17,.69) {$i$};
    \node[auxiliary graph label,above] at (1.73,.69) {$j$};
    \node[auxiliary graph label,above] at (.42,.02) {$p$};
    \node[auxiliary graph label,above] at (1.88,.02) {$s$};
    \node[auxiliary graph label,above] at (2.48,.02) {$t$};
    \node[auxiliary graph decoration,below] at (0,-.33)
      {$\omega_{B\cup\{p\}}$};
    \node[auxiliary graph decoration,below] at (2.90,-.33)
      {$-\psi_t$};
  }%
}

\newcommand{\GraphMijPairExpandedLoop}[1][1]{%
  \AuxiliaryGraphBox[#1]{%
    \node[auxiliary graph vertex] (a) at (0,0) {$1$};
    \node[auxiliary graph vertex] (z) at (1.45,0) {$0$};
    \node[auxiliary graph vertex] (w) at (2.90,0) {$0$};

    \draw[auxiliary graph edge] (a)--(z)--(w);
    \draw[auxiliary graph edge]
      (w) to[out=60,in=120,looseness=4.3] (w);
    \draw[auxiliary graph marking] (a)--(0,.66);
    \draw[auxiliary graph marking] (z)--(1.17,.66);
    \draw[auxiliary graph marking] (z)--(1.73,.66);

    \node[auxiliary graph label,above] at (0,.69) {$B$};
    \node[auxiliary graph label,above] at (1.17,.69) {$i$};
    \node[auxiliary graph label,above] at (1.73,.69) {$j$};
    \node[auxiliary graph label,above] at (.42,.02) {$p$};
    \node[auxiliary graph label,above] at (1.88,.02) {$s$};
    \node[auxiliary graph label,above] at (2.48,.02) {$t$};
    \node[auxiliary graph decoration,below] at (0,-.33)
      {$\omega_{B\cup\{p\}}$};
  }%
}

\newcommand{\GraphMijPairCommonFactor}[1][1]{%
  \AuxiliaryGraphBox[#1]{%
    \node[auxiliary graph vertex] (a) at (0,0) {$1$};
    \node[auxiliary graph vertex] (z) at (1.45,0) {$0$};

    \draw[auxiliary graph edge] (a)--(z);
    \draw[auxiliary graph marking] (z)--(2.12,0);
    \draw[auxiliary graph marking] (a)--(0,.66);
    \draw[auxiliary graph marking] (z)--(1.17,.66);
    \draw[auxiliary graph marking] (z)--(1.73,.66);

    \node[auxiliary graph label,above] at (0,.69) {$B$};
    \node[auxiliary graph label,above] at (1.17,.69) {$i$};
    \node[auxiliary graph label,above] at (1.73,.69) {$j$};
    \node[auxiliary graph label,above] at (.42,.02) {$p$};
    \node[auxiliary graph label,right] at (2.14,.02) {$s$};
    \node[auxiliary graph decoration,below] at (0,-.33)
      {$\omega_{B\cup\{p\}}$};
  }%
}

\newcommand{\GraphMijPairRightPsi}[1][1]{%
  \AuxiliaryGraphBox[#1]{%
    \node[auxiliary graph vertex] (b) at (0,0) {$1$};

    \draw[auxiliary graph marking] (-.66,0)--(b);

    \node[auxiliary graph label,left] at (-.68,.02) {$t$};
    \node[auxiliary graph decoration,below] at (0,-.25)
      {$-\psi_t$};
  }%
}

\newcommand{\GraphMijPairRightLoop}[1][1]{%
  \AuxiliaryGraphBox[#1]{%
    \node[auxiliary graph vertex] (w) at (0,0) {$0$};

    \draw[auxiliary graph marking] (-.66,0)--(w);
    \draw[auxiliary graph marking] (.55,.45)--(w);
    \draw[auxiliary graph edge]
      (w) to[out=60,in=120,looseness=4.3] (w);

    \node[auxiliary graph label,left] at (-.68,.0) {$t$};
    \node[auxiliary graph label,right] at (.58,.47) {$i$};
  }%
}

\newcommand{\EtaPrimeMijPairPullbackDisplay}{%
  \begin{equation*}
    (\eta')^*( M_{ij})
    =
    24\,\GraphMijPairExpandedPsi[.76]
    +\GraphMijPairExpandedLoop[.76]
    =
    \GraphMijPairCommonFactor[.80]
    \mathbin{\otimes}
    \left(
      24\,\GraphMijPairRightPsi[.80]
      +\GraphMijPairRightLoop[.80]
    \right).
  \end{equation*}%
}

\newcommand{\GraphMijFourCommonFactor}[1][1]{%
  \AuxiliaryGraphBox[#1]{%
    \node[auxiliary graph vertex] (a) at (0,0) {$1$};

    \draw[auxiliary graph marking] (a)--(.74,0);
    \draw[auxiliary graph marking] (a)--(0,.66);

    \node[auxiliary graph label,above] at (0,.69) {$B$};
    \node[auxiliary graph label,right] at (.76,0) {$s$};
    \node[auxiliary graph decoration,below] at (0,-.33)
      {$\omega_{B\cup\{s\}}$};
  }%
}

\newcommand{\GraphMijFourRightOne}[1][1]{%
  \AuxiliaryGraphBox[#1]{%
    \node[auxiliary graph vertex] (z) at (0,0) {$0$};
    \node[auxiliary graph vertex] (a) at (1.30,0) {$1$};

    \draw[auxiliary graph edge] (z)--(a);
    \draw[auxiliary graph marking] (z)--(-.66,0);
    \draw[auxiliary graph marking] (z)--(-.24,.64);
    \draw[auxiliary graph marking] (z)--(.24,.64);

    \node[auxiliary graph label,left] at (-.68,0) {$t$};
    \node[auxiliary graph label,above] at (-.24,.67) {$i$};
    \node[auxiliary graph label,above] at (.24,.67) {$j$};

    \node[auxiliary graph decoration,below] at (1.30,-.33)
      {$-\psi_t$};
  }%
}

\newcommand{\GraphMijFourRightTwo}[1][1]{%
  \AuxiliaryGraphBox[#1]{%
    \node[auxiliary graph vertex] (z) at (0,0) {$0$};

    \draw[auxiliary graph marking] (-.66,0)--(z);
    \draw[auxiliary graph edge]
      (z) to[out=125,in=55,looseness=4.5] (z);

    \draw[auxiliary graph marking] (z)--(.67,.46);
    \draw[auxiliary graph marking] (z)--(.67,-.46);

    \node[auxiliary graph label,left] at (-.68,.02) {$t$};
    \node[auxiliary graph label,above] at (.72,.46) {$i$};
    \node[auxiliary graph label,below] at (.72,-.46) {$j$};

    \node[auxiliary graph decoration,below] at (0,-.33)
      {$-\psi_t$};
  }%
}

\newcommand{\GraphMijFourRightThree}[1][1]{%
  \AuxiliaryGraphBox[#1]{%
    \node[auxiliary graph vertex] (a) at (0,0) {$1$};
    \node[auxiliary graph vertex] (z) at (1.30,0) {$0$};

    \draw[auxiliary graph edge] (a)--(z);
    \draw[auxiliary graph marking] (-.66,0)--(a);
    \draw[auxiliary graph marking] (z)--(1.97,.46);
    \draw[auxiliary graph marking] (z)--(1.97,-.46);

    \node[auxiliary graph label,left] at (-.68,0) {$t$};
    \node[auxiliary graph label,above] at (2.15,.46) {$i$};
    \node[auxiliary graph label,below] at (2.15,-.46) {$j$};

    \node[auxiliary graph decoration,below] at (0,-.33)
      {$-\psi_t$};
  }%
}

\newcommand{\GraphMijFourRightFour}[1][1]{%
  \AuxiliaryGraphBox[#1]{%
    \node[auxiliary graph vertex] (z1) at (0,0) {$0$};
    \node[auxiliary graph vertex] (z2) at (1.25,0) {$0$};

    \draw[auxiliary graph marking] (-.66,0)--(z1);

    \draw[auxiliary graph edge]
      (z1) to[out=45,in=135,looseness=1.20] (z2);

    \draw[auxiliary graph edge]
      (z1) to[out=-45,in=-135,looseness=1.20] (z2);

    \draw[auxiliary graph marking] (z2)--(1.92,.46);
    \draw[auxiliary graph marking] (z2)--(1.92,-.46);

    \node[auxiliary graph label,left] at (-.68,.02) {$t$};
    \node[auxiliary graph label,above] at (2.12,.46) {$i$};
    \node[auxiliary graph label,below] at (2.12,-.46) {$j$};
  }%
}

\newcommand{\EtaPrimeMijFourTermPullbackDisplay}{%
  \begin{equation*}
    (\eta')^*(M_{ij})
    =
    \GraphMijFourCommonFactor[.80]
    \mathbin{\otimes}
    \left(
      24\,\GraphMijFourRightOne[.78]
      +\GraphMijFourRightTwo[.78]
      -24\,\GraphMijFourRightThree[.78]
      +2\,\GraphMijFourRightFour[.78]
    \right).
  \end{equation*}%
}

\newcommand{\GraphLPrimeElevenFourTarget}[1][1]{%
  \AuxiliaryGraphBox[#1]{%
    \node[auxiliary graph vertex] (a) at (0,0) {$1$};

    \draw[auxiliary graph marking] (a)--(0,.68);
    \draw[auxiliary graph marking] (a)--(.78,0);

    \node[auxiliary graph label,above] at (0,.71)
      {$[n]\setminus X'$};

    \node[auxiliary graph label,right] at (.80,0)
      {$s$};

    \node[auxiliary graph decoration,below] at (0,-.33)
      {$\omega_{B\cup\{s\}}$};
  }%
}

\newcommand{\LPrimeElevenFourProductDisplay}{%
  \begin{equation*}
    \mathcal L'_{B,X}
    \mathbin{\times}
    \GraphKunnethElevenFour[.82]
    =
    \begin{cases}
      \GraphLPrimeElevenFourTarget[.82]
      \mathbin{\otimes}\alpha
      &
      \text{if }X'=X,
      \\[1.5em]

      \hspace{.5cm}0
      &
      \text{otherwise.}
    \end{cases}
  \end{equation*}%
}

\newcommand{\GraphDoubleEdgeOmegaFactor}[1][1]{%
  \AuxiliaryGraphBox[#1]{%
    \node[auxiliary graph vertex] (a) at (0,0) {$1$};
    \node[auxiliary graph vertex] (z) at (1.40,0) {$0$};

    \draw[auxiliary graph edge]
      (a) to[out=35,in=145,looseness=1.20] (z);

    \draw[auxiliary graph edge]
      (a) to[out=-35,in=-145,looseness=1.20] (z);

    \draw[auxiliary graph marking] (a)--(0,.68);
    \draw[auxiliary graph marking] (z)--(2.05,.48);

    \node[auxiliary graph label,above] at (0,.71)
      {$[n]\setminus X$};

    \node[auxiliary graph label,above] at (2.05,.50)
      {$X$};

    \node[auxiliary graph decoration,below] at (0,-.34)
      {$\omega_B^{S_2}$};
  }%
}

\newcommand{\GraphDoubleEdgeBetaTarget}[1][1]{%
  \AuxiliaryGraphBox[#1]{%
    \node[auxiliary graph vertex] (a) at (0,0) {$1$};

    \draw[auxiliary graph marking] (a)--(0,.68);
    \draw[auxiliary graph marking] (a)--(.72,0);

    \node[auxiliary graph label,above] at (0,.71)
      {$[n]\setminus X'$};

    \node[auxiliary graph label,right] at (.74,0)
      {$s$};

    \node[auxiliary graph decoration,below] at (0,-.34)
      {$\omega_{B\cup\{s\}}$};
  }%
}

\newcommand{\DoubleEdgeElevenFourProductDisplay}{%
  \begin{equation*}
    \GraphDoubleEdgeOmegaFactor[.82]
    \mathbin{\times}
    \GraphKunnethElevenFour[.82]
    =
    \begin{cases}
      \GraphDoubleEdgeBetaTarget[.82]
      \mathbin{\otimes}\beta
      &
      \text{if }X'=X,
      \\[1.5em]

      \hspace{0.5cm}0
      &
      \text{otherwise.}
    \end{cases}
  \end{equation*}%
}

\newcommand{\GraphAlphaLoopTerm}[1][1]{%
  \AuxiliaryGraphBox[#1]{%
    \node[auxiliary graph vertex] (z) at (0,0) {$0$};

    \draw[auxiliary graph marking] (z)--(-.68,0);
    \draw[auxiliary graph marking] (z)--(.66,.48);

    \draw[auxiliary graph edge]
      (z) to[out=125,in=55,looseness=4.5] (z);

    \node[auxiliary graph label,left] at (-.70,.02)
      {$t$};

    \node[auxiliary graph label,above] at (.70,.50)
      {$X'$};

    \node[auxiliary graph decoration,below] at (0,-.34)
      {$-\psi_t$};
  }%
}

\newcommand{\GraphAlphaBoundaryTerm}[1][1]{%
  \AuxiliaryGraphBox[#1]{%
    \node[auxiliary graph vertex] (z) at (0,0) {$0$};
    \node[auxiliary graph vertex] (a) at (1.45,0) {$1$};

    \draw[auxiliary graph edge] (z)--(a);

    \draw[auxiliary graph marking] (z)--(-.68,0);
    \draw[auxiliary graph marking] (z)--(0,.66);
    \draw[auxiliary graph marking] (a)--(2.08,.48);

    \node[auxiliary graph label,left] at (-.70,.02)
      {$t$};

    \node[auxiliary graph label,above] at (0,.69)
      {$T$};

    \node[auxiliary graph label,above] at (2.08,.50)
      {$X'\setminus T$};

    \node[auxiliary graph decoration,below] at (0,-.34)
      {$-\psi_t$};
  }%
}

\newcommand{\GraphBetaDoubleEdgeTerm}[1][1]{%
  \AuxiliaryGraphBox[#1]{%
    \node[auxiliary graph vertex] (z1) at (0,0) {$0$};
    \node[auxiliary graph vertex] (z2) at (1.35,0) {$0$};

    \draw[auxiliary graph marking] (z1)--(-.68,0);
    \draw[auxiliary graph marking] (z2)--(2.03,.48);

    \draw[auxiliary graph edge]
      (z1) to[out=45,in=135,looseness=1.20] (z2);

    \draw[auxiliary graph edge]
      (z1) to[out=-45,in=-135,looseness=1.20] (z2);

    \node[auxiliary graph label,left] at (-.70,0)
      {$t$};

    \node[auxiliary graph label,above] at (2.2,.50)
      {$X'$};
  }%
}

\newcommand{\AlphaBetaDefinitionsDisplay}{%
  \begin{equation*}
    \alpha
    :=
    \GraphAlphaLoopTerm[.82]
    +
    24\!
    \sum_{\substack{\varnothing\neq T\subseteq X'}}
    \GraphAlphaBoundaryTerm[.82]
    \qquad\text{and}\qquad
    \beta
    :=
    \GraphBetaDoubleEdgeTerm[.82].
  \end{equation*}%
}

\newcommand{\GraphEtaStarAlpha}[1][1]{%
  \AuxiliaryGraphBox[#1]{%
    \node[auxiliary graph vertex] (z) at (0,0) {$0$};

    \draw[auxiliary graph marking] (z)--(-.68,0);
    \draw[auxiliary graph marking] (z)--(.68,.48);
    \draw[auxiliary graph marking] (z)--(.68,-.48);

    \draw[auxiliary graph edge]
      (z) to[out=125,in=55,looseness=4.5] (z);

    \node[auxiliary graph label,left] at (-.70,0)
      {$t$};

    \node[auxiliary graph label,above] at (.75,.50)
      {$y$};

    \node[auxiliary graph label,below] at (.75,-.50)
      {$a$};
  }%
}

\newcommand{\GraphEtaStarBeta}[1][1]{%
  \AuxiliaryGraphBox[#1]{%
    \node[auxiliary graph vertex] (z1) at (0,0) {$0$};
    \node[auxiliary graph vertex] (z2) at (1.35,0) {$0$};

    \draw[auxiliary graph marking] (z1)--(-.68,0);
    \draw[auxiliary graph marking] (z2)--(2.03,.48);
    \draw[auxiliary graph marking] (z2)--(2.03,-.48);

    \draw[auxiliary graph edge]
      (z1) to[out=45,in=135,looseness=1.20] (z2);

    \draw[auxiliary graph edge]
      (z1) to[out=-45,in=-135,looseness=1.20] (z2);

    \node[auxiliary graph label,left] at (-.70,.02)
      {$t$};

    \node[auxiliary graph label,above] at (2.2,.50)
      {$y$};

    \node[auxiliary graph label,below] at (2.2,-.50)
      {$a$};
  }%
}

\newcommand{\EtaStarAlphaBetaDisplay}{%
  \begin{equation*}
    \xi^*(\alpha)
    =
    \GraphEtaStarAlpha[.82]
    \mathbin{\otimes}1
    \qquad\text{and}\qquad
    \xi^*(\beta)
    =
    \GraphEtaStarBeta[.82]
    \mathbin{\otimes}1.
  \end{equation*}%
}

\newcommand{\GraphNprime}[1][1]{%
  \GraphBox[#1]{%
    \node[graph vertex] (n1) at (0,0) {$0$};
    \node[graph vertex] (n2) at (1.55,0) {$1$};
    \node[graph vertex,minimum size=7mm,inner sep=1.5pt]
      (n3) at (3.10,0) {$g\!-\!1$};

    \draw[graph edge] (n1) -- (n2);
    \draw[graph edge] (n2) -- (n3);

    \draw[graph marked points] (n1) -- (-0.50,0.55);
    \draw[graph marked points] (n2) -- (1.55,0.72);
    \draw[graph marked points] (n3) -- (3.60,0.55);

    \node[graph label,above] at (-0.59,0.65)
      {$X$};

    \node[graph label,above] at (1.55,0.86)
      {$(X\cup Y)^{c}$};

    \node[graph label,above] at (3.69,0.65)
      {$Y$};

    \node[graph label,above] at (1.14,0.02)
      {$p$};

    \node[graph label,above] at (1.97,0.02)
      {$q$};

    \node[graph label] at (1.55,-0.58)
      {$\omega_{B\cup p}$};
  }%
}

\newcommand{\GraphZ}[1][1]{%
  \GraphBox[#1]{%
    \node[graph vertex] (z1) at (0,0) {$1$};
    \node[graph vertex,minimum size=7mm,inner sep=1.5pt]
      (z2) at (2.15,0) {$g\!-\!1$};

    \draw[graph edge] (z1) -- (z2);

    \draw[graph marked points] (z1) -- (0,0.72);
    \draw[graph marked points] (z2) -- (2.65,0.55);

    \node[graph label,above] at (0,0.86)
      {$A$};

    \node[graph label,above] at (2.74,0.65)
      {$A^{c}$};

    \node[graph label,above] at (0.46,0.02)
      {$p$};

    \node[graph label] at (0,-0.58)
      {$\omega_{B\cup p}$};
  }%
}

\newcommand{\Graphreloneaa}[1][1]{%
  \GraphBox[#1]{%
    \node[graph vertex] (r1) at (0,0) {$0$};
    \node[graph vertex] (r2) at (1.55,0) {$1$};

    \draw[graph edge] (r1) -- (r2);

    \draw[graph marked points] (r1) -- (-0.55,0.50);
    \draw[graph marked points] (r1) -- (-0.55,-0.50);
    \draw[graph marked points] (r2) -- (1.90,0.62);

    \node[graph label,above] at (-0.65,0.59)
      {$x_1$};

    \node[graph label,below] at (-0.65,-0.59)
      {$x_2$};

    \node[graph label,above] at (1.97,0.74)
      {$B$};

    \node[graph label,above] at (1.14,0.02)
      {$p$};

    \node[graph label] at (1.55,-0.58)
      {$\omega_{B\cup p}$};
  }%
}

\newcommand{\Graphrelonebb}[1][1]{%
  \GraphBox[#1]{%
    \node[graph vertex] (r1) at (0,0) {$0$};
    \node[graph vertex] (r2) at (1.55,0) {$1$};

    \draw[graph edge] (r1) -- (r2);

    \draw[graph marked points] (r1) -- (-0.55,0.50);
    \draw[graph marked points] (r1) -- (-0.55,-0.50);
    \draw[graph marked points] (r2) -- (1.90,0.62);

    \node[graph label,above] at (-0.65,0.59)
      {$b_1$};

    \node[graph label,below] at (-0.65,-0.59)
      {$x_2$};

    \node[graph label,above] at (1.97,0.74)
      {$B\setminus b_1\cup x_1$};

    \node[graph label,above] at (1.14,0.02)
      {$p$};

    \node[graph label] at (1.55,-0.58)
      {$\omega_{B\setminus b_1\cup x_1\cup p}$};
  }%
}

\newcommand{\Graphrelonecc}[1][1]{%
  \GraphBox[#1]{%
    \node[graph vertex] (r1) at (0,0) {$0$};
    \node[graph vertex] (r2) at (1.55,0) {$1$};

    \draw[graph edge] (r1) -- (r2);

    \draw[graph marked points] (r1) -- (-0.55,0.50);
    \draw[graph marked points] (r1) -- (-0.55,-0.50);
    \draw[graph marked points] (r2) -- (1.90,0.62);

    \node[graph label,above] at (-0.65,0.59)
      {$x_1$};

    \node[graph label,below] at (-0.65,-0.59)
      {$b_1$};

    \node[graph label,above] at (1.97,0.74)
      {$B\setminus b_1\cup x_2$};

    \node[graph label,above] at (1.14,0.02)
      {$p$};

    \node[graph label] at (1.55,-0.58)
      {$\omega_{B\setminus b_1\cup x_2\cup p}$};
  }%
}

\tikzset{
  graph vertex small/.style={
    graph vertex,
    minimum size=4.2mm
  },
  graph half edge/.style={
    graph label,
    font=\fontsize{7}{8}\selectfont
  }
}

\newcommand{\Graphreltwoaa}[1][1]{%
  \GraphBox[#1]{%
    \node[graph vertex small] (r1) at (0,0) {$0$};
    \node[graph vertex small] (r2) at (0.85,0) {$1$};

    \draw[graph edge] (r1) -- (r2);

    \draw[graph marked points] (r1) -- (-0.38,0.34);
    \draw[graph marked points] (r1) -- (-0.38,-0.34);
    \draw[graph marked points] (r2) -- (0.99,0.48);
    \draw[graph marked points] (r2) -- (1.37,0);

    \node[graph label,above] at (-0.48,0.43)
      {$x_1$};

    \node[graph label,below] at (-0.48,-0.43)
      {$x_2$};

    \node[graph label,above] at (1.03,0.61)
      {$B$};

    \node[graph label,anchor=west] at (1.51,0)
      {$q$};

    \node[graph half edge,above] at (0.51,0.02)
      {$p$};

    \node[graph label] at (0.85,-0.46)
      {$\omega_{B\cup p}$};
  }%
}

\newcommand{\Graphreltwobb}[1][1]{%
  \GraphBox[#1]{%
    \node[graph vertex small] (r1) at (0,0) {$0$};
    \node[graph vertex small] (r2) at (0.85,0) {$1$};

    \draw[graph edge] (r1) -- (r2);

    \draw[graph marked points] (r1) -- (-0.38,0.34);
    \draw[graph marked points] (r1) -- (-0.38,-0.34);
    \draw[graph marked points] (r2) -- (0.99,0.48);
    \draw[graph marked points] (r2) -- (1.37,0);

    \node[graph label,above] at (-0.48,0.43)
      {$b_1$};

    \node[graph label,below] at (-0.48,-0.43)
      {$x_2$};

    \node[graph label,above] at (1.03,0.61)
      {$B\setminus b_1\cup x_1$};

    \node[graph label,anchor=west] at (1.51,0)
      {$q$};

    \node[graph half edge,above] at (0.51,0.02)
      {$p$};

    \node[graph label] at (0.85,-0.46)
      {$\omega_{B\setminus b_1\cup x_1\cup p}$};
  }%
}

\newcommand{\Graphreltwocc}[1][1]{%
  \GraphBox[#1]{%
    \node[graph vertex small] (r1) at (0,0) {$0$};
    \node[graph vertex small] (r2) at (0.85,0) {$1$};

    \draw[graph edge] (r1) -- (r2);

    \draw[graph marked points] (r1) -- (-0.38,0.34);
    \draw[graph marked points] (r1) -- (-0.38,-0.34);
    \draw[graph marked points] (r2) -- (0.99,0.48);
    \draw[graph marked points] (r2) -- (1.37,0);

    \node[graph label,above] at (-0.48,0.43)
      {$x_1$};

    \node[graph label,below] at (-0.48,-0.43)
      {$b_1$};

    \node[graph label,above] at (1.03,0.61)
      {$B\setminus b_1\cup x_2$};

    \node[graph label,anchor=west] at (1.51,0)
      {$q$};

    \node[graph half edge,above] at (0.51,0.02)
      {$p$};

    \node[graph label] at (0.85,-0.46)
      {$\omega_{B\setminus b_1\cup x_2\cup p}$};
  }%
}

\newcommand{\Graphreltwodd}[1][1]{%
  \GraphBox[#1]{%
    \node[graph vertex small] (r1) at (0,0) {$0$};
    \node[graph vertex small] (r2) at (0.85,0) {$1$};

    \draw[graph edge] (r1) -- (r2);

    \draw[graph marked points] (r1) -- (-0.38,0.34);
    \draw[graph marked points] (r1) -- (0,0.52);
    \draw[graph marked points] (r1) -- (-0.38,-0.34);
    \draw[graph marked points] (r2) -- (0.99,0.48);

    \node[graph label,above] at (-0.48,0.43)
      {$b_1$};

    \node[graph label,above] at (0,0.66)
      {$q$};

    \node[graph label,below] at (-0.48,-0.43)
      {$x_2$};

    \node[graph label,above] at (1.03,0.61)
      {$B\setminus b_1\cup x_1$};

    \node[graph half edge,above] at (0.51,0.02)
      {$p$};

    \node[graph label] at (0.85,-0.46)
      {$\omega_{B\setminus b_1\cup x_1\cup p}$};
  }%
}

\newcommand{\Graphreltwoee}[1][1]{%
  \GraphBox[#1]{%
    \node[graph vertex small] (r1) at (0,0) {$0$};
    \node[graph vertex small] (r2) at (0.85,0) {$1$};

    \draw[graph edge] (r1) -- (r2);

    \draw[graph marked points] (r1) -- (-0.38,0.34);
    \draw[graph marked points] (r1) -- (0,0.52);
    \draw[graph marked points] (r1) -- (-0.38,-0.34);
    \draw[graph marked points] (r2) -- (0.99,0.48);

    \node[graph label,above] at (-0.48,0.43)
      {$x_1$};

    \node[graph label,above] at (0,0.66)
      {$q$};

    \node[graph label,below] at (-0.48,-0.43)
      {$b_1$};

    \node[graph label,above] at (1.03,0.61)
      {$B\setminus b_1\cup x_2$};

    \node[graph half edge,above] at (0.51,0.02)
      {$p$};

    \node[graph label] at (0.85,-0.46)
      {$\omega_{B\setminus b_1\cup x_2\cup p}$};
  }%
}

\newcommand{\Graphreltwoff}[1][1]{%
  \GraphBox[#1]{%
    \node[graph vertex small] (r1) at (0,0) {$0$};
    \node[graph vertex small] (r2) at (0.85,0) {$1$};

    \draw[graph edge] (r1) -- (r2);

    \draw[graph marked points] (r1) -- (-0.38,0.34);
    \draw[graph marked points] (r1) -- (0,0.52);
    \draw[graph marked points] (r1) -- (-0.38,-0.34);
    \draw[graph marked points] (r2) -- (0.85,0.52);

    \node[graph label,above] at (-0.48,0.43)
      {$x_1$};

    \node[graph label,above] at (0,0.66)
      {$q$};

    \node[graph label,below] at (-0.48,-0.43)
      {$x_2$};

    \node[graph label,above] at (0.85,0.66)
      {$B$};

    \node[graph half edge,above] at (0.51,0.02)
      {$p$};

    \node[graph label] at (0.85,-0.46)
      {$\omega_{B\cup p}$};
  }%
}

\newcommand{\GraphrelMzeroaa}[1][1]{%
  \GraphBox[#1]{%
    \node[graph vertex] (m1) at (0,0) {$0$};
    \node[graph vertex] (m2) at (1.90,0) {$1$};

    \draw[graph edge] (m1) -- (m2);

    \draw[graph marked points] (m1) -- (-0.55,0.50);
    \draw[graph marked points] (m1) -- (-0.55,-0.50);
    \draw[graph marked points] (m2) -- (1.90,0.72);
    \draw[graph marked points] (m2) -- (2.62,0);

    \node[graph label,above] at (-0.65,0.59)
      {$x_1$};

    \node[graph label,below] at (-0.65,-0.59)
      {$x_2$};

    \node[graph label,above] at (1.90,0.86)
      {$B'$};

    \node[graph label,anchor=west] at (2.76,0)
      {$q$};

    \node[graph label,above] at (1.49,0.02)
      {$p$};

    \node[graph label] at (1.90,-0.58)
      {$\omega_{B'\cup p\cup q}$};
  }%
}

\newcommand{\GraphrelMzerobb}[1][1]{%
  \GraphBox[#1]{%
    \node[graph vertex] (m1) at (0,0) {$1$};
    \node[graph vertex] (m2) at (1.90,0) {$0$};

    \draw[graph edge] (m1) -- (m2);

    \draw[graph marked points] (m1) -- (-0.35,0.63);
    \draw[graph marked points] (m1) -- (0.35,0.63);
    \draw[graph marked points] (m2) -- (1.90,0.72);
    \draw[graph marked points] (m2) -- (2.62,0);

    \node[graph label,above] at (-0.42,0.75)
      {$B'$};

    \node[graph label,above] at (0.42,0.75)
      {$x_2$};

    \node[graph label,above] at (1.90,0.86)
      {$x_1$};

    \node[graph label,anchor=west] at (2.76,0)
      {$q$};

    \node[graph label,above] at (0.46,0.02)
      {$p$};

    \node[graph label] at (0,-0.58)
      {$\omega_{B'\cup x_2\cup p}$};
  }%
}

\newcommand{\GraphrelMzerocc}[1][1]{%
  \GraphBox[#1]{%
    \node[graph vertex] (m1) at (0,0) {$1$};
    \node[graph vertex] (m2) at (1.90,0) {$0$};

    \draw[graph edge] (m1) -- (m2);

    \draw[graph marked points] (m1) -- (-0.35,0.63);
    \draw[graph marked points] (m1) -- (0.35,0.63);
    \draw[graph marked points] (m2) -- (1.90,0.72);
    \draw[graph marked points] (m2) -- (2.62,0);

    \node[graph label,above] at (-0.42,0.75)
      {$B'$};

    \node[graph label,above] at (0.42,0.75)
      {$x_1$};

    \node[graph label,above] at (1.90,0.86)
      {$x_2$};

    \node[graph label,anchor=west] at (2.76,0)
      {$q$};

    \node[graph label,above] at (0.46,0.02)
      {$p$};

    \node[graph label] at (0,-0.58)
      {$\omega_{B'\cup x_1\cup p}$};
  }%
}

\newcommand{\GraphClass}[1]{\makebox[0pt]{\ensuremath{#1}}}

\tikzset{
  graph vertex tiny/.style={
    graph vertex,
    minimum size=3.4mm,
    font=\fontsize{6}{7}\selectfont
  },
  graph label tiny/.style={
    graph label,
    font=\fontsize{6}{7}\selectfont
  },
  graph half edge tiny/.style={
    graph label,
    font=\fontsize{5}{6}\selectfont
  }
}

\newcommand{\GraphMijaa}[1][1]{%
  \GraphBox[#1]{%
    \node[graph vertex tiny] (v1) at (0,0) {$1$};
    \node[graph vertex tiny] (v2) at (0.58,0) {$0$};
    \node[graph vertex tiny] (v3) at (1.16,0) {$1$};

    \draw[graph edge] (v1) -- (v2);
    \draw[graph edge] (v2) -- (v3);

    \draw[graph marked points] (v1) -- (0,0.38);
    \draw[graph marked points] (v2) -- (0.40,0.38);
    \draw[graph marked points] (v2) -- (0.76,0.38);

    \node[graph label tiny,above] at (0,0.48)
      {$X$};

    \node[graph label tiny,above] at (0.36,0.47)
      {$i$};

    \node[graph label tiny,above] at (0.80,0.47)
      {$j$};

    \node[graph half edge tiny,above] at (0.29,0.02)
      {$p$};

    \node[graph label tiny] at (0,-0.36)
      {\GraphClass{\omega_{X\cup\{p\}}}};
  }%
}

\newcommand{\GraphMijbb}[1][1]{%
  \GraphBox[#1]{%
    \node[graph vertex tiny] (v1) at (0,0) {$1$};
    \node[graph vertex tiny] (v2) at (0.58,0) {$0$};

    \draw[graph edge] (v1) -- (v2);
    \draw[graph edge] (v2)
      to[out=135,in=45,looseness=4.5] (v2);

    \draw[graph marked points] (v1) -- (0,0.38);
    \draw[graph marked points] (v2) -- (0.82,0.12);
    \draw[graph marked points] (v2) -- (0.82,-0.30);

    \node[graph label tiny,above] at (0,0.48)
      {$X$};

    \node[graph label tiny,above] at (0.90,0.15)
      {$i$};

    \node[graph label tiny,below] at (0.88,-0.38)
      {$j$};

    \node[graph half edge tiny,above] at (0.29,0.02)
      {$p$};

    \node[graph label tiny] at (0,-0.36)
      {\GraphClass{\omega_{X\cup\{p\}}}};
  }%
}

\newcommand{\GraphMijcc}[1][1]{%
  \GraphBox[#1]{%
    \node[graph vertex tiny] (v1) at (0,0) {$1$};
    \node[graph vertex tiny] (v2) at (0.58,0) {$0$};
    \node[graph vertex tiny] (v3) at (1.16,0) {$1$};

    \draw[graph edge] (v1) -- (v2);
    \draw[graph edge] (v2) -- (v3);

    \draw[graph marked points] (v1) -- (0,0.38);
    \draw[graph marked points] (v2) -- (0.66,0.38);
    \draw[graph marked points] (v3) -- (1.24,0.34);

    \node[graph label tiny,above] at (0,0.48)
      {$X$};

    \node[graph label tiny,above] at (0.68,0.48)
      {$i$};

    \node[graph label tiny,above] at (1.27,0.43)
      {$j$};

    \node[graph half edge tiny,above] at (0.29,0.02)
      {$p$};

    \node[graph label tiny] at (0,-0.36)
      {\GraphClass{\omega_{X\cup\{p\}}}};
  }%
}

\newcommand{\GraphMijdd}[1][1]{%
  \GraphBox[#1]{%
    \node[graph vertex tiny] (v1) at (0,0) {$1$};
    \node[graph vertex tiny] (v2) at (0.58,0) {$0$};
    \node[graph vertex tiny] (v3) at (1.16,0) {$1$};

    \draw[graph edge] (v1) -- (v2);
    \draw[graph edge] (v2) -- (v3);

    \draw[graph marked points] (v1) -- (0,0.38);
    \draw[graph marked points] (v2) -- (0.66,0.38);
    \draw[graph marked points] (v3) -- (1.24,0.34);

    \node[graph label tiny,above] at (0,0.48)
      {$X$};

    \node[graph label tiny,above] at (0.68,0.48)
      {$j$};

    \node[graph label tiny,above] at (1.27,0.43)
      {$i$};

    \node[graph half edge tiny,above] at (0.29,0.02)
      {$p$};

    \node[graph label tiny] at (0,-0.36)
      {\GraphClass{\omega_{X\cup\{p\}}}};
  }%
}

\newcommand{\GraphMijee}[1][1]{%
  \GraphBox[#1]{%
    \node[graph vertex tiny] (v1) at (0,0) {$1$};
    \node[graph vertex tiny] (v2) at (0.58,0) {$1$};
    \node[graph vertex tiny] (v3) at (1.16,0) {$0$};

    \draw[graph edge] (v1) -- (v2);
    \draw[graph edge] (v2) -- (v3);

    \draw[graph marked points] (v1) -- (0,0.38);
    \draw[graph marked points] (v3) -- (1.32,0.30);
    \draw[graph marked points] (v3) -- (1.32,-0.30);

    \node[graph label tiny,above] at (0,0.48)
      {$X$};

    \node[graph label tiny,above] at (1.38,0.38)
      {$i$};

    \node[graph label tiny,below] at (1.38,-0.38)
      {$j$};

    \node[graph half edge tiny,above] at (0.29,0.02)
      {$p$};

    \node[graph label tiny] at (0,-0.36)
      {\GraphClass{\omega_{X\cup\{p\}}}};
  }%
}

\newcommand{\GraphMijff}[1][1]{%
  \GraphBox[#1]{%
    \node[graph vertex tiny] (v1) at (0,0) {$1$};
    \node[graph vertex tiny] (v2) at (0.90,0) {$0$};

    \draw[graph edge] (v1)
      to[out=45,in=135,looseness=1.25] (v2);

    \draw[graph edge] (v1)
      to[out=-45,in=-135,looseness=1.25] (v2);

    \draw[graph marked points] (v1) -- (0,0.38);
    \draw[graph marked points] (v2) -- (1.14,0.30);
    \draw[graph marked points] (v2) -- (1.14,-0.30);

    \node[graph label tiny,above] at (0,0.48)
      {$X$};

    \node[graph label tiny,above] at (1.20,0.38)
      {$i$};

    \node[graph label tiny,below] at (1.20,-0.38)
      {$j$};

    \node[graph half edge tiny] at (0.34,0.12)
      {$p$};

    \node[graph half edge tiny] at (0.34,-0.12)
      {$q$};

    \node[graph label tiny] at (0,-0.40)
      {\GraphClass{\omega^{S_2}}};
  }%
}

\author[A.~Bradford]{Amy Bradford}

\author[A.~Iribar L\'opez]{Aitor Iribar L\'opez}

\author[A.~Kretschmer]{Andreas Kretschmer}

\author[H.~Larson]{Hannah Larson}

\title{The fifteenth cohomology group of moduli spaces of stable curves}

\begin{document}

\begin{abstract}
The new framework of semi-tautological extensions has allowed for explicit understanding of low-degree odd cohomology groups of moduli spaces of stable curves.
Let $S^*_\omega(\Mb_{g,n}) \subset H^*(\Mb_{g,n})$ denote the semi-tautological extension generated by $H^{11}(\Mb_{1,11})$, that is, the smallest system of subrings  of $H^*(\Mb_{g,n})$ containing $H^{11}(\Mb_{1,11})$ and the tautological subrings, which is also
closed under pullbacks along forgetful maps and pushforwards along gluing maps. 
Recently, it has been shown that
$H^{11}(\Mb_{g,n}) = S^{11}_\omega(\Mb_{g,n})$ \cite{CLP23eleventh} and $H^{13}(\Mb_{g,n}) = S^{13}_\omega(\Mb_{g,n})$ \cite{CLPW24polynomial} for all $g$ and $n$. Meanwhile, $H^{15}(\Mb_{1,n}) \neq S_\omega^{15}(\Mb_{1,n})$ for all $n \geq 15$. Nevertheless,
we prove that $H^{15}(\Mb_{2,n}) = S^{15}_\omega(\Mb_{2,n})$ for all $n$. Moreover, we determine all relations in $S_\omega^{15}(\Mb_{g,n})$ for all $g \geq 2$ and give explicit bases for this vector space in terms of decorated graphs. Conditional upon the conjectured dimension of $H^{15}(\Mb_{3,14})$, we also prove that $H^{15}(\Mb_{3,n}) = S^{15}_\omega(\Mb_{3,n})$ for all $n$.
As an application, we prove $H^*(\Mb_{g,n}) = RH^*(\Mb_{g,n})$ for several new small pairs $(g, n)$.
\end{abstract}

\maketitle

\section{Introduction}

A longstanding technique for studying the cohomology of $\Mbar_{g,n}$ has been to consider the system of \textit{tautological subrings} $RH^*(\Mbar_{g,n}) \subseteq H^*(\Mbar_{g,n})$, defined as the smallest system of subrings that is closed under pushforward and pullback along the natural gluing and forgetful morphisms.
There is a known systematic collection of relations among tautological classes, known as Pixton's relations \cite{PPZ}. Conjecturally, these known relations are all relations, thus providing a robust understanding of the cohomology groups that are entirely tautological.
Currently, it is known that $H^k(\Mbar_{g,n}) = RH^k(\Mbar_{g,n})$ for all $g$ and $n$ when $k \leq 5$, when $k = 7$ or $9$, or when $k = 6$ and $g \geq 10$ \cite{AC98, BFP24, CLP24extensions}. Additionally, it is known that the entire cohomology ring of $\Mb_{g,n}$ is tautological for $(g, n)$ satisfying $g = 0$ \cite{Keel}; $g = 1$ and $n \leq 10$ \cite{B}; $g = 2$ and $n \leq 9$; or $3 \leq g \leq 7$ and $2g + n\leq 14$
\cite{CL23chowMgbar}.

Meanwhile, non-tautological classes arise in $H^{11}(\Mb_{1,11})$. Here, one has
\[H^{11}(\Mbar_{1,11}) \otimes \mathbb{C} = H^{11,0}(\Mb_{1,11}) \oplus H^{0,11}(\Mb_{1,11}),\]
where $H^{11,0}(\Mb_{1,11})$ is $1$-dimensional, generated by the holomorphic form $\omega$ corresponding to the weight $12$ cusp form for $\mathrm{SL}_2(\mathbb{Z})$ \cite[p.\ 306]{FPnontautological}.
In \cite{CLP23eleventh}, Canning, Payne, and the fourth author showed that $H^{11}(\Mbar_{g,n})$ is nonzero if and only if $g = 1$ and $n \geq 11$.
In these cases, $H^{11}(\Mbar_{1,n})$ is generated by the pullbacks of the vector spaces $H^{11}(\Mbar_{1,11})$ as one ranges over the forgetful maps that forget all but $11$ of the markings.
In order to make sense of these new non-tautological classes, in
\cite{CLP24extensions}, the same authors introduce \textit{semi-tautological extensions} (STEs).
An STE is a collection of subrings $S^*(\Mbar_{g,n})$ of $H^*(\Mbar_{g,n}) $ which contain $RH^*(\Mbar_{g,n})$ and are closed under pushforwards along gluing maps and under pullbacks along forgetful maps.
A \textit{finitely generated} STE is the smallest STE containing a given finite collection of elements (or equivalently finite-dimensional $\qq$-vector subspaces) of $\coprod_{g,n} H^*(\Mbar_{g,n})$.
For example, the STE generated by $H^{11}(\Mbar_{1,11})$, denoted $S_\omega^*$, contains $H^{11}(\Mbar_{g,n})$ for all $g$ and $n$.
Although it need not be true for all STEs, $S^*_\omega$ is in fact closed under pushforwards along forgetful maps and under pullbacks along gluing maps \cite[Prop.\ 2.4]{CLP24extensions}.

Finitely generated STEs may be studied via combinatorial techniques on \textit{decorated (strata) graphs}.
For example, additive generators for the tautological ring are given by decorated graphs, whose underlying graphs correspond to boundary strata of $\Mbar_{g,n}$, and whose decorations at each of the vertices are monomials in the tautological $\kappa$ and $\psi$ classes.
Meanwhile, additive generators for the STE $S^*_\omega$ are given by decorated graphs where, in addition to the usual tautological decorations, one is permitted to decorate genus $1$ vertices with any class in $H^{11}(\Mb_{1,n})$.
Using this new kind of decorated graph,
in \cite{CLPW24polynomial}, 
the authors prove that $H^{13}(\Mb_{g,n}) = S^{13}_\omega(\Mb_{g,n})$ for all $g$ and $n$, and give generators and relations for the latter.
In this paper, we use the decorated graphs in this STE to study $H^{15}(\Mb_{g,n})$.

 In genus $1$, it is known that $H^{15,0}(\Mb_{1,15}) \neq 0$, corresponding to the weight $16$ cusp form for $\mathrm{SL}_2(\mathbb{Z})$ \cite{Getzler98semiclassical}.
However, examining the decorated graph generators for $S^*_\omega$, one readily sees that in degree $15$ they only produce classes of Hodge type $(13,2)$ and $(2,13)$. More precisely, writing $\mathsf{L}\coloneq H^2(\mathbb{P}^1)$ and $\mathsf{S}_{12}\coloneq H^{11}(\Mbar_{1,11})$ for the Hodge structures or Galois representations, one sees that $S^{15}_\omega(\Mb_{g,n}) = \bigoplus \mathsf{L}^2 \mathsf{S}_{12}$ for all $g$ and $n$. 
In particular,  in genus $1$, we have $H^{15}(\Mb_{1,n}) \neq S^{15}_{\omega}(\Mb_{1,n})$ for all $n \geq 15$.

Nevertheless, for $g \geq 2$, the results of \cite{CLP24extensions} show that $H^{15}(\Mb_{g,n})^\mathrm{ss} = \bigoplus \mathsf{L}^2\mathsf{S}_{12}$, suggesting that $H^{15}(\Mb_{g,n})$ should lie in the STE for all $n$ and all $g \geq 2$. Using an intricate combinatorial argument with decorated graphs, we prove this conjecture in genus $2$.

\begin{theorem}\label{theorem:main}
For $g = 2$ and any $n$, we have $H^{15}(\Mbar_{2,n}) = S^{15}_\omega(\Mb_{2,n})$.
\end{theorem}

Meanwhile, for all $g \geq 2$ and any $n$, 
our second theorem 
determines the structure of $S_\omega^{15}(\Mb_{g,n})$
with generators and relations.
First, in Figure \ref{fig:gentable}, 
we provide a list of decorated graph generators. We think of these graphs as representing classes in $S_\omega^{13,2}(\Mb_{g,n})$, but as in \cite[Section 4.1]{CLPW24polynomial}, one can also interpret these decorated graphs as generators for the $\qq$-vector space $\mathrm{Hom}(\mathsf{L}^2\mathsf{S}_{12}, S_\omega^{15}(\Mb_{g,n}))$ and one has 
\[S^{15}_{\omega}(\Mb_{g,n}) = \mathrm{Hom}(\mathsf{L}^2\mathsf{S}_{12}, S^{15}_\omega(\Mb_{g,n})) \otimes_{\qq} \mathsf{L}^2\mathsf{S}_{12}.\] 
The relations among these decorated graph generators come in three flavors.
\begin{enumerate}
    \item Trivial relations come from different ways of drawing and labeling a graph, as explained in Section \ref{sec:triv}.
    \item Pushed forward relations come from pushing forward a known relation in lower cohomological degree along some gluing map, as explained in Section \ref{sec:rels_all_g}.
    \item Pulled back relations are relations not in the span of the above that come from pulling back relations that arise on a moduli space of curves with a small number of marked points. The initial instances of these relations were discovered by pairing generators against the Poincar\'e dual group. It is an interesting fact that these relations only occur when $g = 2$ and all relations are pulled back from moduli spaces with $n \leq 12$ markings. We describe these relations in Section \ref{sec:g2rel}. 
\end{enumerate}

\begin{theorem} \label{thm:rels}
(a) If $g = 2$, then the trivial relations \eqref{triv1}--\eqref{triv2}, the pushed forward relations \eqref{push1}--\eqref{push4}, \eqref{push5}, and the pulled back relations \eqref{rel10pull}--\eqref{rel12pull} generate all relations among the generators of $S_\omega^{13,2}(\Mb_{g,n})$ in Figure \ref{fig:gentable}.

(b) If $g \geq 3$, then the trivial relations \eqref{triv1}--\eqref{triv2} and pushed forward relations \eqref{push1}--\eqref{push4}, \eqref{push5} generate all relations among the generators of $S_\omega^{13,2}(\Mb_{g,n})$ in Figure \ref{fig:gentable}. 
\end{theorem}

Using our relations to eliminate unnecessary generators, we provide an explicit basis for $S_\omega^{13,2}(\Mb_{g,n})$ in terms of decorated graphs, namely those in Figure \ref{fig:gentable} which also satisfy the restrictions (R1)--(R10) in the third column. From this, it is a straightforward combinatorial exercise to count the dimension of $S_\omega^{13,2}(\Mb_{g,n})$. In Section \ref{sec:dimcounts}, we provide explicit formulas for the dimension in genus $g = 2, 3$.
We have checked these formulas against the
point counting data of Bergstr\"om and Faber \cite{BergstromData} in genus $2$ and $3$ with low $n$, and in every instance, $\dim S^{13,2}_\omega(\Mb_{g,n})$ agrees with the conjectural dimension of $H^{13,2}(\Mb_{g,n})$.
In genus $3$ with $14$ markings, the conjectural dimension is
\begin{equation} \label{bf} 
\dim H^{13,2}(\Mb_{3,14}) = 657657.
\end{equation}

\begin{theorem} \label{thm:g3}
Conditional upon the conjectural assertion of the dimension in \eqref{bf} when $n = 14$, we have $H^{15}(\Mb_{3,n}) = S^{15}_\omega(\Mb_{3,n})$ for all $n$.
\end{theorem}

\begin{remark}
In genus $2$, the data in \cite{BergstromData} are unconditional.
In genus $3$, the predictions of Bergstr\"om--Faber are known unconditionally for $n \leq 11$ by \cite[Theorem 1.9]{CLP24extensions}.    
\end{remark}

Finally, building upon the results of \cite{CL23chowMgbar,CLP24extensions}, we prove that the entire cohomology ring of $\Mb_{g,n}$ is tautological for several new values of $g$ and $n$.
In these cases, we establish that all non-tautological cohomology classes are supported on the boundary, and furthermore lie in $S^*_\omega$. We then use representation theory and symmetry of decorated graphs to show that the non-tautological classes in $S^*_\omega$ must vanish in this range.

\begin{theorem} \label{thm:rh}
We have $H^*(\Mb_{g,n}) = RH^*(\Mb_{g,n})$ for the following pairs $(g,n)$:
\begin{gather} \label{goodpairs} (3,9) \qquad (4, 7) \qquad (4, 8) \qquad (5,5) \qquad (5, 6) \qquad (5,7) \\
(6,3) \qquad (6, 4) \qquad (6, 5) \qquad (7, 1) \qquad (7, 2) \qquad (7, 3) \qquad (8, 0) \notag
\end{gather}
\end{theorem}

\begin{remark}
As noted in the survey article \cite[Proposition 3.9]{Lsurvey}, the cases
$(g,n) = (3,9), (4,7), (5,5), (6,3),$ and $(7,1)$ follow readily from results already in the literature. The proofs for the other cases above utilize our new results on $S_\omega^{15}(\Mb_{g,n})$. 
\end{remark}

\subsection{Outline} In Section \ref{section:gens}, we write down generators for $S_\omega^{13,2}(\Mbar_{g,n})$, as well as a collection of relations between them. In Section \ref{section:genus2}, we show that these relations are complete when $g=2$ and $n$ is arbitrary (Theorem \ref{thm:rels} (a)) by studying the pullback along the gluing maps
$$
\Mbar_{2,n-k+1}\times \Mbar_{0,k+1} \to \Mbar_{2,n}\,, \quad \Mbar_{1,n-k+1}\times \Mbar_{1,k+1} \to \Mbar_{2,n}\,.
$$
This lets us show that $H^{15}(\Mbar_{2,14}) = S_\omega^{15}(\Mbar_{2,14})$ by comparing the dimension of the semi-tautological ring with the calculation
$$
\dim H^{13,2}(\Mbar_{2,14}) = 234234
$$
from \cite{BergstromData}. In Section \ref{section:ReductionsTo14Markings}, we explain how from here, one can show that $H^{15}(\Mbar_{2,n}) = S_\omega^{15}(\Mbar_{2,n})$ for all $n$ (Theorem \ref{theorem:main}). In Section \ref{section:anyg}, we show that the relations between the generators of the STE in degree $15$ are complete (Theorem \ref{thm:rels} (b)) for all $g,n$ with $g \geq 3$ by studying the pullback along the gluing map
$$
\Mbar_{g-1,n+2} \to \Mbar_{g,n}\,.
$$
In Section \ref{sec:dimcounts}, we obtain a formula for $\dim(S_\omega^{13,2}(\Mbar_{g,n}))$ for $g=2,3$ and all $n$, matching \eqref{bf} for $(g,n)=(3,14)$. The reduction step explained in Section \ref{section:ReductionsTo14Markings} also applies in genus $3$, and we deduce Theorem \ref{thm:g3} from this at the end of Section \ref{sec:dimcounts}. In Section \ref{section:cohomology-is-taut} we use results from \cite{Lsurvey} to reduce Theorem \ref{thm:rh} to the cases $(g,n) = (3,9),(4,8), (5,7)$. In these cases, it is known by \cite{CanningRarelyGorenstein} that the even cohomology is tautological, and by \cite{CL23chowMgbar} that $\M_{g,n}$ has the CKgP, so all the odd cohomology comes from the boundary. We show that any odd class from the boundary vanishes after pushforward in these cases by analyzing all the possibilities.

\subsection{Acknowledgements}
Amy Q.\ Li made substantial contributions to this project, especially towards the proof of Theorem \ref{theorem:main}. We are grateful for the input on how to organize the pullbacks of generators of the STE to boundary divisors, which is one of the main difficulties of the proof. Amy also contributed significantly to the writing of earlier drafts of the paper.

This project began at the Graduate Student and Postdoc Bootcamp preceding the
2025 Summer Research Institute in Algebraic Geometry at Colorado State University. We would like to thank the organizers of the Bootcamp, Izzet Coskun, Emanuele Macr\`{i}, Alexander Perry, Kevin Tucker, and Isabel Vogt, as well as the other members of our mentoring group, Daksh Aggarwal,
Erin Dawson,
Federico Moretti,
Ignacio Rojas,
Neelarnab Raha, and Sean Edwards, for bringing us together and fostering a collaborative environment for learning and research. We would also like to thank Jonas Bergstr\"om and Carel Faber for generously sharing their data on the cohomology of $\Mb_{2,n}$. The second author also thanks Samir Canning for helpful discussions about the cohomology of the moduli space of curves during their time at ETH Z\"urich.

A.B. was supported by the NSF Graduate
Research Fellowship DGE-2137420. A.I.L. was supported by SNF-200020-219369. A.K. was supported by the ERC Advanced Grant SYZYGY (No. 834172, PI: Gavril Farkas).
H.L.\ was partially supported by the Clay Research Fellowship and NSF grant DMS-2544727.

\subsection{AI disclosure} We used Large Language Models to find typos in a draft of this document and to produce code for the stable graphs displayed in the paper. OpenAI’s ChatGPT was used to assist with editing the wording of Section 7, particularly Lemma \ref{lem:S2-invariants-H11} and Corollary \ref{cor:graph-averaging-vanishing}. The authors reviewed the resulting text and take full responsibility for its content.

\section{Reduction of Theorems \ref{theorem:main} and \ref{thm:g3} to \texorpdfstring{$n = 14$}{n=14} and strategy}\label{section:ReductionsTo14Markings}

In this section, we explain how existing results reduce Theorem \ref{theorem:main} and Theorem \ref{thm:g3} to the cases when $n = 14$. Given a smooth Deligne--Mumford stack $X$, we write
$W_kH^i(X)$ for the weight $k$ subspace of Deligne's weight filtration on the cohomology of $X$.
Recall that for any $k$, there is a right exact sequence 
\begin{equation} \label{wtknorm}
H^{k-2}(\widetilde{\partial \M_{g,n}}) \longrightarrow H^{k}(\Mbar_{g,n}) \longrightarrow W_{k} H^k(\M_{g,n}) \longrightarrow 0,
\end{equation}
where $\widetilde{\partial \M_{g,n}}$ denotes the normalization of the boundary.
The components of $\widetilde{\partial \M_{g,n}}$ are (finite group quotients of) products of smaller $\Mbar_{g',n'}$.

We are particularly interested in the case $k = 15$ in \eqref{wtknorm}. In this case,
the K\"unneth formula decomposes the left-hand term $H^{13}(\widetilde{\partial \M_{g,n}})$ into a direct sum of (finite group invariants of) tensor products of $H^l(\Mbar_{g',n'})$, where each summand has at least one factor with odd $l$.
By the earlier vanishing results on odd cohomology \cite{AC98,BFP24},
nonzero odd cohomology can only appear when $l = 11$ or $13$. Moreover, in these cases $H^l (\Mbar_{g',n'})$ lies in $S^*_\omega$ by \cite{CLP23eleventh,CLPW24polynomial}.
Thus, $S^{15}_\omega(\Mb_{g,n})$ is the image of the left-hand map in \eqref{wtknorm}, and we see that 
\[H^{15}(\Mb_{g,n}) = S^{15}_\omega(\Mb_{g,n}) \text{ if and only if } W_{15}H^{15}(\M_{g,n}) = 0.\]
In particular, Theorems \ref{theorem:main} and \ref{thm:g3} are equivalent to the claim that $W_{15} H^{15}(\M_{g,n}) = 0$ for every $n$ for $g = 2$ and $3$, respectively.

\subsection{Reduction to \texorpdfstring{$n = 14$}{n=14}} \label{sec:red}
In order to prove the $W_{15} H^{15}(\M_{g,n}) = 0$ vanishing result, we will use induction together with the following lemma.

\begin{lemma}[Lemma 3.1.b of \cite{CLP24extensions}]
\label{lemma:topweight} If $g \geq 2$, $n \geq k$ and $n >0$ then
\[
W_kH^k(\M_{g,n}) = \pi_1^*W_kH^k(\M_{g,n-1}) + \cdots + \pi_n^*W_kH^k(\M_{g,n-1})\, ,
\]
where $\pi_i \colon \M_{g,n} \to \M_{g,n-1}$ is the tautological morphism forgetting the $i$th marking.
\end{lemma}

By the above discussion and lemma, Theorems \ref{theorem:main} and \ref{thm:g3} are reduced to proving, for $g = 2$ and $3$ respectively, that $W_{15}H^{15}(\M_{g,n}) = 0$ for $n \leq 14$, or equivalently that $H^{15}(\Mb_{g,n}) = S^{15}_\omega(\Mb_{g,n})$ for $n \leq 14$.
Since
$$
\pi_*(c_1(\mathcal L)^{k} \cup -) : H^{*}(\Mbar_{g, n+k}) \to H^*(\Mbar_{g,n})
$$
is surjective for $\mathcal L$ a relatively ample line bundle (so in particular $c_1(\mathcal L)$ is tautological) and  $S^*_\omega$ is closed under pushforward \cite[Prop. 2.4.]{CLP24extensions}, it is enough to consider the case $n=14$.

\begin{remark}
In fact, the other cases when $g = 2$ and $n \leq 13$ already follow from results in the literature.
When $n \leq 10$, we have $\dim \Mb_{2,n} \leq 13$ and so $H^{15}(\Mb_{2,n})$ is Poincar\'e dual to $H^{k}(\Mb_{2,n})$ for some odd $k \leq 11$. Thus, $H^{15}(\Mb_{2,n}) = 0$ for $n \leq 10$.
For $n = 11$, \cite[Lemma 3.2]{CLP24extensions} then implies that $H^{15}(\Mb_{2,11}) = S^{15}_\omega(\Mb_{2,11})$.
Next,
\cite[Lemma 7.3]{CLP24extensions} states that $H^{15}(\Mbar_{2,12}) = S^{15}_\omega(\Mb_{2,12})$.
Hence, another application of \cite[Lemma 3.2]{CLP24extensions}  implies $H^{15}(\Mbar_{2,13}) = S^{15}_\omega(\Mb_{2,13})$.
\end{remark}

Thus, to prove Theorems \ref{theorem:main} and \ref{thm:g3}, it remains to show that $H^{15}(\Mbar_{g,14}) = S^{15}_\omega(\Mb_{g,14})$ for $g = 2$ and $3$, respectively. To do so, we will follow a strategy similar to
 \cite[Lemma 7.3]{CLP24extensions}, which we briefly summarize below. First we reduce to proving an equality for a certain piece of the Hodge filtration.

\begin{lemma} \label{132reduction}
Fix some $g \geq 2$ and any $n$. 
We have $H^{15}(\Mb_{g,n}) = S_\omega^{15}(\Mb_{g,n})$ if and only if
$H^{13,2}(\Mb_{g,n}) = S_\omega^{13,2}(\Mb_{g,n})$.
\end{lemma}
\begin{proof}
We have $H^{15}(\Mb_{g,n}) = S_\omega^{15}(\Mb_{g,n})$ if and only if $H^{15}(\Mb_{g,n}) \otimes \cc = S^{15}_\omega(\Mb_{g,n}) \otimes \cc$. Since $g \geq 2$, \cite[Theorem 1.1]{CLP24extensions} tells us that $H^{15}(\Mb_{g,n}) \otimes \cc = H^{13,2}(\Mb_{g,n}) \oplus H^{2,13}(\Mb_{g,n})$. Meanwhile, the decorated graph generators for $S^{15}_\omega(\Mb_{g,n}) \otimes \cc$ involve only decorations of type $(11,0), (0,11), (1,1)$, so
$S^{15}_\omega(\Mb_{g,n})\otimes \cc = S^{13,2}(\Mb_{g,n}) \oplus S^{2,13}(\Mb_{g,n})$. Thus, we have  $H^{15}(\Mb_{g,n}) = S_\omega^{15}(\Mb_{g,n})$ if and only if $H^{13,2}(\Mb_{g,n}) = S_\omega^{13,2}(\Mb_{g,n})$.
\end{proof}

\subsection{Strategy for \texorpdfstring{$g=2$}{g=2} and \texorpdfstring{$n = 14$}{n=14}}\label{subsection:StrategyGenus2Markings14}
In genus $2$, the data in \cite{BergstromData} determining the equivariant Euler characteristics of $\Mb_{2,n}$ is unconditional. In particular, we have that
\[H^{15}(\Mb_{2,14})^{\mathrm{ss}} = (\mathsf{L}^2\mathsf{S}_{12})^{\bigoplus 234234}\]
and so $\dim H^{13,2}(\Mb_{2,14}) = 234234$.
By Lemma \ref{132reduction}, it suffices to show that $\dim S^{13,2}_\omega(\Mb_{2,14}) = 234234$. 

To prove this dimension statement, we first produce a collection of generators for $S_\omega^{13,2}(\Mb_{g,n})$ for all $g$ and $n$.
These generators are given by
decorated graphs, representing classes pushed forward along $H^{12,1}(\widetilde{\partial \M_{g,n}}) \to H^{15}(\Mbar_{g,n})$. In the case $g = 2$ and $n = 14$, there are relations among our generators, which come from pulling back known relations in $H^{13,2}(\Mb_{2,12})$. However, by pulling back these classes to various components of the normalized boundary, we demonstrate that these pulled back relations are (a) independent from each other (Lemma~\ref{lemma:pullbacksMij}) and (b) are the \emph{only} relations among our generators (Theorem~\ref{theorem:Genus2AnyMarkings}). Step (b) is the hardest part and involves a careful sequence of row and column operations on a matrix representing the pullbacks of our generators to selected boundary divisors.
Finally, taking the difference of the number of generators and the number of relations, we then see that $S^{13,2}(\Mb_{2,14})$ has the desired dimension.

\subsection{Strategy for \texorpdfstring{$g=3$}{g=3} and \texorpdfstring{$n = 14$}{n=14}}
The strategy for $g = 3$ and $n = 14$ is similar. 
By Section \ref{sec:red} and Lemma \ref{132reduction}, it suffices to prove that $\dim S_\omega^{13,2}(\Mb_{3,14}) = 657657$.
Once we establish Theorem \ref{thm:rels}, determining the dimension of $S_\omega^{13,2}(\Mb_{3,n})$ becomes a combinatorial exercise in counting decorated graphs, which we carry out in Section \ref{sec:dimcounts}. In fact, we give a formula for $\dim S_\omega^{13,2}(\Mb_{3,n})$ for all $n$, and see that it yields $657657$ when $n = 14$, see Table~\ref{fig:dimensions}.

\section{Decorated graph generators for \texorpdfstring{$S^{13,2}_\omega(\Mb_{g,n})$}{S132Mgn}}\label{section:gens}

\subsection{Notation and conventions}
A stable graph $\Gamma$ corresponds to a product of moduli spaces of curves \[\Mb_{\Gamma} \colonequals \prod_{v \in \Gamma} \Mb_{g(v), n(v)}\] equipped with a gluing map $\xi\colon \Mb_{\Gamma} \to \Mb_{g,n}$. 
We will use a decorated stable graph to represent the $\xi$-pushforward of a class corresponding to the decorations on $\Mb_\Gamma$. For example, when there are no decorations, the graph represents the class $\xi_*(1)$. Note that we do \emph{not} divide by the order of the automorphism group of the graph.

As observed in the proof of \cite[Prop 2.4]{CLP24extensions}, every class in $S^*_\omega(\Mb_{g,n})$ can be represented as a linear combination of decorated graphs all of whose nontautological decorations are $H^{11}$-classes on genus 1 vertices.
Since there are no odd degree tautological classes, any odd degree class in $S^*_\omega$ must be pushed forward from the boundary. Such classes are represented by decorated graphs which, in order to have odd degree, must have at least one genus $1$ vertex with a decoration in $H^{11}$.
Our notation for decorations in $H^{11,0}$ is as follows.
Given $A \subset \{1, \ldots, n\}$ with $|A|  = 11$, let $f_A: \Mb_{1,n} \to \Mb_{1,A}= \Mb_{1,11}$ be the map that forgets all markings besides those in $A$. Additionally, let $\omega \in H^{11,0}(\Mb_{1,11})$ be a generator.
Then, we define $\omega_A \colonequals f_A^*\omega \in H^{11,0}(\Mb_{1,n})$.
We call $\omega_A$ an \emph{$\omega$-decoration}.

To simplify our graphs, we will sometimes use the following conventions:
   \begin{enumerate}
       \item A set (sometimes with cardinality indicated) attached to a half edge signifies that those markings are present on the vertex,
       \item On a genus-$1$ vertex with two edges we set $\omega^{S_2}_B = \frac{1}{2}(\omega_{B \cup \{p\}} + \omega_{B \cup \{q\}})$, where $p$ and $q$ are the markings corresponding to the edges.
   \end{enumerate}

\subsection{The generators and restrictions}
Specializing to degree $15$, we see that $S^{13,2}_\omega(\Mb_{g,n})$ 
is generated by decorated graphs having a genus-$1$ vertex with an $\omega$-decoration and one or two edges.
Figure \ref{fig:gentable} lists a collection of such graphs, which we will eventually show form a basis for $S^{13,2}(\Mb_{g,n})$ when $g \geq 2$.
The goal of this section is to establish that the listed graphs in fact span $S^{13,2}(\Mb_{g,n})$. The first step is the following.

\GeneratorChart

\begin{lemma} \label{lem:span}
Without imposing the restrictions in the third column, the classes displayed in Figure \ref{fig:gentable} generate $S^{13,2}(\Mb_{g,n})$.
\end{lemma}
\begin{proof}
A decorated graph generator for $S^{13,2}_\omega(\Mb_{g,n})$ can have one or two edges. We consider classes of each type in turn.

First suppose that the graph has two edges. In this case, the only decoration is an $\omega$-decoration on a genus-$1$ vertex. Possibilities for this decorated graph are given in the first five rows in Figure \ref{fig:gentable}. We now explain why only the listed $\omega$-decorations are necessary.
For classes of type $\S$ or type $\L$,
only the decorations $\omega_{B \cup p}$ are needed because they form a basis for $H^{11,0}(\Mb_{1, (X \cup Y)^c \cup p})$ by \cite[Corollary 2.4]{CLP23eleventh}. For classes of type $\M$ and $\N$ (which have the same underlying graph), consider $H^{11,0}(\Mb_{1,(X \cup Y)^c \cup \{p, q\}})$ together with the $S_2$ action that exchanges $p$ and $q$. The $S_2$-invariant subspace is spanned by the decorations $\omega_B^{S_2}$ for $B \subset (X \cup Y)^c$ with $|B| = 10$. Meanwhile, the
$S_2$ anti-invariant subspace is spanned by the classes $\omega_{B' \cup p \cup q}$ for $B' \subset (X \cup Y)^c$ with $|B'| = 9$. Thus, the listed $\omega$-decorated graphs span all $\omega$-decorated graphs of this type.
Finally, the graph for classes of type $\D$ has an automorphism exchanging the two edges which acts on the genus-$1$ vertex by exchanging $p$ and $q$. Since the pushforward along gluing equals the composition of the automorphism with the gluing map, all $S_2$ anti-invariant decorations would yield $0$.

Now suppose that the graph has one edge. If the edge is a self-edge, then the unique vertex in the graph would have genus $g-1$ and require a decoration in $H^{13}$ to produce a class in $H^{15}$. Replacing the class in $H^{13}$ with the corresponding $\omega$-decorated graph would yield a two-edged decorated graph.
Therefore, we may assume that the graph has two vertices: one vertex is genus-$1$ with an $\omega$-decoration, while the other vertex has a decoration in $H^2 = H^{1,1}$. If the decoration in $H^2$ is a boundary divisor, then again, we can expand the graph to have two edges. Thus, it suffices to assume the $H^2$ decoration is a $\psi$ or $\kappa$ class.
Meanwhile, by \cite[Corollary 2.4]{CLP23eleventh}, the $\omega$ decorations $\omega_{B \cup p} \in H^{11,0}(\Mb_{X^c \cup \{p\}})$ form a basis. Thus, the last two rows in the chart in Figure \ref{fig:gentable} cover all of the needed $1$-edged generators.
\end{proof}

In Sections \ref{sec:triv}-\ref{sec:g2rel}, we describe relations which allow us to rewrite a disallowed class in terms of allowed classes on the same row or classes on a higher row.
This completes the proof that the classes listed in Figure \ref{fig:gentable} generate $S^{13,2}(\Mb_{g,n})$ even after imposing the 
restrictions.

\subsection{Trivial relations} \label{sec:triv}
First off, we have the following trivial relations among the $\M$ and $\N$ classes, from drawing graphs in different ways:
\begin{align} \label{triv1} \M_{B',X,Y}^a &= -\M_{B',Y,X}^{g-a-1} \\
\N^a_{B, X, Y} &= \N^{g-a-1}_{B, Y, X} \label{triv2}
\end{align}
By applying these relations, we can eliminate classes that do not satisfy (R1), (R2) or (R4).

\subsection{Pushed forward relations} \label{sec:rels_all_g}
We can use known relations in $H^2(\Mb_{g-1,n'})$ to rewrite the $\P$ and $\K$ classes when $g$ is small. For example, the relation in \cite[(2.3)(c)]{AC98} (modified for our convention of not dividing by the order of the automorphism group of the graph) states that
\begin{equation}\label{eq:ACgenus1}
    \psi_i = \frac{1}{24} \delta_{irr} + \sum_{i \in S, |S| \geq 2} \delta_{0,S} \in H^{2}(\Mb_{1,n}).
\end{equation}
Pushing this forward, this gives a relation of the form
\begin{equation} \label{push1}\P_{B,X}^i = \frac{1}{24}\L_{B,X} + \begin{cases} \sum_{S, X = S \sqcup Y} \S^0_{B, S, Y}  & \text{if $i = q$} \\
\sum_{i \in S, X = S \sqcup Y} \S^0_{B, S, Y} + \sum_{i \in S, X = X' \sqcup S} \S^{g-1}_{B,X',S} & \text{if $i \neq q$}  \end{cases}
\end{equation}
when $g = 2$.
Using this relation, we can eliminate classes that do not satisfy restriction (R9).

Similarly, the relations in \cite[(2.3)(c)]{AC98} that express $\kappa_1$ in terms of other tautological classes give rise to the relations
\begin{equation} \label{push2} \K_{B,X} = \sum_{i \in X \cup q} \P^i_{B,X} - \sum_{X = X' \sqcup Y} \S^0_{B, X',Y} - \sum_{X = X' \sqcup Y} \S^1_{B, X', Y}
\end{equation}
if $g = 2$ and
\begin{equation} \label{push3} \K_{B,X} = \sum_{i \in X \cup q} \P^i_{B,X}  + \frac{1}{10} \L_{B,X} - 
\sum_{X = X' \sqcup Y} \S^0_{B, X',Y} - \sum_{X = X' \sqcup Y} \S^2_{B, X', Y}
+ \frac{7}{5}\sum_{X = X' \sqcup Y} \S^1_{B, X',Y}
\end{equation}
if $g = 3$. Using these, we can eliminate classes that do not satisfy the restriction (R10).

There are also relations among the generators that come from pushing forward relations in $H^{12,1}(\Mb_{1,n'})$. Essentially, we apply the same philosophy as above but to the two left vertices of the graph which represent a class in $H^{12,1}$. Generators and relations for $H^{12,1}(\Mb_{1,n'})$ are described in  \cite[Section 4.4]{CLPW24polynomial}.
The generator $Z_{B \subset A} \in H^{12,1}(\Mb_{g,n})$ is defined by the following decorated graph:
\begin{equation}\label{eq:GraphZ}
    Z_{B\subset A}=\GraphZ
\end{equation}

In genus $1$, the relations in $H^{12,1}(\Mb_{1,n'})$ allow one to write the classes with only two markings on the genus $0$ vertex in terms of others with three or more markings.

\begin{lemma}\cite[Lemma~4.11]{CLPW24polynomial}\label{lemma:basisGenus1}
The set
    \begin{equation*}
        \{Z_{B \subset A} : |A^c| \geq 3 \text{ or } |A^c| = 2 \text{ and } \min(A^c) < \min(B)\}
    \end{equation*}
    forms a basis for $H^{12,1}(\Mbar_{1,n})$.
\end{lemma}

Let us choose an ordering on $B' \cup \{x_1,x_2,q\}$ so that $\max(B')<x_1<x_2<q$. Applying 
\cite[Lemma 4.6]{CLPW24polynomial} with $(i, j,k) = (x_1, x_2, q)$ (and tracking additional signs related to the placement of $p$), we obtain the following relation in $H^{13}(\Mb_{1,B' \cup \{x_1,x_2,q\}})$:
\begin{equation*}
    \GraphrelMzeroaa \ = \ \GraphrelMzerobb \ +\  \GraphrelMzerocc
\end{equation*}

Next, we pull this relation back along the forgetful map $\Mb_{1,Y^c \cup q} \to \Mb_{1, B' \cup \{x_1,x_2,q\}}$ and then push the result forward along the gluing map $\Mb_{1,Y^c \cup q} \times \Mb_{g-1, Y \cup q'} \to \Mb_{g,n}$. This yields the relation
\begin{align}
\M^0_{B',\{x_1,x_2\},Y} &= 
\sum_{x_1 \in X' \subset Y^c \smallsetminus (B' \cup x_2)} \S^0_{B' \cup x_2, X', Y} +
\sum_{x_2 \in X' \subset Y^c \smallsetminus (B' \cup x_1)} \S^0_{B' \cup x_1, X', Y} \label{push4}  \\
&\qquad -
\sum_{\{x_1,x_2\} \subsetneq X' \subset Y^c \smallsetminus B'} \M^0_{B',X',Y} \notag
\end{align}
Using this relation, we can eliminate classes that do not satisfy (R3).

To find a similar relation involving the $\N$ classes, we first need to rewrite the symmetrized decoration. Let us write $B = \{b_1, \ldots, b_{10}\}$. The relations in $H^{11}(\Mb_{1,B \cup p \cup q})$ allow us to write
\begin{equation} \label{H11} \omega_B^{S_2} = \frac{1}{2}(\omega_{B \cup p} + \omega_{B \cup q}) = \omega_{B \cup p} + \frac{1}{2} \sum_{i=1}^{10}(-1)^i\omega_{(B \smallsetminus b_i) \cup p \cup q}.
\end{equation}
At this point, it is helpful to introduce a modified $\N$ class, denoted $\N'_{B,X,Y}$, defined by the decorated graph
\begin{equation*}
    \GraphNprime
\end{equation*}

This class $\N'_{B,X,Y}$ will also appear in Section \ref{subsection:NaturalGenerators}. 
Pushing forward \eqref{H11}, this modified class is related to our original generators by the equation
\begin{equation} \label{Nprime}\N'_{B,X,Y} = \N^0_{B,X,Y} - \frac{1}{2} \sum_{i=1}^{10} (-1)^i\M^0_{B\smallsetminus b_i,X,Y}. 
\end{equation}
Next, note that we have the following relation in $H^{13}(\Mb_{1,B \cup \{x_1,x_2\}})$:
\begin{equation*}
\Graphreloneaa \ = \ \Graphrelonebb \ + \ \Graphrelonecc
\end{equation*}

Pulling back along the forgetful map $\Mb_{1, B \cup \{x_1,x_2,q\})} \to \Mb_{1, B \cup \{x_1,x_2\}}$, we obtain the following relation in 
$H^{13}(\Mb_{1,B \cup \{x_1,x_2,q}\})$
\begin{equation*}
    \Graphreltwoaa = \Graphreltwobb + \Graphreltwocc + \Graphreltwodd + \Graphreltwoee - \Graphreltwoff
\end{equation*}

Now, we pull the above relation back along the forgetful map $\Mb_{1, Y^c \cup q} \to \Mb_{1, B \cup \{x_1,x_2,q\}}$ and then pushforward the result along the gluing map $\Mb_{1, Y^c \cup q} \times \Mb_{g-1, Y \cup q'} \to \Mb_{g,n}$. This gives
\begin{align}
\sum_{X' \subset Y^c \smallsetminus (B \cup \{x_1, x_2\})} \N'_{B, X' \cup \{x_1,x_2\}, Y}
&= \sum_{X' \subset Y^c \smallsetminus (B \cup \{x_1, x_2\})} \N'_{B \smallsetminus b_1 \cup x_1, X' \cup \{b_1,x_2\}, Y}
+\N'_{B \smallsetminus b_1 \cup x_2, X' \cup \{b_1,x_1\}, Y}  \notag \\
&+ \S^0_{B \smallsetminus b_1 \cup x_1, X' \cup \{b_1,x_2\}, Y}
+ \S^0_{B \smallsetminus b_1 \cup x_2, X' \cup \{b_1,x_1\}, Y}
- \S^0_{B, X' \cup \{x_1,x_2\}, Y} \label{Nprimerel}
\end{align}
Combining the above equation with \eqref{Nprime}, we obtain
\begin{align}
\N^0_{B, \{x_1,x_2\},Y} &= 
\frac{1}{2}\sum_{i=1}^{10}(-1)^i \M^0_{B \smallsetminus b_i, \{x_1, x_2\}, Y} -\sum_{\{x_1,x_2\} \subsetneq X' \subset Y^c \smallsetminus B}\left(\N^0_{B,X',Y} - \frac{1}{2} \sum_{i=1}^{10} (-1)^i\M^0_{B \smallsetminus b_i, X', Y} \right) \notag \\
& + \sum_{\{b_1,x_2\} \subset X' \subset Y^c \smallsetminus B}\left(\N^0_{B \smallsetminus b_1 \cup x_1,X',Y} - \frac{1}{2} \sum_{i=2}^{10} (-1)^{i-1} \M^0_{B \smallsetminus \{b_1,b_i\} \cup x_1, X', Y}  - \frac{1}{2} \M^0_{B \smallsetminus b_1, X', Y}\right) \notag \\
& + \sum_{\{b_1,x_1\} \subset X' \subset Y^c \smallsetminus B}\left(\N^0_{B \smallsetminus b_1 \cup x_2,X',Y} - \frac{1}{2} \sum_{i=2}^{10} (-1)^{i-1} \M^0_{B \smallsetminus \{b_1,b_i\} \cup x_2, X', Y}  - \frac{1}{2} \M^0_{B \smallsetminus b_1, X', Y}\right) \label{push5} \\
&+ \sum_{X' \subset Y^c \smallsetminus (B \cup \{x_1, x_2\})} \S^0_{B \smallsetminus b_1 \cup x_1, X' \cup \{b_1,x_2\}, Y}
+ \S^0_{B \smallsetminus b_1 \cup x_2, X' \cup \{b_1,x_1\}, Y}
- \S^0_{B, X' \cup \{x_1,x_2\}, Y} \notag
\end{align}
Notice how in \eqref{push5} the $\N$ classes on the right hand side with only two markings on the genus $0$ vertex have $b_1 = \min(B)$ among those two markings.
In particular, using these relations, we can eliminate classes that do not satisfy (R5).

\begin{remark}
The proof of Lemma \ref{lem:span} already uses up all pushed forward relations from $H^{11}(\Mb_{1,n'})$ in order to make the $\omega$-decorations on our generators have a certain shape.
\end{remark}

\subsection{Pulled back relations} \label{sec:g2rel}
We will show later that when $g \geq 3$, the relations in Sections \ref{sec:triv} and \ref{sec:rels_all_g} generate all relations among the decorated graph generators in Figure \ref{fig:gentable}.
However, in genus $2$, there are other, genuinely new relations that do not arise in this way. 

The following three lemmas record these additional relations for $g = 2$ and $n = 10, 11, 12$. We will show later that the pullbacks of these relations together with the relations in Sections \ref{sec:triv} and \ref{sec:rels_all_g}
in fact generate all relations in $S_\omega^{13,2}(\Mb_{2,n})$ for any $n$.

\begin{lemma} \label{lem:rel10}
When $n = 10$, the only relation among the generators is
\begin{equation*} \L_{[10], \varnothing} = 0 \in H^{15}(\Mb_{2,10}).
\end{equation*}
\end{lemma}
\begin{proof}
With $n = 10$, the only generator is $\L_{[10], \varnothing}$. Meanwhile,
$H^{15}(\Mb_{2,10})$ is Poincar\'e dual to $H^{11}(\Mb_{2,10}) = 0$.
\end{proof}

Pulling back the relation in Lemma \ref{lem:rel10} along the map $\Mb_{2,n} \to \Mb_{2,B}$ that forgets the marked points besides $B$ gives the relation 
\begin{equation} \label{rel10pull}
\L_{B, \varnothing} = - \sum_{\varnothing\neq S \subset [n]\smallsetminus B} \L_{B,S}.
\end{equation}
The relation \eqref{rel10pull} allows us to eliminate classes that do not satisfy (R6).

\begin{lemma} \label{lem:rel11}
When $n = 11$, all relations among the generators come from \eqref{rel10pull} together with the relation
\begin{equation*} \D_{[11] \smallsetminus i, i} + \L_{[11] \smallsetminus i, i} + 24  \cdot \S^0_{[11] \smallsetminus i, i, \varnothing} = 0 \in H^{15}(\Mb_{2,11})
\end{equation*}
for $i = 1, \ldots, 11$.
\end{lemma}
\begin{proof}
We have that $H^{13,2}(\Mb_{2,11})$ is Poincar\'e dual to $H^{1,12}(\Mb_{2,11})$. A basis for $H^{1,12}(\Mb_{2,11})$ is given by the (conjugate) $Z_{B \subset A}$ classes for $B \subset A \subset [11]$, defined in \cite[Section 4]{CLPW24polynomial}. There are $11$ choices for the set $B$ of size $10$ and then either $A = B$ or $A = [11]$. Hence $\dim H^{13,2}(\Mb_{2,11}) = 22$. 
After taking into account (R1) -- (R6) and (R9) and (R10), there are $33$ generators remaining, namely
$\D_{[11] \smallsetminus i, i}$, $\L_{[11] \smallsetminus i, i}$, and $\S^0_{[11] \smallsetminus i, i, \varnothing}$, each for $i \in [11]$. Thus, there must be an $11$-dimensional space of relations among these classes.

One readily checks that the class given in the statement of the lemma pairs to zero with each of the conjugate $Z_{B \subset A}$, hence defines a relation. 
Since this gives $11$ relations and $\dim H^{13,2}(\Mb_{2,11}) = 22$, 
there can be no other relations.
\end{proof}

Pulling back the relation in Lemma \ref{lem:rel11} along the map $\Mb_{2,n} \to \Mb_{2,B \cup x}$ that forgets the marked points besides a size $10$ subset $B$ and another marking $x$, we find that

\begin{equation} \label{rel11pull}
\sum_{x \in X \subseteq [n]\smallsetminus B}\D_{B, X} = - \sum_{x \in S \subset [n] \smallsetminus B} \L_{B, S \cup x} - 24 \sum_{\substack{X \sqcup Y \subset [n] \smallsetminus B \\ x \in X}} \S^0_{B,X,Y}
\end{equation}
Using \eqref{rel11pull} allows us to eliminate classes that do not satisfy (R7).

The final restriction (R8) comes from relations with $12$ markings. These are the most complicated relations we describe. As in \cite[Lemma~7.3]{CLP24extensions}, after taking into account the restrictions (R1) -- (R7), (R9), and (R10), there are $891$ decorated graph generators remaining.
Let $M_{ij} \in H^{15}(\Mbar_{2,12})$ be defined as in the proof of \cite[Lemma~7.3]{CLP24extensions}, multiplied by $24$ in order to avoid denominators:

\begin{align}
\begin{split}\label{Mij-def}
M_{ij} &= 24\,\GraphMijaa + \GraphMijbb + 24\,\GraphMijcc + 24\,\GraphMijdd
        - 24\,\GraphMijee + 2\,\GraphMijff \\
        &= 24 \S_{X, \{i,j\}, \emptyset}^0 + \L_{X, \{i,j\}} + 24 \S_{X, \{i\}, \{j\}}^0 + 24\S_{X, \{j\}, \{i\}}^0 - 24 \S_{X, \emptyset, \{i,j\}}^1 + 2\D_{X,\{i,j\}}
\end{split}
\end{align}
Here $X = [12] \smallsetminus \{i,j\}$.
The symbols $M_{ij}$ correspond to the $66$-dimensional $\mathbb{S}_{12}$-representation 
\[\Ind_{\mathbb{S}_{10} \times \mathbb{S}_2}^{\mathbb{S}_{12}}(V_{1^{10}} \boxtimes \mathbf{1}) = V_{2,1^{10}} \oplus V_{3,1^9}.\]
This representation was also studied in \cite[Section 4.4]{CLPW24polynomial}. More concretely, if $Z_{ij} = Z_{\{1,\ldots , \hat{i},\ldots , \hat{j},\ldots 12\}\subset [12]}\in H^{12,1}(\Mbar_{1,12})$\footnote{The decorated graphs $Z_{B\subset A}$ were defined in \eqref{eq:GraphZ}} then, for any $\sigma \in \mathbb S_{12}$, with action $\sigma^* : H^*(\Mbar_{g,12})\to H^*(\Mbar_{g,12})$, we have
\begin{equation}\label{eq:same-action}
\sigma^* Z_{ij} = \operatorname{sgn}\big(\sigma |_{[12]\setminus \{i,j\}}\big)Z_{\sigma(i),\sigma(j)}\,, \quad \sigma^* M_{ij} = \operatorname{sgn}\big(\sigma |_{[12]\setminus \{i,j\}}\big)M_{\sigma(i),\sigma(j)}
\end{equation}
where $ \operatorname{sgn}\big(\sigma |_{[12]\setminus \{i,j\}}\big)$ is computed after the elements $\sigma([12]\setminus \{i,j\})$ are put in increasing order, cf. \cite[(4.3)]{CLPW24polynomial}. In \cite[Lemma 4.6]{CLPW24polynomial} it is shown that there are $55$ relations between the $Z_{ij}$, and in \cite[Lemma~7.3]{CLP24extensions} it is shown that there are $55$ relations among these $M_{ij}$, corresponding to the subrepresentation $V_{3,1^9}$. By \eqref{eq:same-action}, the relations between the $M_{ij}$ are the same as the relations between the symbols $Z_{ij}$, which were computed in \cite[Lemma 4.6]{CLPW24polynomial}. Thus, for $1 \leq i < j < k \leq 12$, we have
\[0 = (-1)^{i+j} M_{ij} - (-1)^{i+k} M_{ik} + (-1)^{j+k} M_{jk}.\]
These relations are generated by the ones for $1 = i < j < k$.
Moreover since $\dim H^{13,2}(\Mb_{2,12}) = 836$, these $55$ relations are the only remaining relations, giving the following lemma, proved in \cite[Lemma~7.3]{CLP24extensions}. 

\begin{lemma} \label{rel12}
When $n = 12$, all relations among the generators come from \eqref{rel10pull} and \eqref{rel11pull} together with the relations
\begin{equation*} 
M_{jk} = (-1)^{k} M_{1j} - (-1)^{j} M_{1k}
\end{equation*}
for all $j,k \in \{2, \ldots, 12\}$, $j < k$.
\end{lemma}
\begin{remark}
To explain the signs, recall that an $\omega$-decoration $\omega_A$ is anti-symmetric in the order of the size $11$ subset $A$. 
If we apply the automorphism that swaps $j$ and $k$, then we will be inserting $j$ into the position of $k$ in some set $A$, which may make it no longer increasing. This is why there are additional signs in the statement of these relations. In the case that $k = j+1$, this swap does not require reordering and the formula for $M_{jk}$ is easily seen to be symmetric in $j$ and $k$.
\end{remark}

Now suppose $n \geq 12$ and we are given $B$ of size $10$ and $j<k$ with $j, k \notin B$ and $\min\{j, k\} > \min(B)$. 
Let $\pi\colon \Mb_{2,n} \to \Mb_{2,B \cup \{j,k\}}$ be the forgetful map. Let 
$b = \min(B) = \min(B \cup \{j,k\})$ and let $\beta(j)$ and $\beta(k)$ be the place rank (in $[12]$) of $j$ and $k$, respectively, in $B \cup \{j,k\}$.
Pulling back the definition of $M_{jk}$ and then applying the pullback of the relation in Lemma \ref{rel12}, we obtain the relation
\begin{align} \label{rel12pull} \D_{B, \{j,k\}} &= \tfrac{1}{2} \pi^* M_{jk} + \text{classes of type $\S^a_{B,X,Y}$, $\L_{B,X}$ and type $\D_{B,X}$ satisfying (R7)}  \notag \\
&= (-1)^{\beta(k)}\D_{(B\smallsetminus\{b\})\cup \{k\}, \{b,j\}} - (-1)^{\beta(j)}\D_{(B\smallsetminus\{b\})\cup \{j\},\{b,k\}}  \\
&\qquad \qquad + \text{classes of type $\S^a_{B,X,Y}$, $\L_{B,X}$ and type $\D_{B,X}$ satisfying (R7)} \notag
\end{align}

This allows us to eliminate classes that do not satisfy (R8).

\section{Generators and Relations for \texorpdfstring{$H^{15}(\Mbar_{2,n})$}{H15M2n}}\label{section:genus2}

In this section, we prove Theorem \ref{theorem:main} and Theorem \ref{thm:rels}(a).
The following lemma establishes that the different pullbacks of subspaces spanned by $M_{ij}$ classes are all independent. This implies in particular that the only relations among the different pullbacks of $M_{ij}$ classes are the pullbacks of relations on $\Mb_{2,12}$. This is a central ingredient in the proof of Theorem~\ref{theorem:Genus2AnyMarkings}, implying Theorems~\ref{theorem:main} and~\ref{thm:rels}(a).

\begin{lemma}\label{lemma:pullbacksMij}
     Let $V \subseteq H^{15}(\Mbar_{2,12})$ be the span of all $M_{ij}$. For any subset $A \subseteq [n]$ of cardinality $|A| = 12$, let $\pi_{[n] \setminus A}: \Mbar_{2,n} \rightarrow \Mbar_{2,A}$ be the forgetful morphism and set $V_A \coloneqq \pi_{[n] \setminus A}^\ast(V) \subseteq H^{15}(\Mbar_{2,n})$. Then the sum $\sum_{A \subseteq [n]} V_A$ is direct. Equivalently, the dimension of the latter is $11 \cdot \binom{n}{12}$ and, as $\mathbb{S}_n$-representations,
    \begin{equation*}
        \sum_{A \subseteq [n]} V_A \cong \mathrm{Ind}_{\mathbb{S}_{12} \times \mathbb{S}_{n-12}}^{\mathbb{S}_n} V_{2,1^{10}}.
    \end{equation*}
\end{lemma}

\begin{proof}
    For any $2$-element subset $\{x,y\} \subseteq [n]$, let
    \begin{equation*}
         \theta_{xy}: \Mbar_{2,([n]\setminus xy) \cup p} \cong \Mbar_{2,([n]\setminus xy) \cup p} \times \Mbar_{0, \{x,y,q\}} \rightarrow \Mbar_{2,n}
    \end{equation*}
   be the map that glues $p$ and $q$. We first assume $\{x,y\} \subseteq A$. Then
   \begin{equation*}
       \theta_{xy}^\ast \pi_{[n] \setminus A}^\ast(V) = \pi_{[n] \setminus A}^\ast \tilde{\theta}_{xy}^\ast(V),
   \end{equation*}
   where $\tilde{\theta}_{xy}$ is the same map as $\theta_{xy}$ with the set $[n]$ replaced by $A$. We claim $\tilde{\theta}_{xy}^\ast(V) = 0$. Indeed, this follows by computing the pullback of each term in $M_{ij}$ using the excess intersection formula on decorated graphs, as we now explain. 
   
   Let $X = A \setminus\{i,j\}$.
   The pullbacks of the terms in $M_{ij}$ vanish individually if $\{i,j\} \neq \{x,y\}$. If $\{i,j\} = \{x,y\}$, the only possible non-vanishing K\"unneth component of the pullback is $H^{15} \otimes H^0$. Since $i$ and $j$ are not on the same vertex, the third and fourth graphs appearing in the definition \eqref{Mij-def} of $M_{ij}$ vanish under pullback to this boundary divisor. We then obtain the following expression for the $H^{15}$ K\"unneth factor of the pullback $\tilde{\theta}_{ij}^*M_{ij}$. The equality of the first line with the second line follows by expanding the $\psi_s$ class using the relations in $H^2(\Mb_{1,2})$ from \cite[Theorem 2.2(c)]{AC98}:

\begin{align*}
&24\GraphEquationChainS[.78]
+\GraphEquationLoopS[.78]
-24\GraphEquationPsiS[.78]
+2\GraphEquationDoubleEdgeS[.78]\\
 &=2\Bigl(
24\GraphEquationChainS[.78]
+\GraphEquationLoopS[.78]
+\GraphEquationDoubleEdgeS[.78]
\Bigr).
\end{align*}
    
\noindent The latter vanishes due to the relation in Lemma \ref{lem:rel11}.

    Next, we assume $\{x,y\} \not\subseteq A$, say $x \in [n] \setminus A$. Then $\theta_{xy}$ can be viewed as the $y$-th section of $\pi_{x}: \Mbar_{2,n} \rightarrow \Mbar_{2, [n] \setminus x}$, up to replacing $y$ by $p$ in the source. Let $\sigma_{p,y} : \Mbar_{2,([n]\setminus xy) \cup p} \rightarrow \Mbar_{2, [n] \setminus x}$ be the isomorphism that maps $p$ to $y$ and all other markings to themselves. It follows that, in $H^{15}(\Mbar_{2, ([n] \setminus xy) \cup p})$, we have
    \begin{equation*}
        \theta_{xy}^\ast(V_A) = \theta_{xy}^\ast \pi_{x}^\ast \pi_{[n] \setminus (A \cup x)}^\ast(V) = \sigma_{p,y}^\ast \pi_{[n] \setminus (A \cup x)}^\ast(V) = \begin{cases}
            \pi_{[n] \setminus (A \cup x)}^\ast(V) & \text{if } y \in A\\
            \pi_{([n] \cup p) \setminus (A \cup xy)}^\ast(V) & \text{if } y \in [n] \setminus A.
        \end{cases}
    \end{equation*}
    The cardinality of $([n] \setminus xy) \cup p$ is $n-1$, so by induction on $n$ we find that, for any fixed $\{x,y\} \subseteq [n]$, the sum
    \begin{equation*}
        \sum_{A \subseteq [n] \setminus \{x,y\}} \pi_{([n] \cup p) \setminus (A \cup xy)}^\ast(V) \subseteq H^{15}(\Mbar_{2, ([n] \setminus xy) \cup p})
    \end{equation*}
    is direct and intersects $\sum_{|A \cap \{x,y\}| = 1} \pi_{[n] \setminus (A \cup xy)}^\ast(V) \subseteq H^{15}(\Mbar_{2, ([n] \setminus xy) \cup p})$ trivially. Hence, the sum
    \begin{equation*}
        \sum_{A \subseteq [n] \setminus \{x,y\}} V_A \subseteq H^{15}(\Mbar_{2,n})
    \end{equation*}
    is also direct and intersects $\sum_{|A \cap \{x,y\}| \geq 1} V_A \subseteq H^{15}(\Mbar_{2,n})$ trivially. Running over all $2$-element subsets $\{x,y\} \subseteq [n]$ then yields the desired result.
\end{proof}

\subsection{The natural generators in genus 2}\label{subsection:NaturalGenerators}
The generators for $S_\omega^{13,2}(\Mbar_{g,n})$ that we want to use for $g = 2$ differ slightly from those introduced in Figure~\ref{fig:gentable}: Instead of the generators $\M^0$ and $\N^0$ where the $\omega$-marking on the middle vertex is anti-symmetric and symmetric, respectively, under exchanging the internal half-edges $p$ and $q$, we choose the decoration $\omega_{B \cup p}$ where $B \subseteq A \cup \{q\}$ is any subset of cardinality $|B| = 10$, where $A \coloneqq (X \cup Y)^c$ in the notation of Figure~\ref{fig:gentable}. Note that $B$ may or may not contain $q$. 
We denote these new classes by $\N'_{B,X,Y}$:
\[
\GraphChainXComplementY .
\]
\noindent They are related to the generators in Figure \ref{fig:gentable} by
\begin{equation} \label{Nprimedef} \N'_{B,X,Y} = \begin{cases}
\M^0_{B \smallsetminus q, X, Y} & \text{if $q \in B$} \\
\N^0_{B,X,Y} - \frac{1}{2} \sum_{i=1}^{10}(-1)^i \M^0_{B \smallsetminus b_i, X, Y} & \text{if $q \notin B$.}
\end{cases}
\end{equation}
As explained following \eqref{push4}, we can assume that if $|X| = 2$, then $\min(X) < \min(B)$ and $q \not\in B$. 

Next, we define \emph{modified double edge generators} for $|X| = 2$ by
\begin{equation}\label{D'}
    \D'_{B,ij} \coloneqq \pi_{[n] \setminus A}^\ast(M_{ij}),
\end{equation}
where $A \coloneqq B \cup \{i,j\}$ and $\pi_{[n] \setminus A}: \Mbar_{2,n} \rightarrow \Mbar_{2,A}$ is the forgetful morphism, i.e., we replace $\D_{B,ij}$ by $\D'_{B, ij}$ for all $i$, $j$ and $B$ with $B \subseteq [n] \setminus \{i,j\}$, $|B| = 10$. We keep the $\D_{B,X}$ with $|X| \geq 3$ and call them \emph{unmodified double edge generators}.

Moreover, we define \emph{modified loop generators} by
\begin{equation}\label{L'}
    \L'_{B,X} \coloneqq
    \L_{B,X} + 24 \sum_{\emptyset \neq T \subseteq X} \S^0_{B,T,X \setminus T}
\end{equation}
for $|X| \geq 3$. This means that we replace $\L_{B,X}$ by $\L'_{B,X}$ for all $X$ with $|X| \geq 3$. We keep $\L_{B,X}$ for $|X| \in \{1,2\}$ and call them \emph{unmodified loop generators}.
All of these modifications only amount to a change of basis for the generators of $S_\omega^{13,2}(\Mbar_{2,n})$; in particular, the resulting classes still span $S_\omega^{13,2}(\Mbar_{2,n})$.
From now on, we refer to this new set of generators for $S_\omega^{13,2}(\Mbar_{2,n})$ as \emph{the natural generators}.\\

\subsection{The natural generators span \texorpdfstring{$H^{15}(\Mbar_{2,n})$}{H15M2n}}
The following is the main result of this section. The proof involves formulas for the K\"unneth components of the pullbacks of our decorated graph generators to various boundary divisors, which are computed using the calculus of decorated graphs and the excess intersection formula. The necessary formulas are given in Lemmas \ref{lemma:intersectionTableH13xH2} -- 
\ref{lemma:independenceGenus1}, which are collected after the proof.

\begin{theorem}\label{theorem:Genus2AnyMarkings}
   (a) We have
    \begin{equation*}
        H^{13,2}(\Mbar_{2,n}) = S_\omega^{13,2}(\Mbar_{2,n})
    \end{equation*}
    for all $n$.
    
    (b) Moreover, the only relations among the natural generators are those pulled back from $\Mbar_{2,12}$, i.e., the only relations are the pullbacks under forgetful maps of the known relations among the $M_{ij}$.
\end{theorem}

\begin{proof}
The strategy to prove part (a) has already been outlined in Section~\ref{section:ReductionsTo14Markings}. Ignoring restriction (R8) in Figure~\ref{fig:gentable}, a simple count of the number of the natural generators in genus $g = 2$ for $n = 14$ markings gives $239239$ (see Section \ref{sec:dimcounts} for general formulas counting generators).
Meanwhile, the number of relations among the $\D'_{B,ij}$ is $55 \cdot \binom{14}{12} = 5005$ by Lemma~\ref{lemma:pullbacksMij}. The difference is precisely the number $234234 = \dim H^{13,2}(\Mbar_{2,14})$ discussed in Section~\ref{subsection:StrategyGenus2Markings14}. 
For part (a), 
it is therefore enough to show that, modulo the $\D'$ classes, the remaining natural generators are linearly independent when $n = 14$.

Meanwhile, for part (b), we must show that, for any number of markings, the only relations are those pulled back from $\Mb_{2,12}$. By Lemma \ref{lemma:pullbacksMij}, the only relations in the subspace spanned by pullbacks of $M_{ij}$ are the pullbacks of relations from $\Mb_{2,12}$. To show there are no more relations, it suffices to show that, modulo the $\D'$ classes, the remaining natural generators are linearly independent for all $n$.

To do so, suppose we are given a linear relation among the natural generators of the form
\begin{align} 
&\sum s^a_{B,X,Y} \S^a_{B,X,Y} + \sum n_{B,X,Y} \N'_{B,X,Y} + \sum_{|X| \in \{1,2\}} \ell_{B,X} \L_{B,X} \notag \\
& \qquad \qquad \qquad \qquad \qquad \qquad  + \sum_{|X| \geq 3} \ell'_{B,X} \L'_{B,X} + \sum_{|X| \geq 3} d_{B,X} \D_{B,X}= \sum_{|X| = 2} d'_{B,X} \D'_{B,X}. \label{eq:rel}
\end{align}
Our goal is to show that $s^a_{B,X,Y} = n_{B,X,Y} = \ell_{B,X} = \ell'_{B,X} = d_{B,X} = 0$.

The basic strategy is to pull back this linear relation along different boundary divisors, obtaining relations in various K\"unneth components of the boundary divisor. We indicate the K\"unneth components under consideration by writing them above the stable graph corresponding to the boundary divisor under consideration.
Ranging over different pullback maps and K\"unneth components allows us to show that different coefficients in the linear relation \eqref{eq:rel} must vanish. When we have shown that the coefficient of a certain generator is zero, we will say that this generator has been eliminated.

We proceed by computing pullbacks in a
carefully chosen order. The first five steps show the vanishing of the coefficients $\ell_{B,X}$, and vanishing of the coefficients for the generators with a rational tail with exactly $2$ markings on the rational component. Step (6) eliminates all remaining (unmodified) natural generators that have a rational tail. The final Steps (8)--(9) inductively eliminate the remaining side generators, all unmodified double edge generators and all modified loop generators.

    \begin{enumerate}
        \item We pull back to \GraphKunnethElevenFour[.75] with $X' = \{i\}$. By Lemma \ref{lemma:intersectionTableH11xH4}, all natural generators give zero except for $\L_{B,i}$ whose pullbacks are linearly independent as we run over $B \subseteq [n] \setminus \{i\}$. We conclude that $\ell_{B,i} = 0$ for all $B$ and $i$.
        
        \item We pull back to  \GraphKunnethThirteenTwo[.75] with $X' = \emptyset$. By Lemmas~\ref{lemma:vanishingGenus2} and~\ref{lemma:D'vanishingGenus2}, the pullbacks of the $\D'$ and $\L'$ classes all vanish. By Lemma~\ref{lemma:intersectionTableH13xH2}, the pullbacks also vanish for all the $\D$ and $\S^1$, as well as the $\S^0_{B,X,Y}$ and $\N'_{B,X,Y}$ with $Y \neq \varnothing$.
Moreover, we have already shown that $\ell_{B,X} = 0$ for $|X| = 1$.
We break the remaining generators into four groups according to the shape of their $H^{13}$ K\"unneth component under this pullback.
\begin{enumerate}
    \item[(I)] The $\S^0_{B,X,\varnothing}$ with $|X|= 2$ and $\L_{B,X}$ with $|X| = 2$ both yield 
$Z_{B \subset A}$ classes in $H^{13}(\Mb_{1,[n] \cup s})$ with $A^c = X \cup s$. In particular, $s \in A^c$ and $|A^c| = 3$.
\item[(II)] The 
$\S^0_{B,X,\varnothing}$ with $|X|> 2$ yield 
$Z_{B \subset A}$ classes in $H^{13}(\Mb_{1,[n] \cup s})$ with $A^c = X \cup s$. In particular, $s \in A^c$ and $|A^c| > 3$.
\item[(III)] The $\S^0_{B, i, \varnothing}$ yield $Z_{B \subset A}$ with $A^c = \{i, s\}$. In particular, $|A^c| = 2$ and $s \in A^c$. 
\item[(IV)] The $\N'_{B,X, \varnothing}$
yield $Z_{B \subset A}$ classes 
with $A = X^c \cup s$. In particular, $s \notin A^c$.
Additionally, our assumption that if $|X| = 2$, then $q \notin B$ and $\min(X) < \min(B)$ ensures that the resulting $Z_{B \subset A}$ satisfy $\min(A^c) < \min(B)$ when $|A^c| = 2$.
\end{enumerate}
We now apply Lemma
~\ref{lemma:basisGenus1} to $H^{13}(\Mb_{1,[n] \cup s})$, using an ordering in which $s$ is minimal.
In particular, the classes in (II), (III) and (IV) are independent  modulo those in (I).     
Thus, we conclude that 
$s^0_{B,X,\varnothing} = 0$ for $|X| \geq 3$, and $|X| = 1$, and $n_{B, X, \emptyset} = 0$ for all $X$.

\item We pull back to \GraphKunnethElevenFour[.75] with $X' = \{i,j\}$. Since the $\L_{B,X}$ with $|X| = 1$ and the $\S^0_{B,i,\emptyset}$ have been eliminated, and all the $\L'$ and $\D'$ generators evaluate to zero by Lemmas~\ref{lemma:vanishingGenus2} and~\ref{lemma:D'vanishingGenus2}, we deduce from Lemma~\ref{lemma:intersectionTableH11xH4} that the only three sets of generators that do not pull back to zero are $\S^0_{B_1,\{i,j\},\emptyset}$, $\S^1_{B_2,\emptyset,\{i,j\}}$ and $\L_{B_3,\{i,j\}}$. Their pullbacks have the same first Künneth components if and only if $B_1 = B_2 = B_3$, and in this case their second Künneth components in $H^4(\Mbar_{1,\{t,i,j\}}) = H^4(\Mb_{1,3})$ can be explicitly shown to be linearly independent using \texttt{admcycles}. Moreover, the first Künneth component of each forms a basis of $H^{11}(\Mbar_{1,([n]\setminus ij) \cup s})$ when running over all $B \subseteq [n] \setminus ij$. 
Hence, we conclude $s^0_{B,X, \varnothing} = s^1_{B, \varnothing, Y} = \ell_{B, X} = 0$ for $|X| = 2$ or $|Y| = 2$ and any $B$.
        
\item We pull back to \GraphKunnethThirteenTwo[.75] with $X' = \{i,j\}$. Since we have already eliminated the $\S^1_{B,\emptyset,\{i,j\}}$ in Step~(3), Lemma~\ref{lemma:intersectionTableH13xH2} now shows that we can also eliminate all the $\S^1_{B,S,\{i,j\}}$ with $S \neq \emptyset$. This is because these are the only remaining natural generators whose pullbacks have a nonzero second Künneth component that, after expanding each $\psi_t$ in terms of boundary divisors via \cite[eq.~(1.11)]{AC98}, involves a reducible boundary divisor with the marking $t$ on the genus~$1$ vertex. Since there are no non-trivial relations among boundary divisors in genus~$1$, this together with Step~(3) eliminates all the $\S^1_{B,X,Y}$ with $|Y| = 2$.

\item We again pull back to \GraphKunnethThirteenTwo[.75], this time with $X'$ arbitrary. By Lemmas \ref{lemma:intersectionTableH13xH2} and \ref{lemma:D'vanishingGenus2} respectively, this K\"unneth component of the pullbacks of the $\D$ and $\D'$ classes vanish. We have already eliminated the $\L$ classes in steps (1) and (3). We break the remaining generators into groups according to the shape of their $H^{13}$ K\"unneth component under this pullback, which yields $Z_{B \subset A}$ classes in $H^{13}(\Mb_{1, [n]\smallsetminus X' \cup s})$. 

\begin{enumerate}
\item[(I)] The $\S$ classes correspond to Rows 1, 2, 3, 6, 7 in Lemma \ref{lemma:intersectionTableH13xH2}. After possibly expanding a $\psi_s$ class in terms of boundary divisors whose genus $0$ component contains $s$, we see that the $H^{13}$ K\"unneth factor of the pullbacks are all linear combinations of $Z_{B \subset A}$ classes with $s \in A^c$.
\item[(II)] After possibly expanding the $\psi_s$ class, the $\L$ classes in Rows 8 and 9 of Lemma \ref{lemma:intersectionTableH13xH2} yield linear combinations of $Z_{B \subset A}$ classes with $s \in A^c$. Since the $\L'$ classes are linear combinations of $\L$ classes and $\S$ classes, combined with our observation in (I), the $\L'$ classes also yield linear combinations of $Z_{B \subset A}$ classes with $s \in A^c$. 
\item[(III)] The $\N'_{B, X, S}$ with $|X| \geq 3$ yield $Z_{B \subset A}$ with $A^c = X$. In particular $s \notin A^c$ and $|A^c| \geq 3$.
\item[(IV)] The $\N'_{B,X,S}$ with $|X| = 2$ and $S \subset X'$ correspond to classes in Rows 4 and 5 of Lemma \ref{lemma:intersectionTableH13xH2}. The $H^{13}$ K\"unneth component is $Z_{B \subset A}$ with $A^c = X$. Notice that the first K\"unneth component is determined just by $B$ and $X$. Moreover, it satisfies $|A^c| = 2$ and $s \notin A^c$, and our assumption that $\min(X) < \min(B)$ ensures that the condition $\min(A^c) < \min(B)$ is satisfied. Applying  Lemma~\ref{lemma:basisGenus1}, we see that these classes are independent of the span of the classes in groups (I) -- (III). Additionally, as we run over all $S \subset X'$, the second K\"unneth factors in $H^2(\Mb_{1,X' \cup t})$ of the pullbacks run over different separating boundary divisors and $-\psi_t$. Since $-\psi_t = -\frac{1}{24} \delta_{irr} + (\text{other boundary classes})$, and there are no relations among the boundary divisors on $H^2(\Mb_{1,X' \cup t})$ \cite[Theorem 2.2(c)]{AC98}, these second K\"unneth components are all independent. 
\end{enumerate}
By applying this argument as we vary over different $X'$, we conclude that $n_{B,X,S} = 0$ for all $|X| = 2$ and any $S$.

        \item We pull back to $G_{X'} \coloneqq$ \GraphKunnethThirteenTwoGenusZero[.75], running over all $X'$ in any order of increasing cardinality, starting with all $X'$ of cardinality $|X'| = 3$. The key observation is that every decorated graph that does not have a rational tail has vanishing $H^{13} \otimes H^2$ K\"unneth component under pullback. 
        Indeed, suppose $\Gamma$ is the underlying stable graph and $\Gamma$ has no rational tail. Any generic $(G_{X'},\Gamma)$-graph $C$ must have a rational tail consisting of a single genus $0$ vertex with the markings $X'$. Moreover, the edge to this tail gets contracted in the map $C \to \Gamma$. Hence, this edge is not an excess edge and so does not get decorated with any $\psi$ classes in the excess intersection formula. Consequently, $C$ contributes only to the $H^{15} \otimes H^0$ K\"unneth component.
        In particular, all $\L'$ classes have vanishing $H^{13} \otimes H^2$ K\"unneth component. We claim that also all $\D'$ classes pull back to $G_{X'}$ trivially. Indeed, by the above, the pullback of $\D'_{B,ij}$ equals the pullback of $-24 \pi_{[n] \setminus A}^\ast(\S^1_{B, \emptyset, ij})$ to $G_{X'}$, where $A = B \cup \{i,j\}$. Now, consider the commutative diagram
        \begin{equation*}
        \begin{tikzcd}
        {\Mbar_{2,([n]\setminus X') \cup \{s\}} \times \Mbar_{0,X' \cup\{t\}}} \arrow[rr] \arrow[d, "\eta"'] &  & {\Mbar_{2,(A \setminus X') \cup\{s\}} \times \Mbar_{0, (A \cap X') \cup \{t\}}} \arrow[d, "\eta'"] \\
        {\Mbar_{2,n}} \arrow[rr, "{\pi_{[n]\setminus A}}"] &  & {\Mbar_{2,A}}.
        \end{tikzcd}
        \end{equation*}
        The horizontal arrows are forgetful morphisms and the vertical arrows glue the markings $s$ and $t$. In particular pulling back to $G_{X'}$ means to single out the $H^{13} \otimes H^2$-Künneth component of $\eta^\ast(-)$. Now, in order to show that the $H^{13} \otimes H^2$-Künneth component of $\eta^\ast(\pi_{[n] \setminus A}^\ast(\S^1_{B,\emptyset,ij}))$ vanishes, it suffices to prove the vanishing of the $H^{13} \otimes H^2$-Künneth component of $(\eta')^\ast(\S^1_{A,\emptyset,ij})$. This, in turn, follows from the second part of Lemma~\ref{lemma:genus2-0bdryDiv}.\\
        Next, since we have eliminated all the natural generators $\N'_{B,X,S}$ and $\S^1_{B,S,X}$ with $|X| = 2$ in the previous steps, we can now apply Lemma~\ref{lemma:genus2-0bdryDiv}, first with $|X'| = 3$, to show the linear independence of all the $\N'_{B,X,S}$ and $\S^1_{B,S,X}$ with $|X| = 3$ modulo the span of all remaining natural generators. For this, we consider the first Künneth components of their pullbacks and use \cite[Lemma~4.12]{CLPW24polynomial}, which states that there are no relations among the $Z_{B \subset A}$ classes in genus $2$. Then, having eliminated all $\N'_{B,X,S}$ and $\S^1_{B,S,X}$ with $|X| \leq 3$, we may apply Lemma~\ref{lemma:genus2-0bdryDiv} with $|X'| = 4$ and deduce the linear independence of all the $\N'_{B,X,S}$ and $\S^1_{B,S,X}$ with $|X| = 4$ etc. In this way, we successively eliminate all the $\N'$ and $\S^1$ classes.
        \item Let us pull back to \GraphKunnethThirteenTwo[.75] with $|X'| \in \{1,2\}$. In the previous steps, we have eliminated all $\S^0_{B,X, \varnothing}$ classes, all $\L_{B,X}$ with $|X| \leq 2$, and all $\N'$ and $\S^1$ classes. Moreover, all $\L'$ and $\D'$ classes pull back to zero by Lemmas~\ref{lemma:vanishingGenus2} and~\ref{lemma:D'vanishingGenus2}. Applying Lemma~\ref{lemma:intersectionTableH13xH2} with $|X'| = 1$ gives the linear independence of the $\S^0_{B,X,X'}$ modulo all remaining generators. Running over all $X'$ of cardinality $|X'| = 1$, we thus eliminate all the $\S^0_{B,X,S}$ with $|S| = 1$. The same reasoning then also works for $|X'| = 2$, so we can eliminate all the $\S^0_{B,X,S}$ classes for $|S| \leq 2$.
\end{enumerate}
\noindent
In the final steps we alternate between pulling back to  \GraphKunnethElevenFour[.75] and \GraphKunnethThirteenTwo[.75], first running over all $X'$ with $|X'| = 3$ for the $H^{11} \otimes H^4$-Künneth component, then over the same set for the $H^{13} \otimes H^2$-Künneth component, afterwards running over all $X'$ with $|X'| = 4$ for the $H^{11} \otimes H^4$-Künneth component, then over the same set for the $H^{13} \otimes H^2$-Künneth component etc. In the process we will eliminate the remaining natural generators on the left-hand side of \eqref{eq:rel}: the $\L'_{B,X}$ with $|X| \geq 3$, the $\D_{B,X}$ with $|X| \geq 3$, and the $\S^0_{B,X,S}$ with $|S| \geq 3$.
\begin{enumerate}
        \item[(8)] 
        We start by pulling back to \GraphKunnethElevenFour[.75] with $|X'| = 3$. All $\D'$ classes pull back to zero by Lemma~\ref{lemma:D'vanishingGenus2} and we have eliminated all $\N'$ and $\S^1$ classes as well as all $\L_{B,X}$ ($|X| \leq 2$) and all $\S^0_{B,X,S}$ classes with $|S| \leq 2$. By Lemma~\ref{lemma:intersectionTableH11xH4}, the $\S^0_{B,X,S}$ with $|S| \geq 3$ pull back to zero. Moreover, by Lemma~\ref{lemma:independenceGenus1}, all other remaining natural generators pull back to zero except for the $\D_{B,X'}$ and the $\L'_{B,X'}$. By the same lemma, the latter two classes have linearly independent second Künneth factors, and their first Künneth factors run through the same basis of $H^{11}(\Mbar_{1,([n] \setminus X') \cup s})$ as we vary $B$. This proves the linear independence of all the $\D_{B,X}$ and $\L'_{B,X}$ for $|X| = 3$.
        \item[(9)] Now we pull back to  \GraphKunnethThirteenTwo[.75] with $|X'| = 3$. As before, by Lemmas~\ref{lemma:vanishingGenus2} and~\ref{lemma:D'vanishingGenus2}, all $\D'$ pull back to zero and so do all the remaining $\L'_{B,X}$ with $|X| \geq 4$. Since we have in particular eliminated all the $\S^0_{B,X,S}$ with $|S| \leq 2$, Lemma~\ref{lemma:intersectionTableH13xH2} shows that the only remaining natural generators which do not pull back to zero are precisely the $\S^0_{B,X,X'}$, and the first Künneth components of their pullbacks are linearly independent by Lemma~\ref{lemma:basisGenus1}. Running over all $X'$ with $|X'| = 3$, this proves that the $\S^0_{B,X,S}$ with $|S| = 3$ are linearly independent modulo the span of all other remaining natural generators.
\end{enumerate}
The last two steps can now be repeated with $|X'| = 4$ instead of $3$: The $H^{11} \otimes H^4$-Künneth component first shows that all the $\D_{B,X'}$ and the $\L'_{B,X'}$ are linearly independent modulo the span of all other remaining natural generators, thus eliminating all the $\D_{B,X}$ and $\L'_{B,X}$ with $|X| = 4$. Then, using that all the $\S^0_{B,X,S}$ with $|S| \leq 3$ have already been eliminated, the $H^{13} \otimes H^2$-Künneth component shows the linear independence of the $\S^0_{B,X,X'}$ modulo all other remaining natural generators. Running over all $X'$ with $|X'| = 4$ thus eliminates all the $\S^0_{B,X,S}$ with $|S| = 4$ as well. We may continue in this way, successively increasing the cardinality of $X'$. In the end, we have eliminated all natural generators on the left-hand side of \eqref{eq:rel}, thus showing that they are independent modulo the $\D'$ classes.
\end{proof}

The following three lemmas compute K\"unneth components of pullbacks of unmodified generators to boundary divisors. These formulas are obtained by computing generic $(\Gamma, \Gamma')$-graphs and applying the excess intersection formula, as in \cite[Appendix A]{GraberPandharipande}.

\begin{lemma}\label{lemma:intersectionTableH13xH2}
    Intersection table for \GraphKunnethThirteenTwo[.75]. If an unmodified generator is not listed, then it pulls back to zero.

\begingroup
\allowdisplaybreaks[4]
\begin{alignat*}{3}
\text{1.}
    &\quad \mathcal S^0_{B,X,S}
    &&\qquad \BoundaryCaseOneSource[.76]
    &&\qquad \longrightarrow\quad \BoundaryCaseOneTarget[.76]
    \\[.75em]
\text{2.}
    &\quad \mathcal S^0_{B,X,X'}
    &&\qquad \BoundaryCaseTwoSource[.76]
    &&\qquad \longrightarrow\quad \BoundaryCaseTwoTarget[.76]
    \\[.75em]
\text{3.}
    &\quad \mathcal S^0_{B,X,S}
    &&\qquad \BoundaryCaseThreeSource[.76]
    &&\qquad \longrightarrow\quad \BoundaryCaseThreeTarget[.76]
    \\[.75em]
\text{4.}
    &\quad \mathcal N'_{B,X,X'}
    &&\qquad \quad \BoundaryCaseFourSource[.76]
    &&\qquad \longrightarrow\quad
      \BoundaryCaseFourTarget[.76]\quad(q\mapsto s)
    \\[.75em]
\text{5.}
    &\quad \mathcal N'_{B,X,S}
    &&\qquad \BoundaryCaseFiveSource[.76]
    &&\qquad \longrightarrow\quad \BoundaryCaseFiveTarget[.76]
    \\[.75em]
\text{6.}
    &\quad \mathcal S^1_{B,S,X}
    &&\qquad ~~ \BoundaryCaseSixSource[.76]
    &&\qquad \longrightarrow\quad \BoundaryCaseSixTarget[.76]
    \\[.75em]
\text{7.}
    &\quad \mathcal S^1_{B,S,X}
    &&\qquad ~ \BoundaryCaseSevenSource[.76]
    &&\qquad \longrightarrow\quad \BoundaryCaseSevenTarget[.76]
    \\[.75em]
\text{8.}
    &\quad \mathcal L_{B,X'}
    &&\qquad \quad ~ \BoundaryCaseEightSource[.76]
    &&\qquad \longrightarrow\quad \BoundaryCaseEightTarget[.76]
    \\[.75em]
\text{9.}
    &\quad \mathcal L_{B,X}
    &&\qquad \quad  \BoundaryCaseNineSource[.76]
    &&\qquad \longrightarrow\quad \BoundaryCaseNineTarget[.76]
\end{alignat*}
\endgroup  
\end{lemma}

\begin{lemma}\label{lemma:intersectionTableH11xH4}
    Intersection table for  \GraphKunnethElevenFour[.75]. If an unmodified generator is not listed, then it pulls back to zero.
    \SecondBoundaryCaseList
\end{lemma}

\begin{lemma}\label{lemma:genus2-0bdryDiv}
    We have
\ProductCaseDisplay
    and
\SecondProductCaseDisplay
\end{lemma}

The next three lemmas compute K\"unneth components of pullbacks of the modified $\D'$ and $\L'$ generators to boundary divisors.

\begin{lemma}\label{lemma:vanishingGenus2}
    Let $|X| \geq 3$ and $|X| > |X'|$ or $|X| = |X'|$ but $X \neq X'$. Then:
    \begin{equation*}
        \L'_{B,X} \times \GraphKunnethThirteenTwo[.75]=0
    \end{equation*}
\end{lemma}
\begin{proof}
    This is immediate for $X' \nsubseteq X.$ If $X'\subsetneq X$, then one computes:
    \LPrimeProductDisplay
    \noindent The sum on the right of the tensor product is 0 by the relation in \cite[Proposition~1.9]{AC98}.
\end{proof}

\begin{lemma}\label{lemma:D'vanishingGenus2}
    For all $B, i, j$, and $X'$, we have that:
    \begin{equation*}
        \D'_{B,\{i,j\}} \times \GraphKunnethThirteenTwo[.75] =0 \qquad \text{and} \qquad \D'_{B,\{i,j\}} \times \GraphKunnethElevenFour[.75]=0
    \end{equation*}
\end{lemma}
\begin{proof}
    Consider the following commuting diagram, where the horizontal maps are forgetful maps and the vertical maps glue $s$ and $t$ together:
    \begin{center}
    \begin{tikzcd}
    {\Mbar_{1,([n]\setminus X') \cup \{s\}} \times \Mbar_{1,X' \cup\{t\}}} \arrow[rr] \arrow[d, "\eta"'] &  & {\Mbar_{1,(A \setminus X') \cup\{s\}} \times \Mbar_{1, (A \cap X') \cup \{t\}}} \arrow[d, "\eta'"] \\
    {\Mbar_{2,n}} \arrow[rr, "{\pi_{[n]\setminus A}}"]                                                   &  & {\Mbar_{2,A}.}                                       \end{tikzcd}
    \end{center}
    Let $A = B \cup \{i,j\}$.
    Since $\D'_{B, ij} = \pi^*_{[n]\setminus A}(M_{ij})$, if we can show that $\eta'^*(M_{ij}) = 0$, then the lemma holds. To this end, we compute the pullback of each term in $M_{ij}$.
    
    For $H^{13} \otimes H^2$, it is straightforward to see that $\eta'^*(M_{ij})$ can be nonzero only in the cases where $A \cap X' = \{i,j\}, \{i\}, \{j\},$ or $\emptyset$. We treat each of these cases in turn.
    \begin{enumerate}
    \item When $A \cap X'=\{i,j\}$,
        \[
(\eta')^*\!\left(
  \GraphEtaPrimeDoubleEdge[.90]
\right)=0.
\]
    \noindent Other than the double edge graph, each term in $M_{ij}$ pulls back to the same underlying graph, but with the decoration $\omega_{B \cup \{s\}} \cdot (-\psi_s)$ on the genus-1 vertex with 11 markings, e.g.: 
        \[
(\eta')^*\!\left(
  \GraphEtaPrimeChainSource[.90]
\right)
=
\GraphEtaPrimeChainTarget[.90].
\]

\noindent As $H^{13}(\Mbar_{1,11}) = 0$, each pullback is 0. 
        \item When $A \cap X' = \{i\}$, we have:

\EtaPrimeMijPullbackDisplay
        
        \noindent The last expression is 0 by Equation \ref{eq:ACgenus1}.
        The same reasoning holds for $A \cap X' = \{j\}$, with $i$ and $j$ swapped.
        \item When $A \cap X' = \emptyset$, we have that 

        \EtaPrimeMijPairPullbackDisplay
        \noindent The last expression is 0 by Equation \ref{eq:ACgenus1}.
    \end{enumerate}
    
    \noindent For $H^{11} \otimes H^4$, it is straightforward to see that $\eta'^*(M_{ij})$ can be nonzero only in the case where $A \cap X' = \{i,j\}$.
    In this case, we have

    \EtaPrimeMijFourTermPullbackDisplay
    \noindent Using \texttt{admcycles}, the second tensor factor is $0$ in $H^4(\Mbar_{1,3}).$
\end{proof}

\begin{lemma}\label{lemma:independenceGenus1}
    For any $X$ with $|X| \geq |X'| \geq 3$, we have:
    
        \LPrimeElevenFourProductDisplay
    \noindent and

        \DoubleEdgeElevenFourProductDisplay
    where $\alpha$ and $\beta$ are classes in $H^4(\Mbar_{1, X' \cup \{t\}})$ defined by: 

    \AlphaBetaDefinitionsDisplay

    \noindent Moreover, $\alpha$ and $\beta$ are both nonzero and linearly independent in $H^4 (\Mbar_{1, X' \cup \{t\}})$.
\end{lemma}

\begin{proof}
The formulas follow from direct computation (similarly to the previous lemmas).
To see that $\alpha$ and $\beta$ are linearly independent, let $y \in X'$ and consider the map
\[
\zeta: \Mbar_{1, \{y,t\} \cup \{a\}} \times \Mbar_{0, (X'\setminus y) \cup \{b\}} \rightarrow \Mbar_{1, X' \cup \{t\}}.
\]
The $H^4 \otimes H^0$ parts of $\zeta^*(\alpha)$ and $\zeta^*(\beta)$ are

\EtaStarAlphaBetaDisplay

\noindent One can check using \texttt{admcycles} that these are linearly independent in $H^4(\Mbar_{1,3})$.
\end{proof}

\section{Proof of Theorem \ref{thm:rels}}\label{section:anyg}

In this section, we prove the independence of the decorated graph generators in Figure \ref{fig:gentable} satisfying (R1)--(R10). 
The restrictions (R1)--(R10) are constructed so that no trivial relations, pushed forward relations, or pulled back relations exist among the remaining classes. Therefore, if the remaining classes are independent, there are no more relations.

The basic strategy for proving the independence of these remaining generators is to use induction on the genus and pullbacks along the self-gluing map $\Mb_{g-1,n+2} \to \Mb_{g,n}$. We have already proved Theorem \ref{thm:rels}(a), which covers the $g = 2$ case, in Theorem
\ref{theorem:Genus2AnyMarkings}.

\subsection{Pullback formulas}\label{sec:pullbackforms}
The first step of the proof is to obtain formulas for the pullbacks of our generators along the self-gluing map
$$
\xi \colon \Mb_{g-1,[n] \cup \{x,y\}} \to \Mb_{g,[n]}
$$
that glues $x$ and $y$. The pullback map is well-defined at the level of decorated graph generators, and we will write down the pullbacks explicitly. We use the convention that if any vertex of a decorated graph has negative genus or is unstable, then the graph is $0$.

We order the set $\{x,y\} \cup \{1, \ldots , n\}$ according to
\begin{equation}\label{eq:order}
    x <y < 1< \cdots < n\,.\tag{$\star$}
\end{equation}

Also, it will be convenient to use the symbols
\begin{equation*}
    \M_{B', X, Y}^a := -\M_{B', Y, X}^{g-1-a} \qquad \text{and} \qquad \N_{B, X, Y}^a := \N_{B, Y, X}^{g-1-a}
\end{equation*}
when $a > \frac{g-1}{2}$. By the usual pullback formulas and the self-intersection formula of \cite{GraberPandharipande}, we have
\begin{align*}
\xi^*\S^a_{B,X,Y} &= \S^a_{B, X, Y \cup \{x,y\}} + \S^{a-1}_{B, X \cup \{x, y\}, Y} \\
\xi^*\M^a_{B',X,Y} &= \M^{a-1}_{B',X \cup \{x, y\}, Y} + \M^a_{B',X,Y \cup \{x,y\}}  \qquad (a, X) \neq (1, \emptyset)\\
\xi^* \M_{B', \emptyset, Y}^1 &=\M_{B', \emptyset, Y \cup \{x,y\}}^1 + \sum_{X\subset Y^c \setminus B'}\big(\S^0_{B'\cup \{x\}, X\cup \{y\}, Y}+\S^0_{B'\cup \{y\}, X\cup \{x\}, Y}\big)\\
&\quad -\sum_{\emptyset \neq X\subset Y^c \setminus B'} \M^0_{B', X\cup \{x,y\}, Y}\\
\xi^* \N_{B, X, Y}^a  &= \N_{B, X \cup \{x,y\}, Y}^{a-1} + \N^a_{B, X, Y \cup \{x,y\}}  \quad \\
\xi^* \L_{B, X} &= \L_{B , X\cup \{x,y\}} - 2\P_{B, X \cup \{x, y\}}^x - 2\P_{B, X \cup \{x, y\}}^y + 2\sum_{\substack{\{\alpha, \beta\} = \{x,y\}\\g_1 + g_2 = g-2\\ X_1 \cup X_2 =X}}\S_{B, X_1 \cup\{\alpha\}, X_2 \cup \{\beta\}}^{g_1} \quad(g\geq 4)\\
\xi^* \D_{B, X} &= \D_{B, X\cup \{x,y\}}  + \text{ additional terms}\,, \quad \text{ additional terms} \in \operatorname{span} (\S, \M, \N, \P) \quad (g\geq 4)\\
\xi^* \P_{B,X}^i &= \P_{B, X \cup \{x,y\}}^i\quad(g\geq 4) \\
\xi^* \K_{B,X} & = \K_{B, X \cup \{x,y\}} \quad (g\geq 5)
\end{align*}
The ordering \eqref{eq:order} ensures that the terms $\D_{B, X \cup \{x,y\}}$ and $\N^{a-1}_{B, X\cup \{x,y\},Y}$ are part of the preferred set of generators of Figure \ref{fig:gentable}.

In $\Mbar_{2, [n] \cup \{x,y\}}$, the $\mathcal P^i_{B,X}$ classes are not part of our preferred set of generators in Figure \ref{fig:gentable}. If we rewrite them using the relation \eqref{push1}, we obtain the following formulas for pullbacks in genus $3$ in terms of our preferred generating set:
\begin{align*}
    \xi^* \L_{B, X} &= \frac{5}{6}\L_{B, X \cup \{x,y\}} - 2\sum_{Y \subset X \cup \{x,y\}} |Y \cap \{x,y\}|\big( \S_{B,Y, X \setminus Y}^0 + \S^1_{B, X \setminus Y, Y}\big) \\
    &\qquad + 2\sum_{\substack{\{\alpha, \beta\} = \{x,y\}\\g_1 + g_2 = 1\\ X_1 \cup X_2 =X}}\S_{B, X_1 \cup\{\alpha\}, X_2 \cup \{\beta\}}^{g_1} \\
    &=\frac{5}{6}\L_{B, X \cup \{x,y\}} - 4\sum_{Y \subset X}\big( \S_{B,Y\cup \{x,y\}, X \setminus Y}^0 + \S^1_{B, X \setminus Y, Y\cup \{x,y\}}\big)\\
    \xi^*\P^i_{B, X} &=\frac{1}{24}\L_{B, X \cup \{x,y\}} + \sum_{i \in Y \subset X \cup \{x,y\}} \big(\S_{B, Y, X\cup \{x,y\}\setminus Y}^0 + \S_{B, X\cup \{x,y\}\setminus Y,Y}^1\big) \quad (i \neq q)\\
    \xi^*\P^q_{B, X} &=\frac{1}{24}\L_{B, X \cup \{x,y\}} + \sum_{Y \subset X \cup \{x,y\}} \S_{B, Y, X\cup \{x,y\}\setminus Y}^0 \\
    \xi^* \D_{B, X} &= \D_{B, X\cup \{x,y\}}  + \text{additional terms}\,, \quad \text{additional terms} \in \operatorname{span} (\S, \M, \N)
\end{align*}

It will be important later that in the formula for $\xi^*\L_{B,X}$, all the $\S$ terms that appear with non-zero coefficient have the markings $x$ and $y$ on the same vertex.

The $\K_{B,X}$ pullback is not in the preferred list of generators of Figure \ref{fig:gentable} on $\Mbar_{3, n \cup \{x,y\}}$. In this case, it is given by \eqref{push3}, yielding the following pullback formula in genus $4$:
\begin{align*}
\xi^*\K_{B,X} &= \sum_{i \in X \cup \{x,y,q\}} \P^i_{B,X\cup \{x,y\}}  + \frac{1}{10} \L_{B,X\cup \{x,y\}} - 
\sum_{X\cup \{x,y\} = X' \sqcup Y} \S^0_{B, X',Y} \\
&\quad -\sum_{X\cup \{x,y\} = X' \sqcup Y} \S^2_{B, X', Y}
+ \frac{7}{5}\sum_{X\cup \{x,y\} = X' \sqcup Y} \S^1_{B, X',Y}
\end{align*}

\subsection{Proof of Theorem \ref{thm:rels} (b).}

The formulas in the previous section express the $\xi^*$ pullbacks of the classes in the distinguished basis from Figure \ref{fig:gentable} for $\Mbar_{g, n}$ in terms of the distinguished basis for $\Mbar_{g-1, [n] \cup \{x,y\}}$ with the ordering \eqref{eq:order}. By arranging them appropriately and choosing projections onto different subspaces, (part of) the pullback matrix can be arranged as in Table \ref{tab:pullbackg3} for $g=3$, or Table \ref{tab:pullbackg4} for $g\geq 4$.
\begin{table}[htbp]
\centering
\caption{The matrix of the pullback $\xi^*  : S^{13,2}(\Mbar_{3,n}) \to S^{13,2}(\Mbar_{2,[n] \cup \{x,y\}})$, projected onto appropriate subspaces. The $\mathbbm{1}$ indicate that the identity matrix can be arranged under the natural ordering that matches the indexing in the symbols indicated in the rows and columns}
\label{tab:pullbackg3}
 
\renewcommand{\arraystretch}{2.2}
\setlength{\tabcolsep}{4pt}
 
\begin{tabular}{|c|*{8}{c|}}
    \hline
    & $\D_{B,X}$
    & $\M_{B',\emptyset,Y}^1$
    & $\begin{gathered}
        \M_{B',X,Y}^a\\
        (a,X)\neq(1,\emptyset)
    \end{gathered}$
    & $\P^q_{B,X}$
    & $\begin{gathered}
        \P_{B,X}^i\\
        (i \neq q)
    \end{gathered}$
    & $\L_{B,X}$
    & $\N_{B,X,Y}^a$
    & $\S_{B,X,Y}^a$ \\
    \hline
 
    $\D_{B, X\cup\{x,y\}}$
    & $\one$ & $0$ & $0$ & $0$ & $0$ & $0$ & $0$ & $0$ \\
    \hline
 
    $\S_{B'\cup \{x\}, \{y\}, Y}^0$
    & $\star$ & $\one$ & $0$ & $0$ & $0$ & $0$ & $0$ & $0$ \\
    \hline
 
    $\M_{B', X, Y \cup \{x,y\}}^a$
    & $\star$ & $\star$ & $\one$ & $0$ & $0$ & $0$ & $0$ & $0$ \\
    \hline
 
    $\S_{B, \{x\}, X \cup \{y\}}^0$
    & $\star$ & $\star$ & $\star$ & $\one$ & $0$ & $0$ & $0$ & $0$ \\
    \hline
 
    $\S_{B, \{i\}^c \cup \{y\}, \{i,x\}}^1$
    & $\star$ & $\star$ & $\star$ & $\star$ & $\one$ & $0$ & $0$ & $0$ \\
    \hline
 
    $\frac{5}{6}\L_{B, X\cup \{x,y\}}$
    & $\star$ & $\star$ & $\star$ & $\star$ & $\star$ & $\one$ & $0$ & $0$ \\
    \hline
 
    $\N^a_{B, X, Y\cup \{x,y\}}$
    & $\star$ & $\star$ & $\star$ & $\star$ & $\star$ & $\star$ & $\one$ & $0$ \\
    \hline
 
    $\S^{a}_{B, X, Y\cup \{x,y\}}$
    & $\star$ & $\star$ & $\star$ & $\star$ & $\star$ & $\star$ & $\star$ & $\one$ \\
    \hline
\end{tabular}
\end{table}
\newpage

\begin{table}[htbp]
\centering
\caption{The matrix of the pullback $\xi^*  : S^{13,2}(\Mbar_{g,n}) \to S^{13,2}(\Mbar_{g-1, [n] \cup \{x,y\}})$ for $g \geq 4$, projected onto appropriate subspaces. The $\mathbbm{1}$ indicate that the identity matrix can be arranged under the natural ordering that matches the indexing in the symbols indicated in the rows and columns}
\label{tab:pullbackg4}
 
\renewcommand{\arraystretch}{2.2}
\setlength{\tabcolsep}{4pt}
 
\begin{tabular}{|c|*{8}{c|}}
    \hline
    & $\D_{B,X}$
    & $\L_{B,X}$
    & $\K_{B,X}$
    & $\P_{B,X}^i$
    & $\M_{B',\emptyset,Y}^1$
    & $\begin{gathered}
        \M_{B',X,Y}^a\\
        (a,X)\neq(1,\emptyset)
    \end{gathered}$
    & $\N_{B,X,Y}^a$
    & $\S_{B,X,Y}^a$ \\
    \hline
 
    $\D_{B, X\cup\{x,y\}}$
    & $\one$ & $0$ & $0$ & $0$ & $0$ & $0$ & $0$ & $0$ \\
    \hline
 
    $-2\P_{B, X \cup \{x,y\}}^x$
    & $\star$ & $\one$ & $0$ & $0$ & $0$ & $0$ & $0$ & $0$ \\
    \hline
 
    \multicolumn{1}{|@{}c@{}|}{%
    \renewcommand{\arraystretch}{1.4}%
    \begin{tabular}[c]{c}
        $\K_{B,X\cup\{x,y\}}\ (g\geq 5)$\\
        \hline
        $\frac{1}{10}\L_{B,X\cup\{x,y\}}\ (g=4)$
    \end{tabular}%
    }
    & $\star$ & $\star$ & $\one$ & $0$ & $0$ & $0$ & $0$ & $0$ \\
    \hline
 
    $\P_{B, X\cup\{x,y\}}^i$
    & $\star$ & $\star$ & $\star$ & $\one$ & $0$ & $0$ & $0$ & $0$ \\
    \hline
 
    $\S_{B'\cup \{x\}, \{y\}, Y}^0$
    & $\star$ & $\star$ & $\star$ & $\star$ & $\one$ & $0$ & $0$ & $0$ \\
    \hline
 
    $\M_{B', X, Y \cup \{x,y\}}^a$
    & $\star$ & $\star$ & $\star$ & $\star$ & $\star$ & $\one$ & $0$ & $0$ \\
    \hline
 
    $\N_{B, X, Y\cup \{x,y\}}^{a}$
    & $\star$ & $\star$ & $\star$ & $\star$ & $\star$ & $\star$ & $\one$ & $0$ \\
    \hline
 
    $\S^{a}_{B, X, Y\cup \{x,y\}}$
    & $\star$ & $\star$ & $\star$ & $\star$ & $\star$ & $\star$ & $\star$ & $\one$ \\
    \hline
\end{tabular}
\end{table}

The vanishings to the right of each diagonal box are clear by examining the pullback formulas in \sectionsign \ref{sec:pullbackforms}. For example, the fifth row in Table \ref{tab:pullbackg3} has zeros to the right of the diagonal because the pullback formulas for $\L, \N, \S$ classes always have $x$, $y$ in the same vertex.

That each diagonal block can be arranged as the identity also follows from a careful examination of the pullback formulas. We give two examples to illustrate this. In the seventh diagonal block of Table \ref{tab:pullbackg3}, the decorated graph $\mathcal N_{B_1,X_1,Y_1 \cup \{x,y\}}^{a_1}$ appears in the pullback of $\mathcal N_{B_2, X_2, Y_2}^{a_2}$ if and only if $a_1 = a_2$, $B_1= B_2$, $X_1 = X_2$, and $Y_1 = Y_2$. This is because $\mathcal N_{B_2, X_2, Y_2}^{a_2}$ is uniquely recoverable from $\mathcal N_{B_2, X_2, Y_2\cup \{x,y\}}^{a_2}$ by locating the vertex with $\{x,y\}$, removing them from the set of markings and increasing the genus of this vertex by $1$. 

Since Tables \ref{tab:pullbackg3} and \ref{tab:pullbackg4} are lower triangular, we conclude that if there are no relations between the natural generators in genus $g-1$, there cannot be relations between the natural generators in genus $g$, for $g \geq 3$. By Theorem \ref{thm:rels} (a), there are no relations for $g=2$ between the preferred set of generators. By induction, this concludes the proof of Theorem \ref{thm:rels} (b).

\newpage
\section{Dimension counts and proof of Theorem \ref{thm:g3}} \label{sec:dimcounts}

By Theorem \ref{thm:rels}, the only relations among the decorated graph generators are \eqref{triv1}--\eqref{push4}, \eqref{push5}--\eqref{rel12pull} and once (R1)--(R10) are imposed, no relations remain among the restricted generators. Thus, in order to determine the dimension of $S_\omega^{13,2}(\Mb_{g,n})$, one simply needs to count the number of decorated graphs of each type satisfying the restrictions (R1)--(R10). We provide explicit formulas in genus $2$ and $3$. It is possible to derive similar formulas in any genus.

In genus $2$, summing the terms below provides a formula for $\dim S_\omega^{13,2}(\Mb_{2,n})$.
\begin{align*}
\#\S^{0}_{B,X,Y} &= \sum_{x=1}^{n} \sum_{y=0}^{n} {n \choose x} {n-x \choose y}{n-x-y \choose 10}\\
\#\S^{1}_{B,X,Y} &= \sum_{x=0}^{n} \sum_{y=2}^{n} {n \choose x} {n-x \choose y}{n-x-y \choose 10} \\
\#\M_{B',X,Y}^0 &= \sum_{x=3}^n\sum_{y=0}^{n} {n \choose x}{n-x \choose y}{n - x - y \choose 9}\\
\#\N_{B,X,Y}^0 &= \sum_{y=0}^n {n \choose y}{n-y \choose 12}11 + \sum_{x=3}^n \sum_{y=0}^n {n \choose x} {n - x \choose y}{n - x - y \choose 10}\\
\#\L_{B,X} &= \sum_{x=1}^n {n \choose x} {n -x \choose 10} \\
\#\D_{B,X} &= \sum_{x=3}^n {n \choose x} {n-x \choose 10} + {n \choose 12}11  \\
\# \P^i_{B,X} &= 0\\
\#\K_{B,X} &= 0
\end{align*}

In genus $3$, summing the terms below provides a formula for $\dim S_\omega^{13,2}(\Mb_{3,n})$.
{\allowdisplaybreaks
\begin{align*}
\#\S^{0}_{B,X,Y} &= \sum_{x=1}^{n} \sum_{y=0}^{n} {n \choose x} {n-x \choose y}{n-x-y \choose 10}\\
\#\S^{1}_{B,X,Y} &= \sum_{x=0}^{n} \sum_{y=0}^{n} {n \choose x} {n-x \choose y}{n-x-y \choose 10} \\
\#\S^{2}_{B,X,Y} &= \sum_{x=0}^{n} \sum_{y=2}^{n} {n \choose x} {n-x \choose y}{n-x-y \choose 10} \\
\#\M_{B',X,Y}^0 &= \sum_{x=3}^n\sum_{y=0}^{n} {n \choose x}{n-x \choose y}{n - x - y \choose 9}\\
\#\M_{B',X,Y}^1 &= \frac{1}{2} \left(\sum_{x=0}^n \sum_{y=0}^n{n \choose x}{n -x \choose y}{n - x - y \choose 9}\right) - \frac{1}{2}{n \choose 9} \\
\#\N_{B,X,Y}^0 &= \sum_{y=0}^n {n \choose y}{n-y \choose 12}11 + \sum_{x=3}^n \sum_{y=0}^n {n \choose x} {n - x \choose y}{n - x - y \choose 10}\\
\#\N_{B,X,Y}^1 &= \frac{1}{2} \left(\sum_{x=0}^n \sum_{y=0}^n{n \choose x}{n -x \choose y}{n - x - y \choose 10} \right) + \frac{1}{2} {n \choose 10} \\
\#\L_{B,X} &= \sum_{x=0}^n {n \choose x} {n -x \choose 10} \\
\#\D_{B,X} &= \sum_{x=0}^n {n \choose x} {n-x \choose 10} \\
\# \P^i_{B,X} &= \sum_{x=0}^{n} {n \choose x}{n-x \choose 10}(x+1) \\
\#\K_{B,X} &= 0
\end{align*}
}

We summarize the results of these formulas for low values of $n$ in the following tables.
\begin{table}[h]
\caption{The dimensions of $S_\omega^{13, 2}(\Mbar_{g,n})$ for $g =2,3$ and $n \leq 20$.}
\label{fig:dimensions}
\begin{center}
\begin{tabular}{r|r}
$n$ & $\dim S_\omega^{13,2}(\Mb_{2,n})$ \\
 \hline
 \hline
 $\leq 9$ & 0 \\
10 & 0 \\
11 & 22 \\
12 & 836 \\
13 & 16874 \\
14 & 234234 \\
15 & 2532530 \\
16 & 22902308 \\
17 & 181094914 \\
18 & 1289698982 \\
19 & 8446397674 \\
20 & 51653049526
\end{tabular}
\hspace{1in}
\begin{tabular}{r|r}
$n$ & $\dim S_\omega^{13,2}(\Mb_{3,n})$ \\
 \hline
 \hline
 $\leq 9$ & 0 \\
    10 & 15 \\
     11 & 363 \\
     12 & 5467 \\
     13 & 65065 \\
     14 & 657657 \\
     15 & 5832827 \\
     16 &  46381192 \\
     17 & 336358022 \\
     18 & 2256303478 \\
     19 & 14165150142 \\
     20 & 84038634836
\end{tabular}
\end{center}
\end{table}

The dimension formulas above agree with the data of Bergstr\"om--Faber \cite{BergstromData}, which was obtained from point counts over finite fields.

\begin{proof}[Proof of Theorem \ref{thm:g3}]
Assuming that $\dim H^{13,2}(\Mb_{3,14}) = 657657$ as in the statement of the theorem, the fact that $\dim S_\omega^{13,2}(\Mb_{3,14}) = 657657$ from the table above shows that
$H^{13,2}(\Mb_{3,14}) = S_\omega^{13,2}(\Mb_{3,14})$. By Lemma \ref{132reduction}, we have $H^{15}(\Mb_{3,14}) = S_\omega^{15}(\Mb_{3,14})$. The theorem for all $n$ now follows from the reduction to $n = 14$ explained in Section \ref{sec:red}. 
\end{proof}

\section{Proof of Theorem \ref{thm:rh}}\label{section:cohomology-is-taut}

In this section, we prove Theorem \ref{thm:rh} that $H^*(\Mb_{g,n}) = RH^*(\Mb_{g,n})$ for various small values of $g$ and $n$. Let $(g,n)$ be as in Theorem \ref{thm:rh}.
It was shown in \cite[Table 2]{CanningRarelyGorenstein} that $H^k(\Mb_{g,n}) = RH^k(\Mb_{g,n})$ for all even $k$. Thus, our task is to prove that the odd degree cohomology of $\Mb_{g,n}$ vanishes in these cases. First, we demonstrate that the odd cohomology is pushed forward from the boundary, using results on the Chow--K\"unneth generation property, which we now recall.

\begin{definition}~\cite[Definition~4.1]{CLP24extensions}
    Let $X$ be a smooth algebraic stack of finite type over a field, stratified by quotient stacks. We say that $X$ has the \textit{Chow--Künneth generation Property} (CKgP) if for all algebraic stacks $Y$ (of finite type, stratified by quotient stacks) the exterior product
    \[
    A_*(X) \otimes A_*(Y) \to A_*(X \times Y)
    \]
    is surjective, where $A_*(X)$ and $A_*(Y)$ are the Chow groups of $X$ and $Y$ respectively. 
\end{definition}

If $X$ is a smooth Deligne--Mumford stack with the CKgP, then the
cycle class map
    \[
    \mathrm{cl}: \bigoplus_i A_i(X) \to \bigoplus_k W_kH^k(X)
    \]
    is surjective by \cite[Lemma~4.3]{CLP24extensions}. In particular, if $k$ is odd then $W_kH^k(X)=0$.

The following table records the currently known $(g,n)$ such that $\M_{g,n}$ has the CKgP and $A^*(\M_{g,n}) = R^*(\M_{g,n})$. The cases with \(g\leq 7\), apart from the updated genus-\(5\)
entry, are recorded in \cite[Table~1]{CLP24extensions}. The cases
\(\M_{5,8}\) and \(\M_{5,9}\) follow from
\cite[Theorem~1.3]{LiuM58M59}. Finally, the arguments of
\cite{CL24chowMg} showing the Chow ring is tautological also establish the CKgP for \(\M_8\) and
\(\M_9\), as recorded in
\cite[Remark~3.1 and Section~3.3]{Lsurvey}.

\begin{table}[ht]
\centering
\begin{tabular}{|c|c|c|c|c|c|c|c|c|c|c|}
\hline
\(g\)    & \(0\)      & \(1\)  & \(2\)  & \(3\)  & \(4\)
         & \(5\)      & \(6\)  & \(7\)  & \(8\)  & \(9\) \\
\hline
\(c(g)\) & \(\infty\) & \(10\) & \(10\) & \(11\) & \(11\)
         & \(9\)      & \(5\)  & \(3\)  & \(0\)  & \(0\) \\
\hline
\end{tabular}
\caption{$\M_{g,n}$ has the CKgP and $A^*(\M_{g,n}) = R^*(\M_{g,n})$ for $n \leq c(g)$.}
\label{tab:c-values}
\end{table}

For $0 \leq g \leq 9$ and $n \leq c(g)$ and $k$ even, it follows that $W_kH^k(\M_{g,n})$ is generated by tautological classes. Meanwhile, for $0 \leq g \leq 9$ and $n \leq c(g)$ and $k$ odd, we have
\begin{equation}  \label{eq:oddsurj}
H^{k-2}(\widetilde{\partial \M}_{g,n}) \to H^{k}(\Mb_{g,n}) \to W_{k}(\M_{g,n}) = 0.
\end{equation}
In particular, all odd degree cohomology is pushed forward from the boundary.
Since $H^{13}(\Mb_{g,n}) = S^{13}_\omega(\Mb_{g,n})$ for all $g$ and $n$, \eqref{eq:oddsurj} with $k = 15$ immediately yields the following.

\begin{lemma} \label{lem:h=s}
If $0 \leq g \leq 9$ and $n \leq c(g)$, then $H^{15}(\Mb_{g,n}) = S_{\omega}^{15}(\Mb_{g,n})$.
\end{lemma}

The following lemma explains how repeated application of \eqref{eq:oddsurj} allows us to reduce Theorem \ref{thm:rh} to the cases $(g,n) = (3,9), (4,8)$ and $(5,7)$.

\begin{lemma}\label{lem: implies taut for many g,n}
Suppose $H^*(\Mb_{g,n}) = RH^*(\Mb_{g,n})$ for $(g,n) = (3,9), (4,8)$ and $(5,7)$. Then $H^*(\Mb_{g,n}) = RH^*(\Mb_{g,n})$ for all $(g,n)$ in Theorem \ref{thm:rh}.
\end{lemma}

\begin{proof}
  In \cite[$\mathsection$5]{Lsurvey}, it is shown that the fact that the cohomology of $\Mb_{3,9}$ is tautological implies that the cohomology of $\Mb_{g,n}$ is tautological for $(g,n)\in \{(4,7), (5,5), (6,3), (7,1)\}$. We will review this argument supposing the cohomology of $\Mb_{4,8}$ is tautological, and show this implies that the cohomology of $\Mb_{g,n}$ is tautological for $(g,n) \in \{(5,6),(6,4),(7,2),(8,0)\}$.
  
 To prove the case $(g,n) = (5,6)$, we consider the exact sequence
    \[
    H^{k-2}( \widetilde{\partial\M_{5,6}}) \to H^k(\Mb_{5,6}) \to W_kH^k(\M_{5,6}).
    \]
By Table \ref{tab:c-values}, $W_kH^k(\M_{5,6})$ is generated by tautological classes. Thus, it suffices to prove that all classes pushed forward from the boundary are also tautological. We assume $\Mb_{4,8}$ is tautological and so all classes from the irreducible boundary divisor are also tautological. Now consider a separating boundary divisor which is the image of a gluing map
    \[
    \Mb_{g_1,n_1+1} \times \Mb_{g_2, n_2+1} \to \Mb_{5,6}.
    \]
   By the smaller cases of Theorem \ref{thm:rh} already established, every $\Mb_{g_i,n_i+1}$ appearing as a factor of a separating boundary divisor has tautological cohomology. The proofs for the remaining pairs of $(g,n)\in \{(6,4),(7,2),(8,0)\}$ are similar. 

    The remaining $(g,n)$ in Theorem \ref{thm:rh} are proved by the same argument applied to $\Mb_{5,7}$.
\end{proof}

The case $(g,n) = (3,9)$ of Theorem \ref{thm:rh} was already proved in \cite[$\mathsection$5]{Lsurvey}, by observing that $\Mb_{3,9}$ has polynomial point count and hence its odd cohomology vanishes. We briefly explain an alternative proof using the results of this paper.
From earlier results on odd cohomology, one immediately sees that $H^{k}(\Mb_{3,9}) = 0$ for $k \leq 13$ odd.  Meanwhile, Lemma \ref{lem:h=s} shows that $H^{15}(\Mb_{3,9}) = S^{15}_\omega(\Mb_{3,9})$. Then, since there are no generators satisfying the restrictions in Figure \ref{fig:gentable} for $(g,n) = (3,9)$, it follows that $S^{15}_\omega(\Mb_{3,9}) = 0$. Finally, the higher degree odd cohomology vanishes by Poincar\'e duality.

The cases $(g,n) = (4,8)$ and $(5, 7)$ are more involved, but repeated application of \eqref{eq:oddsurj} and Lemma \ref{lem:h=s} help us organize generators for
$H^k(\Mb_{g,n})$ for odd $k$,
building on the generators in Figure \ref{fig:gentable}. These cases are treated in Sections \ref{sec:48} and \ref{sec:57}, respectively.

In order to show that these resulting generators vanish, we make use of the following representation-theoretic lemmas.

\begin{lemma}\label{lem:S2-invariants-H11}
Let $11 \leq n \leq 20$ and $2m \leq n$, and let $\mathbb{S}_2^{m} \subset \mathbb{S}_n$ be the Young subgroup generated by $m$ disjoint transpositions. If $m \geq n-10$ then the invariant subspace $H^{11}(\Mb_{1,n})^{\mathbb{S}_2^m}=0$. 
\end{lemma}
\begin{proof}
By \cite[Proposition~2.3]{CLP23eleventh}, the $\mathbb S_n$-representation underlying each Hodge summand of $H^{11}(\Mb_{1,n})$ is the hook representation
$$
V_{(n-10,1^{10})}
\cong
\bigwedge^{10}\mathrm{Std}_n,
$$
where $\mathrm{Std}_n$ denotes the standard representation of $\mathbb S_n$. Restricting the standard representation to $\mathbb{S}_2^m$ gives
$$
\operatorname{Res}^{\mathbb S_n}_{\mathbb{S}_2^m}\mathrm{Std}_n
\cong
\mathbf{1}^{\oplus(n-m-1)}
\oplus
\chi_1\oplus\cdots\oplus\chi_m,
$$
where $\chi_i$ is the nontrivial character of the $i$th factor of $\mathbb{S}_2^m$. Since the characters $\chi_1,\ldots,\chi_m$ are independent, the product of any nonempty subset of them is nontrivial. Therefore,
$$
\left(\bigwedge^{10}\mathrm{Std}_n\right)^{\mathbb{S}_2^m}
\cong
\bigwedge^{10}\mathbf{1}^{\oplus(n-m-1)}.
$$
If $m\geq n-10$, then $n-m-1\leq 9$, so the right-hand side vanishes.
\end{proof}

\begin{corollary}\label{cor:graph-averaging-vanishing}
Let $\Gamma$ be a stable graph with a genus-one vertex $v$ incident to $n$ half edges. Suppose that a decoration of $\Gamma$ has the form
$$
\gamma=\alpha\otimes\beta,
\qquad
\alpha\in H^{11}(\Mb_{1,n}),
$$
where $\beta\in \bigotimes_{w \in V(\Gamma) \smallsetminus \{v\}} H^{k(w)}(\Mb_{g(w),n(w)} )$ denotes the decorations on all remaining vertices. Assume that $\operatorname{Aut}(\Gamma)$ contains a subgroup
$
K \cong \mathbb{S}_2^m
$
that fixes $v$ and $\beta$, and acts on the half edges incident to $v$ by $m$ disjoint transpositions. If $11 \leq n \leq 20$ and $m\geq n-10$, then the associated decorated graph vanishes upon pushforward along the gluing morphism $\xi_\Gamma$.
\end{corollary}

\begin{proof}
For every \(\sigma\in K\), we have
$
\xi_\Gamma\circ\sigma=\xi_\Gamma.
$
Consequently,
\[\xi_{\Gamma *}(\gamma) = \frac{1}{|K|}
\xi_{\Gamma *}\left(\sum_{\sigma\in K}
\sigma\cdot\gamma\right).\]
Thus, it suffices to show the right hand side vanishes. Since $\beta$ is $K$-invariant, Lemma~\ref{lem:S2-invariants-H11} gives
\begin{equation*}
    \frac{1}{|K|}\sum_{\sigma\in K}\sigma \cdot \gamma= \Bigg(\frac{1}{2^m}\sum_{\sigma\in K}\sigma \cdot \alpha \Bigg)\otimes \beta=0. \qedhere
\end{equation*}
\end{proof}

\subsection{Proving \texorpdfstring{$\Mb_{4,8}$}{Mb(4,8)} is tautological} \label{sec:48}
In this section, we prove that all odd-degree cohomology groups of
\(\Mb_{4,8}\) vanish. Together with the known description of its
even-degree cohomology, this shows that its cohomology is tautological.

\begin{proposition}
    The cohomology ring of $\Mb_{4,8}$ is tautological; that is,
    $$H^*(\Mb_{4,8})=RH^*(\Mb_{4,8}).$$
\end{proposition}
\begin{proof}
    By \cite[Theorem 8]{CanningRarelyGorenstein} the even degree cohomology of $\Mb_{4,8}$ is tautological. It therefore remains to show that its odd degree cohomology vanishes. By \cite{BFP24} and \cite{CLP23eleventh},
    $$H^k(\Mb_{4,8})=0 \qquad \text{for }k\in\{1,3,5, \ldots, 11\}. $$ 
    Since $\dim \Mb_{4,8}=17$, Poincaré duality reduces the remaining odd degree vanishing to degrees, $13,15$, and $17$. 
    
    By \cite[Theorem~1.7]{CLPW24polynomial},
    \(
    H^{13}(\Mbargn)=0 
    \)
    whenever $n<10$.
    In particular $H^{13}(\Mb_{4,8})=0$, so it remains to consider degrees $15$ and $17$. As in the preceding sections, we represent the remaining odd degree classes by decorated graphs and show that these classes vanish. 

Let $k=15$. By Lemma \ref{lem:h=s}, $H^{15}(\Mb_{4,8})=S_\omega^{15}(\Mb_{4,8})$ thus we may refer to Figure \ref{fig:gentable} for a list of generators. When $(g,n)=(4,8)$, there are not enough marked points to realize any of these generators. Therefore $H^{15}(\Mb_{4,8})=0$.

Next, let $k=17$. Equation \eqref{eq:oddsurj} gives a surjection
\[
H^{15}(\widetilde{\partial\M_{4,8}}) \to H^{17}(\Mb_{4,8}).
\]
Each separating boundary divisor is the image of a gluing map $\Mb_{a,A \cup s} \times \Mb_{4-a,A^c \cup t} \to \Mb_{4,8}$ where $|A \cup s|, |A^c \cup t| \leq 9$.
By the Künneth decomposition,
\begin{align}
  H^{15}(\Mb_{a, A \cup s} \times \Mb_{4-a, A^c \cup t}) &=  H^{15}(\Mb_{a, A \cup s} ) \otimes H^{0}(\Mb_{4-a, A^c \cup t} )  \oplus H^{13}(\Mb_{a, A \cup s} ) \otimes H^{2}(\Mb_{4-a, A^c \cup t} ) \label{eq: H^15 of separating boundary for 4,8}\\ & \oplus H^{2}(\Mb_{a, A \cup s} ) \otimes H^{13}(\Mb_{4-a, A^c \cup t} ) \oplus H^{0}(\Mb_{a, A \cup s} ) \otimes H^{15}(\Mb_{4-a, A^c \cup t} ).  \notag
\end{align}

The $H^{13}$-terms vanish because each factor has too few marked points.  Lemma \ref{lem:h=s} applies to both $({a,A \cup s})$ and $(4-a, A^c \cup t)$. Figure \ref{fig:gentable} shows that these factors have too few marked points to support a nonzero generator of $S_\omega^{15}$.

Thus, the only potential nonzero contribution comes from the image of $H^{15}(\Mb_{3,[8]\cup \{s,t\}})$ under the self-gluing map. Here, we identify $s, t$ with the last two markings in $[10]$. By Lemma \ref{lem:h=s}, we obtain the following generating set from Figure \ref{fig:gentable}:
\begin{align}\label{eq:gs}
\begin{split}
    \S^1_{[10],\emptyset,\emptyset}& = \GraphSthreeten \\
    \M^1_{[10]\setminus i,i,\emptyset}&=\GraphMthreeten \\
    \N^1_{[10], \emptyset, \emptyset}&=\GraphNthreeten\\
    \L_{[10],\emptyset}&=\GraphLthreeten\\
    \D_{[10],\emptyset}&=\GraphDthreeten\\
    \P^q_{[10],\emptyset}&=\GraphPthreeten   
\end{split}
\end{align}

\noindent Corollary \ref{cor:graph-averaging-vanishing} now shows the images of all these generators under the self-gluing map
\begin{equation}\label{eq:eta gluing M_3,10 to M_4,8}
\eta_1 \colon \Mb_{3,[8]\cup\{s,t\}}\longrightarrow \Mb_{4,8},
\end{equation}
which glues the marked points \(s\) and \(t\), vanishes.
For example, to show that $(\eta_1)_*(\S^1_{[10],\emptyset,\emptyset}) = 0$, we consider the graph
\begin{align*}
   \Gamma_1= \GraphSthreetenpushforward
\end{align*}
 obtained by gluing $s$ and $t$ on $\S^1_{[10],\varnothing,\varnothing}$ and the decoration  $\gamma_1= \omega_{[10] \cup \{p\}} \otimes 1 \otimes 1$.
Then $\mathrm{Aut}(\Gamma_1) \cong\mathbb{S}_2$ fixes $v$ and $\beta=1 \otimes 1$. Since $1 \geq 11-10$, Corollary \ref{cor:graph-averaging-vanishing} shows $(\eta_1)_*(\S^1_{[10], \emptyset,\emptyset})=0$.
A similar argument is used to show the vanishing of the pushforward under $\eta_1$ of the other generators. Indeed, each of these pushforwards is represented by a decorated graph obtained by adding an edge to one of the graphs in \eqref{eq:gs} between $s$ and $t$. This always results in a graph whose  automorphism group has a subgroup acting in the manner required to apply Corollary \ref{cor:graph-averaging-vanishing}.
\end{proof}

\begin{remark}
This result is sharp in genus $4$ in the sense that odd cohomology classes appear once we add one more marking. Indeed, by Theorem \ref{thm:rels} and Figure \ref{fig:gentable}, we have $\M^1_{[9],\varnothing, \varnothing} \neq 0$, so $H^{15}(\Mb_{4,9}) \neq 0$.
\end{remark}

\subsection{Proving \texorpdfstring{$\Mb_{5,7}$}{Mb(5,7)} is tautological}\label{sec:57}
By Lemma \ref{lem: implies taut for many g,n}, the following proposition is the final case needed to prove Theorem \ref{thm:rh}.

\begin{proposition}
    The cohomology ring of $\Mb_{5,7}$ is tautological; that is, $$H^*(\Mb_{5,7})=RH^*(\Mb_{5,7}).$$
\end{proposition}

\begin{proof}
    By \cite[Theorem~8]{CanningRarelyGorenstein} the even degree cohomology of $\Mb_{5,7}$ is tautological so it remains to show that its odd degree cohomology vanishes. Since $\dim \Mb_{5,7}=19$, Poincaré duality reduces this to proving $H^k(\Mb_{5,7})=0$ for $k=13,15,17,19$. By \cite[Theorem~1.7]{CLPW24polynomial}, we have $H^{13}(\Mb_{5,7})=0$. We treat the remaining degrees case by case.

    Let $k=15$. By Lemma \ref{lem:h=s}, $H^{15}(\Mb_{5,7})=S_\omega^{15}(\Mb_{5,7})$. When $(g,n)=(5,7)$ there are not enough marked points to obtain a nonzero generator in Figure \ref{fig:gentable}.

    Let $k=17$. Equation \eqref{eq:oddsurj} gives the surjection
    \[
    H^{15}(\widetilde{\partial\M_{5,7}}) \to H^{17}(\Mb_{5,7}).
    \]
Each separating boundary divisor is the image of a gluing map $\Mb_{a,A \cup s} \times \Mb_{5-a,A^c \cup t} \to \Mb_{5,7}$ where $|A \cup s|, |A^c \cup t| \leq 8$.
For each separating divisor, the Künneth decomposition, together with the restrictions on the number of marked points, gives
\[
H^{15}(\Mb_{a,A \cup s} \times \Mb_{5-a,A^c \cup t})=0.
\]
Thus the only possible nonzero contribution comes from the image of $H^{15}(\Mb_{4,[7] \cup \{s,t\}})$ under the self-gluing map. By Lemma \ref{lem:h=s}, we refer to Figure \ref{fig:gentable} which gives the single generator
\[
\M^1_{[9], \emptyset,\emptyset}=\GraphMfournine \ ,
\]
where we identify $s$ and $t$ with the last two markings of $[9]$.
Consider its image under the self-gluing map
\begin{equation}\label{eq:eta gluing M_4,9 to M_5,7}
\eta_2: \Mb_{4,[7] \cup \{s,t\} } \to \Mb_{5,7},
\end{equation}
which glues the marked points $s,t$. Applying Corollary \ref{cor:graph-averaging-vanishing} shows the pushforward vanishes in $H^{17}(\Mb_{5,7})$.

Let $k=19$. We have the surjection
\[
H^{17}(\widetilde{\partial\M_{5,7}}) \to H^{19}(\Mb_{5,7}).
\]
We begin by analyzing the separating boundary divisors. 
First consider $\Mb_{a,A \cup s} \times \Mb_{5-a,A^c \cup t} \to \Mb_{5,7}$ with $a = 5$. In this case, $|A^c| \geq 2$, so $|A \cup s| \leq 6$. By the smaller cases of Theorem \ref{thm:rh} already proved, the cohomology of the $\Mb_{5,m}$ with $m \leq 6$ is tautological (see proof of 
Lemma \ref{lem: implies taut for many g,n}). Hence, these separating boundary divisors make no contribution.

Next consider $\Mb_{a,A \cup s} \times \Mb_{5-a,A^c \cup t} \to \Mb_{5,7}$ with $a = 4$. In this case, $|A \cup s| \leq 8$. We already showed that the $\Mb_{4,m}$ with $m \leq 8$ have tautological cohomology, so these separating boundary divisors likewise make no contribution. 

Finally when $a = 3$, \cite[Table~3]{CanningRarelyGorenstein} shows that all odd degree cohomology of $\Mb_{2,n}$ and $\Mb_{3,n}$ vanishes for $n\leq 8$. Consequently, none of the separating boundary divisors contributes to $H^{19}(\Mb_{5,7})$.

It remains to analyze the image of $H^{17}(\Mb_{4,[7]\cup\{s,t\}}) \to H^{19}(\Mb_{5,7})$. Using \eqref{eq:oddsurj}, the map
\[
H^{15}(\widetilde{\partial\M_{4,9}}) \to H^{17}(\Mb_{4,9})
\]
is surjective. The normalization of the boundary is
\begin{align*}
\widetilde{\partial\M_{4,9}}=\bigcup \Mb_{4,Y_1\cup \{r_1\}} \times\Mb_{0,Y_1^c\cup \{r_2\}} \sqcup \bigcup \Mb_{3,Y_2\cup \{v_1\}} \times \Mb_{1,Y_2^c\cup \{v_2\}} \\
\sqcup \bigcup \Mb_{2,Y_3\cup \{w_1\}} \times \Mb_{2,Y_3^c\cup \{w_2\}} \sqcup\Mb_{3,[9]\cup \{x,y\}}/S_2.
\end{align*}
Here, the union is over $Y_i \cup Y_i^c=[9]$, with $0 \leq |Y_1|\leq 7$, and $0 \leq |Y_j| \leq 9$ for $j=2,3$. By Lemma \ref{lem:h=s}, together with the marked point constraints established above, the separating boundary divisors with a genus $0$ component make no contribution. We proceed case by case for the remaining boundary divisors.

First, consider the boundary divisor $\Mb_{3,Y_2\cup \{v_1\}} \times \Mb_{1,Y_2^c\cup \{v_2\}}$, \cite[Table~3]{CanningRarelyGorenstein} shows the only possible nonzero generators can occur when $|Y_2 \cup \{v_1\}|=9,10$. Furthermore in these cases,
\begin{equation}\label{eq:H15 of M3,Y2 times M1,Y2^c}
\begin{aligned}
H^{15}(\Mb_{3,Y_2\cup \{v_1\}} \times \Mb_{1,Y_2^c\cup \{v_2\}})&= H^{15}(\Mb_{3,Y_2\cup \{v_1\}}) \otimes H^0(\Mb_{1,Y_2^c\cup \{v_2\}})\\
& \quad\oplus H^{13}(\Mb_{3,Y_2\cup \{v_1\}}) \otimes H^2(\Mb_{1,Y_2^c\cup \{v_2\}}).
\end{aligned}
\end{equation}

When \(|Y_2\cup\{v_1\}|=9\), Lemma~\ref{lem:h=s} and
Figure~\ref{fig:gentable} give
\(H^{15}(\Mb_{3,9})=0\), while
\cite[Theorem~1.7]{CLPW24polynomial} gives
\(H^{13}(\Mb_{3,9})=0\). To show vanishing of the pushforwards of classes in \eqref{eq:H15 of M3,Y2 times M1,Y2^c} when $|Y_2\cup \{v_1\}|=10$, we first note the following diagram commutes:
\[
\begin{tikzcd}
\Mb_{3,10} \times \Mb_{1,1} \arrow[r, "\xi"] \arrow[d, "\eta_1 \times \mathrm{id}"']
  & \Mb_{4,9} \arrow[d, "\eta_2"] \\
\Mb_{4,8} \times {\Mb_{1,1}} \arrow[r, "\xi"']
  & \Mb_{5,7}
\end{tikzcd}.
\]
\noindent By the preceding section, $H^k(\Mb_{4,8})=0$ for all odd $k$. Combining this with the commutative diagram above shows that the pushforward of \eqref{eq:H15 of M3,Y2 times M1,Y2^c} under the gluing map \eqref{eq:eta gluing M_4,9 to M_5,7} is zero.

We next consider the separating divisor $\Mb_{2,Y_3 \cup \{w_1\} } \times \Mb_{2,Y_3^c \cup \{w_2\} }$. By \cite[Table~3]{CanningRarelyGorenstein}, a nonzero contribution can occur only if  $|Y_3 \cup \{w_1\}|=10$ or $|Y_3^c \cup \{w_2\}|=10$. By symmetry it suffices to consider the first scenario. In this case
\begin{equation}\label{eq:H15 of M2,Y3 times M2,Y3^c}
\begin{aligned}
H^{15}(\Mb_{2,Y_3 \cup \{w_1\} } \times \Mb_{2,Y_3^c \cup \{w_2\} })&=H^{15}(\Mb_{2,Y_3 \cup \{w_1\} }) \otimes H^0(\Mb_{2,Y_3^c \cup \{w_2\} }) \\
& \quad \oplus H^{13}(\Mb_{2,Y_3 \cup \{w_1\} }) \otimes H^2(\Mb_{2,Y_3^c \cup \{w_2\} }).\\ 
\end{aligned}
\end{equation}

\noindent As in the preceding case, the following diagram commutes:
\[
\begin{tikzcd}
\Mb_{2,10} \times \Mb_{2,1} \arrow[r, "\xi"] \arrow[d, "\eta_3 \times \mathrm{id}"']
  & \Mb_{4,9} \arrow[d, "\eta_2"] \\
\Mb_{3,8} \times {\Mb_{2,1}} \arrow[r, "\xi"']
  & \Mb_{5,7}
\end{tikzcd}.
\]
Here $\eta_3: \Mb_{2,10} \to \Mb_{3,8}$ is the self gluing map. By \cite[Table~3]{CanningRarelyGorenstein}, $H^{k}(\Mb_{3,8})=0$ for all odd $k$. Thus \eqref{eq:H15 of M2,Y3 times M2,Y3^c} vanishes under further pushforward to $H^{19}(\Mb_{5,7})$ under the self gluing map \eqref{eq:eta gluing M_4,9 to M_5,7}.

The remaining boundary component is the nonseparating component
\(\Mb_{3,[9]\cup\{x,y\}}\). It remains to analyze its
degree-\(15\) cohomology. We identify $x$ and $y$ with the last two markings in $[11]$. By Lemma \ref{lem:h=s}, we obtain the following generating set from Figure \ref{fig:gentable}:

{\allowdisplaybreaks[4]
\begin{alignat}{2}
&\S^1_{B,\emptyset,\emptyset}
    &&\quad= \GraphLambdaoneB ,
\\
&\S^1_{B,i,\emptyset}
    &&\quad= \GraphLambdatwo ,
    \notag\\
&\S^1_{B,\emptyset,i}
    &&\quad= \GraphLambdathree ,
    \notag\\
&\S^0_{B,i,\emptyset}
    &&\quad= \GraphLambdafour ,
    \notag\\
&\M^1_{B',i,\emptyset}
    &&\quad= \GraphLambdafive ,
    \notag\\
&\M^1_{B',\{i,j\},\emptyset}
    &&\quad= \GraphLambdasix ,
    \notag\\
&\M^1_{B',i,j}
    &&\quad= \GraphLambdaseven ,
    \notag\\
&\N^1_{B,\emptyset,\emptyset}
    &&\quad= \GraphLambdanine ,
    \notag\\
&\N^1_{B,i,\emptyset}
    &&\quad= \GraphLambdaten ,
    \notag\\
&\L_{B,\emptyset}
    &&\quad= \GraphLambdaeleven ,
    \notag\\
&\L_{B,i}
    &&\quad= \GraphLambdatwelve ,
    \notag\\
&\D_{B,\emptyset}
    &&\quad= \GraphDLambdathirteen ,
    \notag\\
&\D_{B,i}
    &&\quad= \GraphDLambdaforteen ,
    \notag\\
&\P^q_{B,\emptyset}
    &&\quad= \GraphLambdafifteen ,
    \notag\\
&\P^q_{B,i}
    &&\quad= \GraphLambdasixteen ,
    \notag\\
&\P^i_{B,i}
    &&\quad= \GraphLambdaseventeen .
    \notag
\end{alignat}
}
\noindent For each generator, the composite pushforward
$$
H^{15}(\Mb_{3,11})
\longrightarrow H^{17}(\Mb_{4,9})
\longrightarrow H^{19}(\Mb_{5,7})
$$
is obtained by choosing two disjoint pairs of marked points and gluing the points in each pair. Considering all possible such choices, Corollary~\ref{cor:graph-averaging-vanishing} applies directly to every resulting graph. 
For clarity we write out two examples of how to apply Corollary~\ref{cor:graph-averaging-vanishing}.

\begin{example}
    Consider the generator $\M^1_{B',x,\varnothing}$ and push it forward to $H^{19}(\Mb_{5,7})$ by gluing the marked points $x,y$ and $s,t$. The resulting graph is
    \[
    \Lambda:= \GraphMoneBxpush.
    \]

    In the context of Corollary \ref{cor:graph-averaging-vanishing}, $v$ is the central vertex, $\alpha=\omega_{B'\cup \{p\} \cup \{q\} }$, $\beta= 1 \otimes 1$, and $K =\mathrm{Aut}(\Lambda)\cong (\mathbb{S}_2)^2$. Since $2\geq12-10$, Corollary
\ref{cor:graph-averaging-vanishing} implies that $\Lambda$ vanishes.
\end{example}

In this example, the subgroup \(K\) from Corollary~\ref{cor:graph-averaging-vanishing} coincides with the full automorphism group of \(\Lambda\). The next example illustrates a situation in which \(K\) must instead be chosen as a proper subgroup, since the full automorphism group does not satisfy the hypotheses of the corollary.

\begin{example}
    Consider the generator $\P^q_{B,y}$ and push it forward to $H^{19}(\Mb_{5,7})$ by gluing the marked points $x,y$ and $s,t$. The resulting graph is
    \[
\Delta= \GraphDLambdasixteenpush ~.
\]
In the context of Corollary \ref{cor:graph-averaging-vanishing}, $v$ is the leftmost vertex, $\alpha= \omega_{B \cup \{p\} }$, and $\beta= \psi_q$. Choose $K\cong \mathbb{S}_2$ to be the subgroup of $\mathrm{Aut}(\Delta)\cong (\mathbb{S}_2)^2$ generated by the transposition $(st)$. Then $K$ fixes $v$ and $\beta$ and $1 \geq 11-10$, so $\Delta$ vanishes. Note that if we had instead chosen $K =\mathrm{Aut}(\Delta)$, we would not be able to apply Corollary \ref{cor:graph-averaging-vanishing}, as the whole automorphism group does not fix the decoration $\psi_q$.
\end{example}

All other possible pushforwards vanish under similar arguments. This proves that all odd degree cohomology of $\Mb_{5,7}$ vanishes and completes the proof.
\end{proof}


\bibliography{references}

@misc{CLPW24polynomial,
      title={Moduli spaces of curves with polynomial point counts}, 
      author={Samir Canning and Hannah Larson and Sam Payne and Thomas Willwacher},
      year={2024},
      eprint={2410.19913},
      archivePrefix={arXiv},
      primaryClass={math.AG},
      url={https://arxiv.org/abs/2410.19913}, 
}

@misc{Lsurvey,
      title={Chow rings, cohomology rings, and point counts of moduli spaces of curves}, 
      author={Hannah Larson},
      year={2026},
      eprint={2606.29656},
      archivePrefix={arXiv},
      primaryClass={math.AG},
      url={https://arxiv.org/abs/2606.29656}, 
}

@article {GraberPandharipande,
    AUTHOR = {Graber, T. and Pandharipande, R.},
     TITLE = {Constructions of nontautological classes on moduli spaces of
              curves},
   JOURNAL = {Michigan Math. J.},
  FJOURNAL = {Michigan Mathematical Journal},
    VOLUME = {51},
      YEAR = {2003},
    NUMBER = {1},
     PAGES = {93--109},
      ISSN = {0026-2285,1945-2365},
   MRCLASS = {14H10 (14C15 14C25)},
  MRNUMBER = {1960923},
MRREVIEWER = {Jos\'e\ M.\ Mu\~noz Porras},
       DOI = {10.1307/mmj/1049832895},
       URL = {https://doi.org/10.1307/mmj/1049832895},
}

@incollection {FPnontautological,
    AUTHOR = {Faber, C. and Pandharipande, R.},
     TITLE = {Tautological and non-tautological cohomology of the moduli
              space of curves},
 BOOKTITLE = {Handbook of moduli. {V}ol. {I}},
    SERIES = {Adv. Lect. Math. (ALM)},
    VOLUME = {24},
     PAGES = {293--330},
 PUBLISHER = {Int. Press, Somerville, MA},
      YEAR = {2013},
      ISBN = {978-1-57146-257-2},
   MRCLASS = {14H10 (14C17 14F25)},
  MRNUMBER = {3184167},
MRREVIEWER = {Lars\ Halvard\ Halle},
}

@book {B,
    AUTHOR = {Belorousski, Pavel},
     TITLE = {Chow rings of moduli spaces of pointed elliptic curves},
      NOTE = {Thesis (Ph.D.)--The University of Chicago},
 PUBLISHER = {ProQuest LLC, Ann Arbor, MI},
      YEAR = {1998},
     PAGES = {65},
      ISBN = {978-0591-85283-7},
   MRCLASS = {99-05},
  MRNUMBER = {2716762},
       URL =
              {http://gateway.proquest.com/openurl?url_ver=Z39.88-2004&rft_val_fmt=info:ofi/fmt:kev:mtx:dissertation&res_dat=xri:pqdiss&rft_dat=xri:pqdiss:9832113},
}

@article {Keel,
    AUTHOR = {Keel, Sean},
     TITLE = {Intersection theory of moduli space of stable {$n$}-pointed
              curves of genus zero},
   JOURNAL = {Trans. Amer. Math. Soc.},
  FJOURNAL = {Transactions of the American Mathematical Society},
    VOLUME = {330},
      YEAR = {1992},
    NUMBER = {2},
     PAGES = {545--574},
      ISSN = {0002-9947,1088-6850},
   MRCLASS = {14C15 (14C17 14H10)},
  MRNUMBER = {1034665},
MRREVIEWER = {Steven\ E.\ Landsburg},
       DOI = {10.2307/2153922},
       URL = {https://doi.org/10.2307/2153922},
}

@article {PPZ,
    AUTHOR = {Pandharipande, Rahul and Pixton, Aaron and Zvonkine, Dimitri},
     TITLE = {Relations on {$\overline{ M}_{g,n}$} via {$3$}-spin
              structures},
   JOURNAL = {J. Amer. Math. Soc.},
  FJOURNAL = {Journal of the American Mathematical Society},
    VOLUME = {28},
      YEAR = {2015},
    NUMBER = {1},
     PAGES = {279--309},
      ISSN = {0894-0347,1088-6834},
   MRCLASS = {14H10 (14N35)},
  MRNUMBER = {3264769},
MRREVIEWER = {Shengmao\ Zhu},
       DOI = {10.1090/S0894-0347-2014-00808-0},
       URL = {https://doi.org/10.1090/S0894-0347-2014-00808-0},
}

@article{CLP24extensions,
    AUTHOR = {Canning, Samir and Larson, Hannah and Payne, Sam},
     TITLE = {Extensions of tautological rings and motivic structures in the
              cohomology of {$\overline {\mathcal M}_{g,n}$}},
   JOURNAL = {Forum Math. Pi},
  FJOURNAL = {Forum of Mathematics. Pi},
    VOLUME = {12},
      YEAR = {2024},
     PAGES = {Paper No. e23, 31},
      ISSN = {2050-5086},
   MRCLASS = {14H10 (14C15)},
  MRNUMBER = {4827752},
       DOI = {10.1017/fmp.2024.24},
       URL = {https://doi.org/10.1017/fmp.2024.24},
}

@article{CLP23eleventh,
    AUTHOR = {Canning, Samir and Larson, Hannah and Payne, Sam},
     TITLE = {The eleventh cohomology group of
              {${\overline{\mathcal{M}}_{g,n}}$}},
   JOURNAL = {Forum Math. Sigma},
  FJOURNAL = {Forum of Mathematics. Sigma},
    VOLUME = {11},
      YEAR = {2023},
     PAGES = {Paper No. e62, 18},
      ISSN = {2050-5094},
   MRCLASS = {14H10 (14C30)},
  MRNUMBER = {4616435},
MRREVIEWER = {Milagros\ Izquierdo},
       DOI = {10.1017/fms.2023.59},
       URL = {https://doi.org/10.1017/fms.2023.59},
}

@article{CL24chowMg,
    AUTHOR = {Canning, Samir and Larson, Hannah},
     TITLE = {The {C}how rings of the moduli spaces of curves of genus 7, 8,
              and 9},
   JOURNAL = {J. Algebraic Geom.},
  FJOURNAL = {Journal of Algebraic Geometry},
    VOLUME = {33},
      YEAR = {2024},
    NUMBER = {1},
     PAGES = {55--116},
      ISSN = {1056-3911,1534-7486},
   MRCLASS = {14H10 (14C15 14C17)},
  MRNUMBER = {4693574},
MRREVIEWER = {Yongnam\ Lee},
}

@misc{CL23chowMgbar,
      title={On the Chow and cohomology rings of moduli spaces of stable curves}, 
      author={Samir Canning and Hannah Larson},
      year={2023},
      eprint={2208.02357},
      archivePrefix={arXiv},
      primaryClass={math.AG},
      url={https://arxiv.org/abs/2208.02357}, 
}

@article{Getzler98semiclassical,
    AUTHOR = {Getzler, E.},
     TITLE = {The semi-classical approximation for modular operads},
   JOURNAL = {Comm. Math. Phys.},
  FJOURNAL = {Communications in Mathematical Physics},
    VOLUME = {194},
      YEAR = {1998},
    NUMBER = {2},
     PAGES = {481--492},
      ISSN = {0010-3616,1432-0916},
   MRCLASS = {14H10 (14H15 14H52)},
  MRNUMBER = {1627677},
MRREVIEWER = {Andreas\ Gathmann},
       DOI = {10.1007/s002200050365},
       URL = {https://doi.org/10.1007/s002200050365},
}

@article{AC98,
    AUTHOR = {Arbarello, Enrico and Cornalba, Maurizio},
     TITLE = {Calculating cohomology groups of moduli spaces of curves via
              algebraic geometry},
   JOURNAL = {Inst. Hautes \'Etudes Sci. Publ. Math.},
  FJOURNAL = {Institut des Hautes \'Etudes Scientifiques. Publications
              Math\'ematiques},
    NUMBER = {88},
      YEAR = {1998},
     PAGES = {97--127},
      ISSN = {0073-8301,1618-1913},
   MRCLASS = {14H10 (14C30)},
  MRNUMBER = {1733327},
MRREVIEWER = {Elham\ Izadi},
       URL = {http://www.numdam.org/item?id=PMIHES_1998__88__97_0},
}

@unpublished{BergstromData,
    author = {Bergstr{\"o}m, Jonas},
    title = {Cohomology of moduli spaces of curves},
    note = {\url{https://github.com/jonasbergstroem/Cohomology-of-moduli-spaces-of-curves}}
}

@article {BFP24,
    AUTHOR = {Bergstr\"om, Jonas and Faber, Carel and Payne, Sam},
     TITLE = {Polynomial point counts and odd cohomology vanishing on moduli
              spaces of stable curves},
   JOURNAL = {Ann. of Math. (2)},
  FJOURNAL = {Annals of Mathematics. Second Series},
    VOLUME = {199},
      YEAR = {2024},
    NUMBER = {3},
     PAGES = {1323--1365},
      ISSN = {0003-486X,1939-8980},
   MRCLASS = {14H10 (11G20 14C30 14F20 14G15)},
  MRNUMBER = {4740541},
MRREVIEWER = {R.\ Parthasarathi},
       DOI = {10.4007/annals.2024.199.3.7},
       URL = {https://doi.org/10.4007/annals.2024.199.3.7},
}

@article {CanningRarelyGorenstein,
    AUTHOR = {Canning, Samir},
     TITLE = {The tautological ring of
              {$\overline{\mathcal{M}}_{g,n}$} is rarely
              {G}orenstein},
   JOURNAL = {Geom. Topol.},
  FJOURNAL = {Geometry \& Topology},
    VOLUME = {29},
      YEAR = {2025},
    NUMBER = {7},
     PAGES = {3905--3919},
      ISSN = {1465-3060,1364-0380},
       DOI = {10.2140/gt.2025.29.3905},
   MRCLASS = {14C15 (14C17)},
  MRNUMBER = {4974994},
}

@misc{LiuM58M59,
    AUTHOR = {Liu, Yuhan},
     TITLE = {On the {C}how Rings of the Moduli Spaces
              {$\mathcal{M}_{5,8}$} and {$\mathcal{M}_{5,9}$}},
      YEAR = {2025},
    EPRINT = {2509.02950},
ARCHIVEPREFIX = {arXiv},
PRIMARYCLASS = {math.AG}
}
\bibliographystyle{alphaurl}

\end{document}